\documentclass[12pt]{amsart} 
\usepackage[left=3cm,top=2.5cm,bottom=2.5cm,right=3cm]{geometry}
\usepackage{times}
\usepackage{amssymb, amsmath, amsthm}
\usepackage{graphicx,xspace}
\usepackage{epsfig}
\usepackage{enumitem}
\usepackage[usenames,dvipsnames]{xcolor}
\usepackage{tikz}
\usepackage{gensymb}
\usepackage[T1]{fontenc}
\usepackage[utf8]{inputenc}
\usepackage{bm, changes}
\usepackage[config, labelfont={normalsize}]{caption,subfig}
\usepackage{mathrsfs}
\definecolor{green}{RGB}{0,127,0}
\definecolor{red}{RGB}{191,0,0}
\usepackage[colorlinks,cite color=red,link color=green,pagebackref=true]{hyperref}
\usepackage{todonotes}

\newcommand{\Gremoved}[1]{\color{cyan}{ --REMOVED--}\color{black}}

\usetikzlibrary{matrix,arrows,calc}
\usetikzlibrary{decorations.markings}
\usetikzlibrary{patterns}
\usetikzlibrary{snakes}
\usetikzlibrary{shapes}

\usepackage[capitalize]{cleveref}

\theoremstyle{plain}

\newtheorem{lemma}{Lemma}[section]
\newtheorem{theorem}[lemma]{Theorem}
\newtheorem{corollary}[lemma]{Corollary}
\newtheorem{proposition}[lemma]{Proposition}

\newtheorem{defprop}[lemma]{Definition-Proposition}
\newtheorem{definition}[lemma]{Definition}
\newtheorem{definition-lemma}[lemma]{Definition-Lemma}

\theoremstyle{remark}
\newtheorem{remark}{Remark}
\newtheorem{example}{Example}

\newcommand{\xdownarrow}[1]{{\left\downarrow\vbox to #1{}\right.\kern-\nulldelimiterspace}}
\newcommand{\splus}{\!+\!}
\newcommand{\sminus}{\!-\!}
\newcommand{\ssum}{\!\!\sum}

\newcommand{\bI}[1]{\mathbf{1}_{#1}}

\newcommand{\R}{\mathcal{R}}
\newcommand{\A}{\mathcal{A}}

\newcommand{\N}{\mathbb{N}}

\newcommand{\Sym}[1]{\mathfrak{S}_{#1}}

\newcommand{\QQ}{\mathbb{Q}}

\newcommand{\bM}{\mathbf{M}}
\newcommand{\bN}{\mathbf{N}}

\newcommand{\xx}{\bm{x}}
\newcommand{\yy}{\bm{y}}
\newcommand{\pp}{\mathbf{p}}
\newcommand{\qq}{\mathbf{q}}
\newcommand{\rr}{\mathbf{r}}
\newcommand{\JJ}{\bm{J}}

\newcommand{\Y}{\mathbb{Y}}

\newcommand{\YY}{\tilde{Y}}
\newcommand{\ZZ}{\tilde{Z}}

\newcommand{\Sphere}{\mathbb{S}^2}

\newcommand{\PPP}{\mathcal{P}}

\def\MultiJack{\tau^{(k)}_{b}}

\def\la{\lambda}
\def\La{\Lambda}

\def\a{\alpha}

\DeclareMathOperator{\RHS}{R.H.S.}

\DeclareMathOperator{\End}{End}

\newcommand{\zz}{\mathbf{z}}

\DeclareMathOperator{\Span}{Span}

\DeclareMathOperator{\Symm}{Sym}

\DeclareMathOperator{\hook}{hook}

\def\uu{\bm{u}}

\title[Non-orientable branched coverings, $b$-Hurwitz numbers,  and positivity]{Non-orientable branched coverings, $b$-Hurwitz numbers,  and positivity for multiparametric Jack expansions}

\author{Guillaume Chapuy}
\address{CNRS, IRIF UMR 8243, Université Paris Cité.}%

\author{Maciej Dołęga}
\address{
Institute of Mathematics, 
Polish Academy of Sciences, 
ul. Śniadeckich 8, 
00-956 Warszawa, Poland.
}

\thanks{This project has received funding from the European Research
  Council (ERC) under the European Union’s Horizon 2020 research and
  innovation programme (grant agreement No. ERC-2016-STG 716083
  “CombiTop”). MD is supported from {\it Narodowe Centrum Nauki}, grant UMO-2017/26/D/ST1/00186. Emails: {\tt guillaume.chapuy@irif.fr,  mdolega@impan.pl}.  }

\begin{document}

\begin{abstract}
	We introduce a one-parameter deformation of the 2-Toda
        tau-function \sloppy of (weighted) Hurwitz numbers, obtained by deforming Schur functions into Jack symmetric functions. We show that its coefficients are polynomials in the deformation parameter $b$ with nonnegative integer coefficients.
	These coefficients count generalized branched coverings of the sphere by an arbitrary surface, orientable or not, with an appropriate $b$-weighting that ``measures'' in some sense their non-orientability.

	Notable special cases include
	non-orientable {\it dessins d'enfants} for which we prove the
        most general result so far towards the Matching-Jack conjecture
        and the ``$b$-conjecture'' of Goulden and Jackson from 1996, expansions of the $\beta$-ensemble matrix model, deformations of the HCIZ integral, and $b$-Hurwitz numbers that we
        introduce here and that are $b$-deformations of classical
        (single or double) Hurwitz numbers obtained for $b=0$.

	A key role in our proof is played by a combinatorial model of non-orientable constellations equipped with a suitable $b$-weighting, whose partition function satisfies an infinite set of PDEs. These PDEs have two definitions, one given by Lax equations, the other one following an explicit combinatorial decomposition. 
      \end{abstract}

\maketitle

\section{Introduction}

{\bf Hurwitz numbers and tau-functions.}
Hurwitz numbers, in their most general sense, count the number of
combinatorially inequivalent branched coverings of the sphere by an
orientable surface with a given number of ramification points and
given ramification profiles. Hurwitz numbers and their variants (dessins d'enfants, weighted, monotone, orbifold Hurwitz numbers)  have numerous connections to mathematical physics, combinatorics, and the moduli spaces of curves~\cite{Kontsevich1992,GouldenJackson1997a,EkedahlLandoShapiroVainshtein2001,GraberVakil2003,GouldenJacksonVakil2005,OkounkovPandharipande2006,Mirzakhani2007, Guay-PaquetHarnad2017}.

	Hurwitz himself~\cite{Hurwitz1891} showed that Hurwitz numbers can be expressed in terms of characters of the symmetric group. Equivalently,
	generating functions of Hurwitz numbers can be expressed
        explicitly in terms of Schur functions, which gives them a
        rich structure. A fundamental fact in the field, whose origins go back to 
        Pandariphande~\cite{Pandharipande2000},
        Okounkov~\cite{Okounkov2000a}, Orlov and Scherbin \cite{OrlovScherbin2000},
 and now understood in a wide generality (see e.g.~\cite{GouldenJackson2008,Guay-PaquetHarnad2017})  is that Hurwitz numbers can be used to define a formal power series which is a tau-function of the KP, or more generally 2-Toda hierarchy~\cite{MiwaJimboDate2000}. Explicitly, in the case of $k+2$ ramification points, this tau-function has the form
\begin{align}\label{eq:SchurIntro}
	\tau^{(k)}(t;\pp,\qq,u_1,\dots,u_k) := \sum_{n \geq
	0}t^n \sum_{\la \vdash n} \left(\frac{f_\lambda}{n!}\right)^2 \tilde{s}_\la(\pp)
	\tilde{s}_\la(\qq)\tilde{s}_\la(\underline{u_1})\tilde{s}_\la(\underline{u_2})\dots \tilde{s}_\la(\underline{u_k}),
\end{align}
    where $\tilde{s}_\la=\frac{n!}{f_\la}\cdot s_\la$ %
    is the normalized Schur function indexed by the
    integer partition $\lambda$ of $n$, expressed as a polynomial in the
    power-sum variables $\mathbf{p}=(p_i)_{i\geq 1}$ or
    $\mathbf{q}=(q_i)_{i\geq 1}$,  where $\underline{u}=(u,u,\dots),$
    and where $f_\lambda$ is the dimension of the irreducible representation of the symmetric group indexed by $\lambda$.
From this function (or more precisely its logarithm) one can extract all the forms $\omega_{g,n}$ associated to the contribution of coverings from surfaces of genus $g$ with $n$ boundaries, which obey the Chekhov-Eynard-Orantin topological recursion~\cite{ChekhovEynardOrantin2006, EynardOrantin2007, AlexandrovChapuyEynardHarnad2020}.
Weighted Hurwitz numbers~\cite{Guay-PaquetHarnad2017} correspond in some sense to the case $k=\infty$, which contains the Okounkov-Pandariphande Hurwitz numbers as a special case (see~\cref{sec:infinite}).
The case $k=1$ (three ramification points) corresponds to \emph{dessins d'enfants} or Belyi curves (\emph{bipartite maps} in the language of combinatorialists).

\smallskip
{\bf Jack polynomials and $b$-deformations.}
In this paper we consider the one-parameter deformation, or \emph{$b$-deformation}, of the function $\tau^{(k)}$ defined by 
\begin{align}\label{eq:JackIntro}
  \MultiJack(t;\pp,\qq,u_1,\dots,u_k) :&=
	\sum_{n \geq 0}t^n\sum_{\la
	\vdash n}\frac{1}{j_\la^{(\alpha)}} J_\la^{(\alpha)}(\pp)J_\la^{(\alpha)}(\qq) J_\la^{(\alpha)}(\underline{u_1}) \dots J_\la^{(\alpha)}(\underline{u_k}) ,
   \end{align}
where $J^{(\alpha)}_\la$ is the Jack symmetric function of parameter
$\alpha = 1+b$, for formal or complex $b$, and where
$j_\la^{(\alpha)}$ is a natural $b$-deformation of
$n!^2/{f_\lambda^2}$, see~\cref{sec:Jack}.

Jack functions are obtained as a one-parameter limit of Macdonald polynomials that interpolates between Schur and zonal polynomials, respectively for $b=0,1$~\cite{Jack1970/1971,Macdonald1995}.
In particular the function~$\tau^{(k)}_b$ is equal to~$\tau^{(k)}$ for $b=0$.
Many classical problems in algebraic combinatorics (dealing with symmetric functions, maps, coverings, tableaux, partitions) are connected to Schur or zonal polynomials. Understanding how to use Jack symmetric functions to build continuous deformations between them has become an important research goal in the last decades, see~\cite{Stanley1989, HanlonStanleyStembridge1992,GouldenJackson1996, DolegaFeray2016,BorodinGorinGuionnet2017,GuionnetHuang2019}. It often requires to develop new methods that shed new light even on the most classical results.
In our context, the deformation~\eqref{eq:JackIntro} was introduced by
Goulden and Jackson~\cite{GouldenJackson1996} in the case $k=1$ of
dessins d'enfants and is strongly related to the Matching-Jack
conjecture and the $b$-conjecture\footnote{this conjecture was
  originally called the \emph{Hypermap-Jack conjecture}
  in~\cite{GouldenJackson1996}, but later La Croix referred to it as the
  $b$-conjecture in~\cite{LaCroix2009} --- this name turned to be the
  one used in the literature afterwards} of these authors, see the discussion below.

\smallskip
{\bf Non-orientable branched coverings.}
Our main result gives a geometric (and combinatorial) meaning to the
coefficients of $\MultiJack$ in terms of generalized branched
coverings of the sphere. Let $\mathcal{S}$
  be a compact connected surface, orientable or not, and let $\Sphere$
  denote the two-dimensional sphere.
Let $\widetilde{\mathcal{S}}$ be the orientation-double-cover of $\mathcal{S}$.
A \emph{generalized branched
  covering of $\Sphere$ by $\mathcal{S}$} is a continuous function 
  $f: \mathcal{S} \rightarrow\Sphere_+$ from $\mathcal{S}$ to the closed upper hemisphere $\Sphere_+$, 
which can be lifted to a branched covering 
$\tilde{f}:\widetilde{\mathcal{S}} \rightarrow \Sphere$ in a certain sense.
A precise definition, together with the definition of degree, ramification points and
ramification profiles, is given in~\cref{sec:constellations}.

 Generalized branched coverings with $k+2$ ramification points are in bijection (\cref{sec:constellations}) with some combinatorial embedded graphs on the surface $\mathcal{S}$ that we call $k$-constellations. These objects come with a natural notion of \emph{rooting} which consists in marking and orienting an angular sector (\cref{sec:constellations}). Our main result can be summarized as follows (see in particular \cref{thm:mainInSection5} page~\pageref{thm:mainInSection5} and \cref{rem:translation} page~\pageref{rem:translation}). 
In this paper, if $\lambda=(\lambda_1\geq \lambda_2\geq \dots \geq \lambda_\ell)$ is an integer partition and $(p_i)_{i\geq 1}$ is a sequence of variables, we write $p_\lambda= p_{\lambda_1}p_{\lambda_2}\dots p_{\lambda_\ell}$.

\begin{theorem}[Main result -- abbreviated]\label{thm:abbreviated}
For every $k\geq 1$, we have 
	\begin{align}\label{eq:logIntro}
		(1+b)\frac{t\partial }{\partial t} \ln \MultiJack(t;\pp,\qq,u_1,&\dots,u_k) =
		\sum_{f: \mathcal{S} \rightarrow \Sphere_+}
		 		 \kappa(f) t^{|f|} b^{\nu_\rho(f)}, 
	\end{align}
	where the sum is taken over all rooted generalized branched coverings $f$ of the sphere $\Sphere$ by a connected compact surface, orientable or not, with $k+2$ ramification points.  Here $|f|$ is the degree of the covering, and
	$$	\kappa(f)=p_{\lambda^{-1}(f)}q_{\lambda^{0}(f)} u_1^{v_1(f)} \dots u_k^{v_k(f)} 
$$
where
	the integer partitions $\lambda^{-1}(f)$ and $\lambda^{0}(f)$ are respectively the ramification profile of the first two points, and $v_1(f), \dots v_k(f)$ are the multiplicities of the $k$ other points. 
	Moreover $\nu_\rho(f)$ is a nonnegative integer attached to $f$ which is zero if and only if the base surface $\mathcal{S}$ is orientable.

	In particular, the coefficients of the LHS of~\eqref{eq:logIntro} are polynomials in $b$, and they have nonnegative integer coefficients. 
\end{theorem}
For each covering $f: \mathcal{S} \rightarrow \Sphere_+$  contributing to~\eqref{eq:logIntro}, the homeomorphism type of the covering surface $\mathcal{S}$ is fully determined by the quantity $\kappa(f) t^{|f|} b^{\nu_\rho(f)}$. Indeed, orientability is controlled by the parameter $\nu_\rho(f)$, and the Euler characteristic is deduced from the Riemann-Hurwitz formula, see~\eqref{eq:EulerCharacteristic}. In particular, \eqref{eq:logIntro} implicitly contains a full topological expansion -- see also Remark~\ref{rem:genus}.

When $\mathcal{S}$ is orientable, generalized
branched coverings are in bijection with (usual) branched
coverings.
 Therefore for $b=0$, our result recovers the classical
interpretation of the tau-function~\eqref{eq:SchurIntro} in terms of
branched coverings (see
e.g.~\cite{GouldenJackson2008,AlexandrovChapuyEynardHarnad2020}). For
$b=1$, our theorem says that $\tau^{(k)}_b$ counts generalized
branched coverings of the sphere by arbitrary surfaces, without any
$b$-weighting. This fact could probably be proved by (now standard)
ideas close to the one used by Goulden,
Jackson~\cite{GouldenJackson1996} and Hanlon, Stanley,
Stembridge~\cite{HanlonStanleyStembridge1992} which cover the case
$k=1$  using the connection with representation theory of the Gelfand pair $(\mathfrak{S}_{2n},\mathbb{H}_n)$. However, for $b\not \in\{0,1\}$ our result is inaccessible by these methods, due to the lack of a well-adapted representation theoretic connection to Jack polynomials.

{\bf PDEs and Lax structure.} Our method of proof goes by showing that both sides of Equation~\eqref{eq:logIntro} satisfy the same PDEs. The differential operators defining these PDEs take two different forms: for the ``Jack polynomial'' side, they are defined by two companion {\it Lax equations}, while for coverings (or constellations), they follow explicitly from a combinatorial decomposition. Proving that the ``Lax'' and ``combinatorial'' forms are in fact equal is one of the hardest tasks of the paper. The presence of this Lax structure, which holds for general $b$, indicates that traces of integrability remain present beyond the two classical points $b\in\{0,1\}$.

{\bf $b$-Hurwitz numbers.} As a consequence of our work we introduce new $b$-deformations of
weighted and classical Hurwitz numbers and we investigate their properties, including the Cut-And-Join equation and piecewise polynomiality.

{\bf Link with the Matching-Jack conjecture and the $b$-conjecture.}
The deformation~\eqref{eq:JackIntro} was introduced by Goulden and Jackson~\cite{GouldenJackson1996} in the case $k=1$ of dessins d'enfants (in fact \cite{GouldenJackson1996} considers a more general function where the sequence $\underline{u_1}$ is replaced by a third arbitrary sequence of parameters). Using the connection between zonal polynomials and representation theory of the Gelfand pair $(\mathfrak{S}_{2n},\mathbb{H}_n)$, they proved that for $b=1$ this function enumerates analogues of dessins on general surfaces (orientable or not).
In the same paper they formulate the ``$b$--conjecture'' and the 
related ``Matching-Jack conjecture'', among the most remarkable open problems in algebraic combinatorics. They assert that the coefficients have an interpretation for arbitrary $b$:
they count dessins on general surfaces, with a weight which is a
polynomial in $b$ with nonnegative coefficients, as in our main
theorem. 

The representation theoretic tools used in the case $b\in\{0,1\}$ do not
apply for general $b$, and these conjectures are still wide open despite many
partial results~\cite{GouldenJackson1996,BrownJackson2007,LaCroix2009,KanunnikovVassilieva2016,DolegaFeray2017,Dolega2017a,KanunnikovPromyslovVassilieva2018}. Our
results are not strictly comparable with the Matching-Jack conjecture
and the $b$-conjecture: the case $k=1$ of our result is a special case
of both, and the case of general $k$ is incomparable with
them. However, the case $k=1$ of our results
is by
far the most general progress towards them
and covers, largely, all the
previously proven cases. Moreover, the $b$-conjecture and
our main theorem both involve \emph{three}
infinite families of parameters, 
which leads us to the following question -
is it possible to deduce the $b$-conjecture from our main theorem?
We do not know the answer to this question, but in analogy with the cases $b=0,1$, we believe that it may be positive.
Indeed, one can construct an isomorphism between the polynomial algebras generated by the three infinite sets of parameters used in this paper, and the ones appearing in the $b$-conjecture. In the cases $b=0,1$, the coefficients of the generating function $\MultiJack$ have a natural multiplicativity property that enables one to transfer the positivity and integrality of coefficients via this isomorphism.
This suggests a three-step program (1. isomorphism; 2. multiplicativity; 3. positivity and integrality of coefficients of $\MultiJack$) to attack the $b$-conjecture. The main result of this paper realizes the third step in full generality, but the second step remains to be done, and we plan to continue our research in this direction in the
future.
We observe that in the special cases
$b=0,1$ where it can be fully applied, this program leads to a new
proof of the ``$b$-conjecture'' (that is to say, of the Schur/Zonal expansion for the generating function of bipartite orientable/non-orientable maps) relying only on techniques of the present paper and not
on representation theory.
See \cref{subsec:betaEnsembles} for
more on the Matching-Jack conjecture and the $b$-conjecture.

{\bf Possible developments.}
Our paper sets the foundations of the study of $b$-Hurwitz
numbers. Many natural questions arise, which are not the subject of
this paper. 
First, the cases $b=0,1$ are related respectively to the integrable hierarchies KP and BKP  -- at least in certain cases, see e.g. \cite{AdlervanMoerbeke2001} in the context of $\beta$-ensembles.
The key role played by Lax structures in this paper indicates that
some of the properties related to integrability may still be present
for general $b$. Also, potential links of our work with
\cite{NatanzonOrlov2017}, where the Hurwitz numbers for (non-generalized)
branched coverings of the
projective plane were studied in the context of the BKP hierarchy, could be
investigated. Further, the
tau-function for $b=0$ famously satisfies, at least in some special
cases, the so-called \emph{Virasoro constraints} (e.g.~\cite{KazarianZograf2015}, see also
\cref{rem:VirasoroPatched}). We plan to address the $b$-deformation of
these in detail in a forthcoming work (for the case of $\beta$-ensembles, see again \cite{AdlervanMoerbeke2001}). In another direction, although our
results are not strictly comparable to the Matching-Jack conjecture
and the $b$-conjecture, they are by far the best partial progress
towards them. It is conceivable that in the future, results of this
paper are used in new attacks to these problems. Finally, Hurwitz
numbers are classically linked to the moduli space of complex curves
in several ways (most famously via the ELSV
formula~\cite{EkedahlLandoShapiroVainshtein2001}) and the $b$-deformed
{\it dessins d'enfants} appear for $b=1$ in work of Goulden, Harer,
and Jackson on the moduli space of real
curves~\cite{GouldenHarerJackson2001}. These authors were the first to
ask for the possible significance of the ``$b$ parameter'' in this
geometric picture. Our paper is an advance in that
direction. Understanding the integer parameter $\nu_\rho$ that we
introduce in this paper, in a purely geometric way, seems to be a
natural question to consider to go further. It should be related in
some sense to a ``stratification'' of the moduli space, yet to be
understood.

{\bf Extended abstract.}
The results of this paper will be presented at the conference FPSAC'21 and an extended abstract of 12 pages, without proofs, announcing the combinatorial part of our results will appear in the conference proceedings.

\medskip
{\bf Overview of the paper and intermediate results.}
The paper is almost entirely dedicated to the proof of our main result -- only \cref{sec:infinite} is independent and examines the projective limit when $k\rightarrow \infty$.
However several intermediate concepts and results appearing along the way are interesting in themselves, even in the case of $b\in \{0,1\}$.
Here is a short overview of our main contributions and a roadmap to our paper.

In \cref{sec:constellations} we introduce generalized branched
coverings and their combinatorial counterparts, $k$-constellations. In
\cref{sec:MonTutte} we introduce the notion of Measure of
Non-Orien\-tability (MON) and the combinatorial decomposition, that
together give rise to the $b$-weights and the parameter $\nu_\rho$.
We also state that the generating function of generalized branched coverings (or constellations) satisfies an explicit equation reflecting the combinatorial decomposition. This is \cref{thm:TutteConnected} page~\pageref{thm:TutteConnected}.

In \cref{subsec:proofTutte} we prove the decomposition
equation by analysing carefully the combinatorial decomposition. As far as we
know, this equation is interesting even for $b=0$ as it did not appear
earlier in full generality. \cref{subsec:Commutation} contains
the key idea of the paper: the combinatorial operators appearing in
the decomposition equation can alternatively be defined by recursive commutation relations (\cref{thm:commut1,thm:commut2} pages~\pageref{thm:commut1}--\pageref{thm:commut2}) or equivalently by two Lax equations (\cref{prop:Lax}).
Proving this claim is the hardest part of the paper. \cref{subsec:heuristic} sketches a combinatorial proof for $b\in\{0,1\}$, which serves as an inspiration for the general proof, given in \cref{subsec:proofCommut1,subsec:proofCommut2}.
 
\cref{sec:Jack} deals with Jack polynomials and shows that
the function $\tau^{(k)}_b$ is annihilated by the operators
constructed in the previous section. This makes the connection with
generalized branched coverings and constellations, and proves our main theorem. Interestingly, and as far as we know, this proof is also new in the classical case $b=0$: it is the first proof of the Schur function expansion of the generating function of coverings that does not rely on representation theory. The same is true of course for $b=1$. 
	
In \cref{sec:infinite} we show how to take a projective limit
of our results in order to build a non-orientable, $b$-weighted,
analogue of the weighted Hurwitz numbers. All results of the paper are
extended to this setting, including $b$-weights, decomposition
equations, $b$-polynomiality. In
\cref{subsec:bHurwitz} we study the case of (simple or double)
$b$-weighted Hurwitz numbers, which correspond to the case where all
ramification points except the first two are simple. We prove a deformed version of the Cut-And-Join equation, and piecewise polynomiality. In \cref{subsec:betaEnsembles} we discuss dessins d'enfants and $\beta$-ensembles, and we make a detailed account of the $b$- and Matching-Jack conjectures of Goulden and Jackson, and how they relate to our results. In \cref{subsec:HCIZ} we discuss monotone Hurwitz numbers and the HCIZ integral.

Finally, the appendix contains the proof of two lemmas relying on computations that present no difficulty, but are included for completeness.

\section{Coverings, maps,  and constellations}
\label{sec:constellations}

In this section we quickly review branched coverings before introducing their non-orientable generalization. We then introduce $k$-constellations as a combinatorial model for them.

\subsection{Branched coverings}

Let $\mathcal{S}$ be a \emph{surface}, that is to say a compact, two dimensional, real
manifold. By the classification theorem a connected surface $\mathcal{S}$ is
uniquely determined by its Euler characteristic 
$\chi_{\mathcal{S}} \leq 2$ 
 (or, equivalently, its \emph{genus} $g_{\mathcal{S}} \in \frac{1}{2}\N$  given by $\chi_\mathcal{S} = 2-2 g_{\mathcal{S}}$) together with 
the information whether $\mathcal{S}$ is orientable or not. 
For two surfaces $\mathcal{S}_1, \mathcal{S}_2$ we call a continuous map $$f: \mathcal{S}_1 \to \mathcal{S}_2$$ a
\emph{branched covering} (also known as \emph{ramified covering} or \emph{branched
  cover}) 
  of $\mathcal{S}_2$ by $\mathcal{S}_1$ 
  if every point $s \in
\mathcal{S}_2$ has an open neighborhood $U\ni s$ such that $f^{-1}(U)$ is a union of disjoint open sets $V_1\dots,V_\ell$, for some $\ell\geq 1$, such that 
 on each $V_i$ the map $f$ is topologically equivalent to the complex map
 $z \to z^{p_i}$ for some positive integer $p_i$ (with $s$ corresponding to $0\in \mathbb{C}$).
For each $s\in \mathcal{S}_2$, we can reorder the multiset
$\{p_1,\dots,p_\ell\}$ to form a \emph{partition} $\la = (\la_1,\dots,\la_\ell)$ of length $\ell$, that is a sequence of integers such that
$\la_1\geq \cdots \geq \la_\ell >0$. 
This partition, denoted by $\lambda(s)$, is called the \emph{ramification profile over $s$}, or \emph{profile over $s$} for short.
When the surfaces are connected, the size
$n=\lambda_1+\dots+\lambda_\ell$ of the partition $\lambda(s)$ does
not depend on the point $s$, and it is called the \emph{degree} of the covering. In particular, it is equal to the number of preimages of a generic point of $\mathcal{S}_2$.
The integer
$\ell$ which is \emph{the length} of the partition $\la(s)$
is called the \emph{multiplicity} of $s$. 

There are finitely
 many points $s_1\dots,s_k \in \mathcal{S}_2$ of multiplicity
 smaller than the degree -- they are called \emph{critical values}, or
 \emph{ramification points}. 
 The multiset $\{\la(s_1),\dots,\la(s_k)\}$ of
 profiles over all ramification points is called \emph{the full profile}
 of the covering $f$.
 Sometimes the ramification points will be numbered from $1$ to $k$, in which case the full profile will be defined as the ordered $k$-tuple $(\la(s_1),\dots,\la(s_k))$. 
 We will (classically) allow the partitions $\la(s_i)$ to be equal to $[1^n]$, {\it i.e.} we allow  ``trivial ramification points''.
 We say that two branched coverings $f_1: \mathcal{S}_1 \to
 \mathcal{S}, f_2: \mathcal{S}_2 \to
 \mathcal{S}$ are \emph{equivalent} if there exists a homeomorphism $\phi: \mathcal{S}_1 \to
 \mathcal{S}_2$ such that $f_1 = f_2\circ\phi$. When ramification points of
 $f_1$ and $f_2$ are numbered we additionally require $\phi$ to
 preserve this numbering.

\subsection{Generalized branched coverings.}
When $f:\mathcal{S}_1 \to \mathcal{S}_2$ is a branched covering and
$\mathcal{S}_2$ is orientable, then necessarily $\mathcal{S}_1$ is
orientable. This is the case in particular when $\mathcal{S}_2$ is the sphere $\Sphere$. We will now
generalize the definition of branched coverings to allow arbitrary
surfaces as the covering space of the sphere.

Let $\mathcal{S}$ be a surface. We let
$\pi_{\mathcal{S}}:\widetilde{\mathcal{S}} \rightarrow \mathcal{S}$ be
the orientation double cover of $\mathcal{S}$, and we let $\sigma:
\widetilde{\mathcal{S}} \to \widetilde{\mathcal{S}}$ be the
corresponding involution of $\widetilde{\mathcal{S}}$, so that
$\mathcal{S} \equiv \widetilde{\mathcal{S}}/\sigma$.
We also define $\Sphere_+$ and $\Sphere_-$ as the upper and the
  lower hemisphere, respectively, (with the equator $\partial \Sphere_+ = \partial
  \Sphere_-$ as common boundary) and the natural projection $p:\Sphere \to
  \Sphere_+$ identifying both hemispheres. 
  \begin{definition}\label{def:generalizedBC}
Let $f: \mathcal{S} \to \Sphere_+$ be a continuous map, which
(restricting its domain) is a covering of
$\Sphere_+\setminus\partial\Sphere_+$. We say that $f$ is a
\emph{generalized branched covering} of the sphere if there exists a
branched covering of the sphere
$\tilde{f}:\widetilde{\mathcal{S}}\to \Sphere$ such that
	  \begin{align}\label{eq:compatibility}
		  f \circ \pi_{\mathcal{S}} = p\circ\tilde{f}.
	  \end{align}
	We say that two generalized branched
coverings $f: \mathcal{S}\to \Sphere_+$ and $f':\mathcal{S}'\to \Sphere_+$ are
        \emph{equivalent}  if the branched
coverings $\tilde{f}$ and $\tilde{f'}$ are equivalent.
\end{definition}

When the covering space is orientable, generalized branched coverings are in natural bijection with branched coverings. Indeed, when $\mathcal{S}$ is orientable the associated
orientation double cover is simply
$\widetilde{\mathcal{S}}= \mathcal{S}\uplus\mathcal{S}$,
%
%
%
%
%
and the branched covering $\tilde{f}:\widetilde{\mathcal{S}}\to \Sphere$ subject
to $f \circ \pi_{\mathcal{S}} = p\circ\tilde{f}$ is determined by its
behaviour on a single copy of $\mathcal{S}$,
which gives the desired bijection (the uniqueness of $\widetilde{f}$ up to equivalence will be proved in the proof of \cref{prop:Correspondence2}).

We will be interested in enumerative properties of the
generalized branched coverings of the sphere $\Sphere$.
First note that for a
generalized branched covering $f:\mathcal{S}\to \Sphere_+$
and for each $s \in \Sphere$, if $s$ is a ramification point of $\tilde{f}$,
 then necessarily $s$ lies on the equator $\partial\Sphere_+$ (note the assumption that $f$ restricts to a \emph{covering} of $\Sphere_+\setminus\partial\Sphere_+$, and not to a \emph{branched covering}).
 The profile of the associated branched
covering $\tilde{f}$ over $s$ is a partition of
the form $\lambda \uplus \lambda :=
 (\lambda_1,\lambda_1,\lambda_2,\lambda_2,\dots,\lambda_\ell, \lambda_\ell) \vdash 2n$ for some partition $\la
 \vdash n$.  We will call $p(s)\in \Sphere_+$ a \emph{ramification point} of $f$. We denote the partition $\lambda$ by $\la(s)$ and we will call it the \emph{profile} over
 $s$. The \emph{full profile} of the generalized branched covering is given
 by the multiset $\{\la(s):s \text{ is a ramification point}\}$. As
 before, if ramification points are numbered, the full profile will be
 an ordered tuple, and we may allow trivial ramification points. Also, when ramification points are numbered, the equivalence between $\tilde{f}$ and $\tilde{f}'$ in Definition~\ref{def:generalizedBC} is understood between branched coverings with \emph{numbered} ramification points.
The integer $n$, which is half the degree of the branched covering $\tilde{f}$, is called the \emph{degree} of the generalized branched covering $f$. All these definitions are compatible with the standard definitions in the case where $\mathcal{S}$ is orientable, via the natural bijection of the previous paragraph.

\subsection{Maps and constellations}

\medskip

The problem of counting branched coverings of the sphere is equivalent to counting certain embedded graphs called \emph{maps}, that we now define.
An embedding of a graph (possibly with multiple edges and loops) into a surface which cuts it into simply
connected pieces (called \emph{faces}) is called \emph{a map}. We
consider maps up to homeomorphisms of surfaces. A small neighborhood of an edge around a vertex is called a
\emph{half-edge} and a small neighborhood of a vertex delimited by two
consecutive half-edges is called \emph{a corner}. It is convenient to represent a map by its \emph{ribbon graph}, which is the surface with boundary made by a small neighbourhood of the graph on the surface it embeds in (see \cref{fig:constellation}--Right).

Lando and Zvonkine introduced in their book~\cite{LandoZvonkin2004} a
particular set of vertex-coloured maps, subject to local coloring
constraints, called \emph{constellations}, that are in bijection with
branched coverings of the sphere $\Sphere$. The constellation
associated to a covering $f:\mathcal{S}\rightarrow\Sphere$  with $k+2$
numbered ramification points, is the map on $\mathcal{S}$ formed by
the preimage of a ``base graph'' drawn on the sphere going through
some of these points. The standard choice of base graph is a star
centered at a generic point, and connected in cyclic order to the
points numbered $0, 1, \dots, k$. These maps satisfy some simple local
colouring constraints that fully characterize them. Different choices
of base graph lead to different definitions which are easily seen to be equivalent, see e.g.~\cite[Figure 1.34]{LandoZvonkin2004} or~\cite{BousquetSchaeffer2000, AlexandrovChapuyEynardHarnad2020}.

To construct generalized constellations we will use a similar
principle but it will be important to choose a base graph that does
not depend on an orientation of the sphere. For this reason we will
use a path going through the ramification points rather than a star, see~\cref{fig:Coverings}. We leave to combinatorialist readers the pleasure of designing a direct bijection, in the orientable case, between the model we introduce and the one of~\cite{LandoZvonkin2004}.

\begin{figure}%
\center\subfloat[]{
\label{subfig:vertex}
\includegraphics[width=0.3\linewidth]{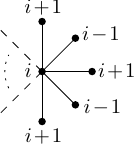}}
\qquad
\subfloat[]{
\label{subfig:constellation}
\includegraphics[width=0.5\linewidth]{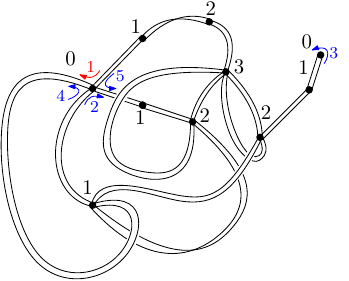}}
\caption{Left: the local constraints around a vertex of colour $i\in
  (0,k)$. Right: a labeled $3$-constellation of size $5$, in
	ribbon-graph representation; the corners of colour $0$ are labeled from $1$ to $5$, and each of them is oriented;  disregarding the labels of corners and the orientation of all blue corners (so that there remains only one distinguished and oriented corner, the red one) gives a rooted $3$-constellation (of size $5$).}
\label{fig:constellation}
\end{figure}

\begin{definition}[Constellation, see Figure~\ref{fig:constellation}]
	Let $k\geq 1$ be an integer. A \emph{$k$-constellation} is a map, equipped with a coloring of its vertices with colors in $\{0,1,2,\dots,k\}$, such that
\begin{enumerate}
\item
each vertex colored by $0$ ($k$, respectively) has only neighbours of
color $1$ ($k-1$ respectively),
\item
for any $0 < i < k$ and for any vertex $v$ colored by $i$, each corner
of $v$ separates vertices colored by $i-1$ and $i+1$.%
\end{enumerate}
\end{definition}

	The \emph{degree} of a face in a $k$-constellation is the
        number of corners of colour $0$ it contains, which is the same
        as the number of corners of colour $k$, and as half the number
	of corners of any other colour (the \emph{color} of a corner is the color
        of the vertex it is incident to). The \emph{degree} of
        a vertex of color $0$ or $k$ in a $k$-constellation is the
        number of its adjacent corners, while the degree of a vertex of
        color $i \in [1..k-1]$ in a $k$-constellation is half the number
	of its adjacent corners. The \emph{size} of a constellation $\bM$ is its number of corners of
	colour $0$ and is denoted by $|\bM|$. In particular any
        $k$-constellation $\bM$ has $k\cdot|\bM|$ edges.
	A constellation of size $n$ is \emph{labeled} if its corners of colour $0$ are labeled with the integers from $1$ to $n$, and if each such corner carries an (arbitrary) orientation.
A constellation  is \emph{rooted} if it is equipped with a
distinguished oriented corner of colour $0$, called the \emph{root}
(if the constellation is already labeled, the orientation of the root
corner is already given, but for unlabeled maps, this orientation is
part of the information given by the rooting). The \emph{root vertex
  (or face, respectively)} is the vertex (or face) incident to the
root corner. The \emph{color} of an edge is the pair $\{i,j\}$ formed by
        the colors of its two endpoints. \emph{The full profile} of a $k$-constellation is the $k+2$-tuple 
$(\la^{-1},\la^0,\la^1,\dots,\la^k)$, where $\la^{-1}$ is the
partition encoding face degrees and $\la^i$ is the
partition encoding degrees of vertices of colour~$i$.

\begin{figure}
\center
\includegraphics[width=0.85\linewidth]{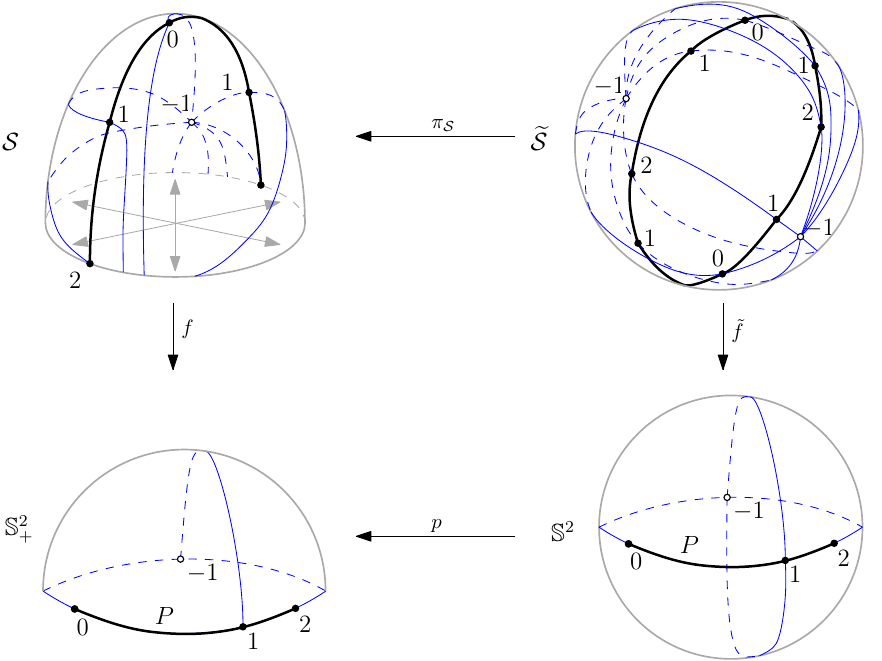}
\caption{The correspondence between generalized branched coverings and
	constellations in the case of $k+2=4$ ramification
        points. The base graph $P$, which is a path, and the corresponding
        constellations are drawn in fat black. Blue lines give the
        triangulations described in the proof of \cref{prop:Correspondence2}. On this example, $\widetilde{\mathcal{S}}$ is the
        sphere and $\mathcal{S}$ is the projective plane and the full
        profile is given by $((2),(2),(1,1),(2))$.
	}
\label{fig:Coverings}
\end{figure}

We can now state the correspondence between coverings and
  constellations. The proof uses a classical
  result in the orientable case.

\begin{proposition}[see
  \cref{fig:Coverings}-Left]\label{prop:Correspondence2}
  	Let $k\geq 1$ be an integer and let $f:
        \mathcal{S}\rightarrow \Sphere_+$ be a generalized
        branched covering of the sphere with $k+2$
        ramification points as in \cref{def:generalizedBC} monotonically numbered from $-1$ to $k$  along the equator $\partial\Sphere_+$. Let $P$ be a path
        on the equator  going through the points $0,1,\dots,k$ in this
        order. Then the  preimage $f^{-1}(P)\subset \mathcal{S}$ is a $k$-constellation on $\mathcal{S}$.

	This construction gives a one-to-one correspondence between
        equivalence classes of generalized branched
coverings of the sphere with  $k+2$ monotonically numbered ramification points and full profile
	$(\la^{-1},\la^0,\la^1,\dots,\la^k)$, and $k$-constellations with
the same full profile.
\end{proposition}

\begin{proof}
  The fact that the embedded graph $f^{-1}(P)$ on $\mathcal{S}$
  satisfies the local constraints of constellations (with vertex
  colours given by the pull-back of the numbering of ramification
  points) is clear. The fact that it is a well-defined map comes from
  the fact that each region delimited by this graph on $\mathcal{S}$
  retracts to the neighbourhood of a preimage of the ramification point~$-1$. 
	Each such neighbourhood is homeomorphic to a disk by
  definition of a generalized branched covering. We need to prove that
  $f^{-1}(P)$ gives the same constellation for equivalent
  generalized coverings.
  First note that $P$ is naturally identified with
  $p^{-1}(P)$, since $p$ is the identity on $\partial\Sphere_+ =
  \Sphere_+\cap \Sphere_-$, so we will denote both paths simply by
	$P$. In particular, using~\eqref{eq:compatibility} we have 
	$\tilde{f}^{-1}(P) = (p \tilde{f})^{-1}(P) = (f \pi_S)^{-1}(P)$, 
	 which does not depend on the
    choice of $\widetilde{f}$. This means that the constellation
	$\tilde{f}^{-1}(P)$ on $\tilde{\mathcal{S}}$ is determined by $f$. It is a standard fact in the orientable case (see~\cite[Chapter 1]{LandoZvonkin2004}, adapted to our slightly different choice of base graph)
    that the constellation $\tilde{f}^{-1}(P)$ determines
    $\tilde{f}$ uniquely up to equivalence. Now, let $f$ and $g$
	be equivalent generalized coverings, and consider the (unique up to equivalence) associated branched coverings 
	$\tilde{f}$, $\tilde{g}$.
	By definition  $\tilde{f}$ and $\tilde{g}$ are equivalent, therefore $\tilde{f}^{-1}(P)$ and $\tilde{g}^{-1}(P)$ are
	the same constellations. Using~\eqref{eq:compatibility} again shows that $f^{-1}(P)=\pi_S (\tilde{f}^{-1}(P))$ (and the analogue statement for $g$), therefore the two constellations $f^{-1}(P)$ and $g^{-1}(P)$ are also the same, as we wanted to prove.

	Now let $\bM$ be a constellation of $\mathcal{S}$. Then
        $\widetilde{\bM}=\pi_{\mathcal{S}}^{-1}(\bM)$ is a map on the
	orientable surface $\widetilde{\mathcal{S}}$. Using the standard arguments of
	the orientable case~\cite[Chapter 1]{LandoZvonkin2004} (adapted again to our
	choice of base graph),  we can  construct from $\tilde{\bM}$
        a branched covering $\tilde{f}: \widetilde{\mathcal{S}}\rightarrow
        \Sphere$ as follows. Triangulate $\Sphere$ by triangles with
        vertices given by the ramification points labeled by $-1,i,i-1$ for $i \in [1..k]$. In this way, we obtain $k$ triangles on the upper hemisphere $\Sphere_+$
        and $k$ corresponding (through $p$) triangles on the lower
	hemisphere, and such that the equator $\Sphere_+\cap\Sphere_-$
        is a cycle $(-1,0,\dots,k)$, see \cref{fig:Coverings}. Triangulate each face of $\bM$ by putting a new
        vertex $-1$ inside each face and connecting it to all the
        corners of the corresponding face. Pick an orientation on
        $\widetilde{\mathcal{S}}$ to send triangles with the set of
        vertices $-1,i,i-1$ visited in this order to the corresponding
        triangle in $\Sphere_+$ and visited in the reverse order to the corresponding
        triangle in $\Sphere_-$. Note that applying
        $\pi_{\mathcal{S}}$ to the triangulation of
        $\widetilde{\mathcal{S}}$ we obtain a triangulation of
          $\mathcal{S}$, which allows us to construct $f$ by sending
          triangles of the form $-1,i,i-1$ into corresponding
          triangles in $\Sphere_+$. The compatibility relation $f
          \circ \pi_{\mathcal{S}} = p\circ\tilde{f}$ is satisfied
          since the triangulations of $\Sphere_+$ and $\Sphere_-$ are
          identified by $p$.

	The fact that the two constructions are inverse of each other, and that the full profile is preserved, is direct by construction.
      \end{proof}

We remark that the Euler characteristic $\chi(\mathcal{S})$ of the covering surface can be recovered from the full profile $(\lambda^{-1}, \dots, \lambda^k)$ via the Riemann-Hurwitz/Euler formula:
\begin{align}\label{eq:EulerCharacteristic}
	\chi(\mathcal{S}) = 2n - \sum_{i=-1}^k (n- \ell(\lambda^i)).
\end{align}
Indeed, this formula is true for branched coverings and by construction this immediately implies that it holds for generalized branched coverings as well.
We remark that~\eqref{eq:EulerCharacteristic} only involves the length of each partition $\lambda^i$. In this paper we will enumerate generalized branched coverings of the sphere without controlling the full profile, but with enough control to keep track of these lengths, hence of the Euler characteristic of the underlying surface. 

\begin{remark}\label{rem:translation}
	Now that the correspondence between generalized branched coverings and
        constellations is established, in the rest of the paper, we
        will work with $k$-constellations, which are more convenient to
        enumerative purposes.
	The theorem stated in the introduction
        (\cref{thm:abbreviated}) will be proved in the language of
        constellations (Theorem~\ref{thm:mainInSection5}). The sum
        over rooted coverings $f$ in this theorem is understood as the
        sum over rooted constellations $(\bM,c)$
        in~\eqref{eq:mainInSection5}. The integer parameter
        $\nu_\rho(f)$ in that theorem is understood as the parameter
        $\nu_\rho(\bM,c)$ that we introduce in the next section, while
        the quantities $|f|, \kappa(f), v_i(f)$ in the theorem are
        understood respectively as the quantities  $|\bM|,
        \kappa(\bM,c), v_i(\bM)$ defined in \cref{sec:constellations}
        and \cref{sec:MonTutte}.
\end{remark}

\begin{remark}
	Some authors may prefer to call $(k+1)$-constellations what we call $k$-constellations here. This is related to the fact that in our main function~\eqref{eq:JackIntro} we have \emph{two} sets of ``time'' parameters $\pp$ and $\qq$. In many applications one studies the specialization $q_i=\bI{i=1}$, which on coverings corresponds to the case where the second marked ramification point is trivial (this is the same as viewing single Hurwitz numbers as special cases of double ones). Among these two natural choices of terminology, we kept the one that was shorter and more convenient for our purposes.
\end{remark}

	We conclude this section by introducing the notion of duality. 
	\begin{definition}\label{def:duality}
	\emph{Duality} is the involution on $k$-constellations defined
        as follows. Given a $k$-constellation $\bM$, add a new vertex of colour $-1$ inside each face and
link it by a new edge to all corners of label $k$. Then remove all vertices of $\bM$ of colour $0$ and edges
incident to them. Finally, exchange colours $-1\leftrightarrow 0$ and $k+1-i\leftrightarrow i$ for each $1 \leq i \leq k$. The map $\tilde \bM$ thus obtained is called the \emph{dual of $\bM$.} 
	\end{definition}
The fact that duality is an involution is clear from the interpretation on coverings, since it can be interpreted as a reflection of $\Sphere_+$ together with a renumbering of ramification points. 
There is a natural correspondence between corners of colour $0$ in
$\bM$ and $\tilde \bM$, which enables to lift duality at the level of
labeled and/or rooted objects. Indeed, let $\bM'$ be an intermediate map constructed from
  $\bM$ by adding a new vertex of colour $-1$ inside each face and
linking it by a new edge to all corners of label $k$. Note that each
face of $\bM'$ contains precisely one corner of color $-1$ and one
corner of color $0$. In particular orienting this face is equivalent
to orienting the corresponding corner of color $-1$ and it is
also equivalent to orienting the corresponding corner of color $0$.

\section{MON's and the $b$-deformed decomposition equation}
\label{sec:MonTutte}

\subsection{MON's and weights}

\begin{figure}%
\center\includegraphics[width=\linewidth]{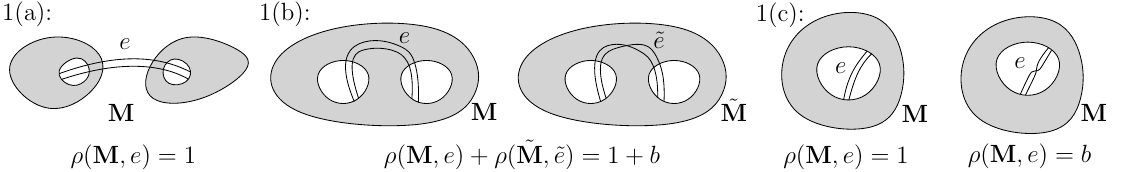}
	\caption{The main axioms of MONs.}
\label{fig:MON}
\end{figure}

Our way to assign a $b$-dependent weight to a map proceeds by repeated edge-deletions. The weight attached to each deletion depends on a number of arbitrary choices subject to suitable axioms, encompassed by the concept of \emph{measure of non-orientability}. 
\begin{definition}[MON; see \cref{fig:MON}]\label{def:MON}
	A \emph{measure of non-orientability (MON)} is a function
        $\rho(\cdot,\cdot)$ with value in $\mathbb{Q}[b]$ that
        associates to a vertex-colored map  $\bM$ and an edge $e$ in $\bM$, some value $\rho(\bM,e)$ and that satisfies the following properties.
	\begin{enumerate}
		\item
			Let $\bN:=\bM\setminus \{e\}$ and let $c_1,c_2$ be the two corners delimited by the endpoints of $e$ in~$\bN$.
			\begin{enumerate}
				\item If $c_1,c_2$ belong to two distinct connected components of $\bN$, then  $\rho(\bM,e)=1$.
				\item If $c_1,c_2$ belong to the same connected component of $\bN$ but to two different faces, then let $\tilde e$ be the other edge that could be added to $\bN$ between these corners to form a new map $\tilde\bM$. Then
					$\rho(\bM,e)+\rho(\tilde\bM,\tilde{e})=1+b$.
									\item If $c_1,c_2$ belong to the same face of $\bN$, then $\rho(\bM,e)=1$ if $e$ splits this face into two faces (``untwisted diagonal)'' and $\rho(\bM,e)=b$ otherwise (``twisted diagonal'').
		\end{enumerate}
		\item the value of  $\rho(\bM,e)$ depends only on the connected component of $\bM$ containing $e$.
	\end{enumerate}
	If $e_1,e_2,\dots,e_i$ are edges of $\bM$, we will denote
	$$
	\rho(\bM; e_1,\dots,e_i):=\rho(\bM,e_1) \rho(\bM\setminus\{e_1\},e_2)\dots \rho(\bM\setminus\{e_1,\dots,e_{i-1}\},e_i).
	$$
	This quantity in general depends on the ordering of the edges $e_1,\dots,e_i$. We will also use the notation $\rho(\bM;L)$ where $L=(e_1,\dots,e_i)$ is an ordered list of edges.
      \end{definition}

 \begin{example}\label{ex:computMON}
	 Let us compute the value of $\rho(\bM; e_1,e_2,e_3)$  for the three examples of Figure~\ref{fig:ExampleCompMON}.  In case a), both $\bM$ and $\bM\setminus\{e_1\}$ have
        one face, therefore $e_1$ is a ``twisted diagonal'' and
        $\rho(\bM;e_1) = b$ by condition 1c) of \cref{def:MON}. Removing 
        $e_2$ from $\bM\setminus\{e_1\}$ also produces a map with one
        face, so this is again the same case and
        $\rho(\bM;e_1,e_2) = b^2$. Finally removing $e_3$ from
        $\bM\setminus\{e_1,e_2\}$ disconnects this map,
        so we are in case 1a), and finally
        $\rho(\bM; e_1,e_2,e_3)=b^2$.
        The labeled map in case b) is identical to the
        one in case a), but the order of edge removals is
        different. The removal of $e_1$ from $\bM$ produces a map with two faces, therefore we are in case 1b) and $\rho(\bM,e_1)$ can be a priori any
        polynomial $X \in \QQ[b]$. However, we need to remember that
        1b) says
	there exists a labeled map
        $\tilde\bM$ and an edge $e_1$ (shown in case c) ) such that
        $\rho(\tilde\bM,e_1) = 1+b-X$. Removing the edge $e_2$ from
        $\bM\setminus\{e_1\}$ (which is the same map as
        $\tilde\bM\setminus\{e_1\}$) produces a map whose number of
        faces is smaller by one, therefore this edge is an ``untwisted diagonal'', 
	     which gives from 1c)
        $\rho(\bM;e_1,e_2)=X$ and
        $\rho(\tilde\bM;e_1,e_2)=1+b-X$. Removing the last edge
        corresponds to case 1a) and therefore $\rho(\bM;e_1,e_2,e_3)=X$ and
        $\rho(\tilde\bM;e_1,e_2,e_3)=1+b-X$.
Finally, note that the map $\tilde\bM$ from case c) is
        orientable. 
	     If we assume that $\rho$ is \emph{integral} (Definition~\ref{def:niceMON}),
	we have necessarily $X=b$, imposing
        $\rho(\tilde\bM;e_1,e_2,e_3)=1$, and 
        $\rho(\bM;e_1,e_2,e_3)=b$ for the map $\bM$ from case b).
 \end{example}

     \begin{figure}[h]
          \center\includegraphics[width=0.72\linewidth]{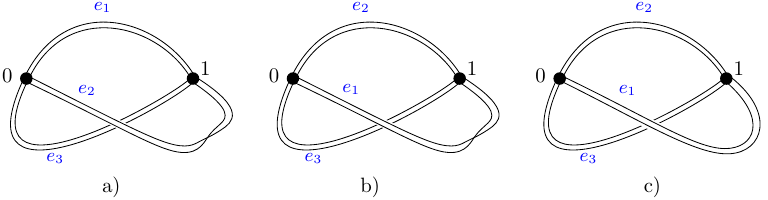}
	     \caption{Three maps $\bM$ with edges labeled $e_1,e_2,e_3$. See Example~\ref{ex:computMON} for the computation of $\rho(\bM; e_1,e_2,e_3)$ in each case.}
	     \label{fig:ExampleCompMON}
        \end{figure}

Our main results also require us to define integral and coherent MONs.
\begin{definition}[Integral MON, Coherent MON; see \cref{fig:niceMON}]\label{def:niceMON}

		A MON $\rho$ is \emph{integral} if $\rho(\bM,e)$ belongs to
	$\{1,b\}$ for any $\bM$ and $e$, and if the following is true: for every pair $(\bM,e)$ which is in case (1)(b) of Definition~\ref{def:MON} and such that $\bM$ is orientable, we have $\rho(\bM,e)=1$ and $\rho(\tilde{\bM},\tilde{e})=b$, in the notation of Definition~\ref{def:MON}.

	A MON $\rho$ is \emph{coherent} if for any colored map $\bM$, for any corner $c$ of $\bM$ of color $j$, and for any face $f$ of $\bM$ having an even number of corners of color $j+1$, the following is true. 
	Choose an arbitrary orientation of $f$, and number
        $c_1,\dots,c_{2d}$ the corners of color $j+1$ in $f$ (these
        corners inherit the orientation of $f$). Also choose an arbitrary orientation for $c$. 
	For $i\in [1..2d]$ let $e_i, \tilde{e}_i$ be the two possible edges connecting $c$ to $c_i$, where $e_i$ is the one that respect the corner orientations, and $\tilde{e}_i$ is its twist.
	Then for any $i$: 
	$$\rho(\bM\cup\{e_i\},e_i)+\rho(\bM\cup\{\tilde
        e_{i+1}\},\tilde{e}_{i+1})=(1+b)$$
        with the convention that $\tilde{e}_{2d+1} := \tilde{e}_{1}$.
\end{definition}
The idea of using MON's or their variants already appeared in previous works on the $b$-conjecture at least since Lacroix~\cite{LaCroix2009} (see also~\cite{DolegaFeraySniady2014,Dolega2017a}). Here we have added Axiom (2) which is necessary for the generating function arguments in the next section. 
Note also that previous authors only consider what we call here \emph{integral} MON's. Although considering non-integrals MON's is not needed strictly speaking for this paper, we believe that this is natural and may be useful for further developments (see \cref{rem:rhoSYM} below).

\begin{figure}%
\center\includegraphics[width=\linewidth]{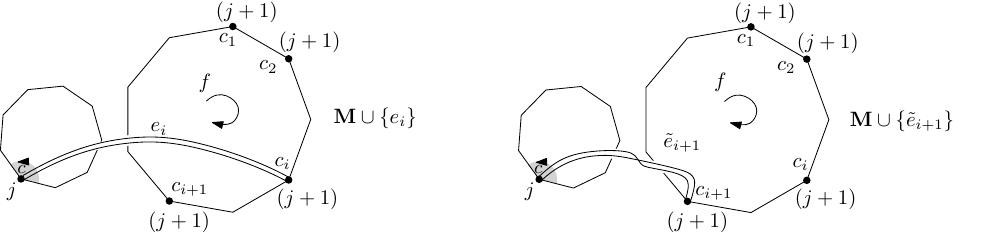}
	\caption{A coherent MON: $\rho(\bM\cup\{e_i\},e_i)+\rho(\bM\cup\{\tilde e_{i+1}\},\tilde{e}_{i+1})=(1+b).$}
\label{fig:niceMON}
\end{figure}

It is easy to see that MONs exist and to construct them. In fact, we have
\begin{lemma}\label{lemma:MONexist}
There exist MONs that are both coherent and integral.
\end{lemma}
\begin{proof}
The only choices that we have to make to construct a MON are
        the values of $\rho(\bM,e)$ and $\rho(\tilde{\bM},\tilde{e})$
        for pairs $(\bM,e)$ that are in case 1(b) of
        \cref{def:MON}. Indeed, all other values are imposed
        by the axioms. If we only wanted to respect Axiom (1), we
        could choose these values arbitrarily in $\{1,b\}$. Here we
        also have to be careful to perform  these choices
        simultaneously to respect Axiom (2) and to ensure
        coherence and integrality. For this we adapt the arguments of \cite[Section
        5.1]{Dolega2017a} to our settings.

	We first equip every connected vertex-coloured map with a fixed orientation of all its faces, given by a global orientation if the surface is orientable, and chosen arbitrarily for each face otherwise.
	Given an edge $e$ in a vertex-coloured map $\bM$ such that the pair $(\bM,e)$ is in case 1(b) of \cref{def:MON}, we let $\bN$ be the connected component of $\bM$ containing $e$.
Removing $e$ from $\bN$ creates a smaller map with two marked corners. We let $\rho(M,e)=1$ if the edge $e$ respects the fixed orientation of these corners, and $\rho(M,e)=b$ otherwise. By construction this choice respects Axiom (2). Coherence is clear, because the edges $e_i$ and $\tilde e_{i+1}$ in \cref{def:niceMON} have opposite conventional orientations along their faces and therefore are associated with the two values $1$ and $b$ (in some order). The fact that this MON is integral comes from the fact that we have chosen the orientation of faces from a global surface orientation whenever the map is orientable.
\end{proof}

\begin{remark}\label{rem:rhoSYM}
	We can construct a MON by choosing the two values in case 1(b) of \cref{def:MON} to be both equal to $(1+b)/2$. We obtain a MON denoted by $\rho_{SYM}$, which is not integral, but is coherent. Introducing $\rho_{SYM}$ is not necessary, strictly speaking, to prove the results of this paper, but it played a role in our discovery of the ``heuristic proofs'' given in \cref{subsec:proofCommut1}. Also, since we prove in this paper that the $b$-weighted enumeration of constellations is independent of the choice of a coherent MON (\cref{cor:independentOfMONPatched} page~\pageref{cor:independentOfMONPatched}), it is natural to expect that further works on the subject use the possibility to work with the coherent MON $\rho_{SYM}$, which is simple and canonical. 
\end{remark}

\subsection{Right-paths and the combinatorial decomposition}

\begin{definition}[Right-path](see \cref{fig:rightPathExample})
	Let $\bM$ be a $k$-constellation and $c$ be an oriented corner of color $0$ in $\bM$, lying in some face $f$. The sequence of $k$ edges $e_1,e_2,\dots,e_k$ that follow $c$ around $f$, in the orientation of $c$, is called the \emph{right-path} of $c$. Note that the edge $e_i$ has color $\{i-1,i\}$. 
\end{definition}
\begin{figure}%
\center\includegraphics[height=45mm]{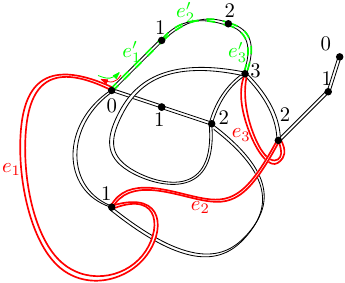}
	\caption{The $3$-constellation of \cref{fig:constellation}. The right-path $(e_1,e_2,e_3)$ of the (oriented) root is shown in (bold) red. In (dotted) green, we show the right-path $(e_1',e_2',e_3')$ of the same corner but with opposite orientation.}
\label{fig:rightPathExample}
\end{figure}

From the local coloring constraints that define $k$-constellations, we clearly have:
\begin{lemma}
	If $\bM$ is a $k$-constellation and $P$ is a right-path in $\bM$, then $\bM\setminus P$ is again a $k$-constellation (with possibly more connected components than $\bM$).
\end{lemma}

 We now introduce the \emph{combinatorial decomposition}%
 \footnote{Along the years more and more complicated algorithms to
   decompose rooted maps by edge deletion were found, and the present example may be the most complicated to date. Sometimes the name ``Tutte decomposition'' is used generically for them, here we prefer ``combinatorial decomposition''. 
   Tutte's original work~\cite{Tutte1962a} dealt with planar maps. Lehman and Walsh~\cite{WalshLehman1972} were the first to write a decomposition for the case of higher genera, and many works followed in the context of enumeration of orientable or non-orientable maps, see e.g.~\cite{BenderCanfield1991, Gao1993}. A combinatorial decomposition  for $k$-constellations in the orientable case appears in Fang's PhD thesis~\cite{Fang:PhD} in the case were one only tracks the face degrees (and not vertices of colour $0$). The decomposition presented in this paper contains these examples as special cases.
 The equations obtained by analyzing these decompositions are often called Tutte equations, but we prefer to use ``decomposition equations'' below, see \cref{subsec:TutteEq}.
 In the context of mathematical physics, similar equations are often called Dyson-Schwinger equations or loop equations, see e.g.~\cite{LandoZvonkin2004, Eynard:book}, although not all loop equations directly reflect a combinatorial decomposition, see e.g.~\cite{AlexandrovChapuyEynardHarnad2020}.}%
which consists in removing right-paths recursively from a connected $k$-constellation until the whole map has been exhausted.

 In this definition we assume that some underlying MON $\rho$ has been fixed.
\begin{definition}[Combinatorial decomposition]\label{def:randomDecomposition}
	Let $(\bM,c)$ be a rooted connected $k$-constellation.
	The combinatorial decomposition of $(\bM,c)$ is the recursive algorithm defined as follows:
\begin{itemize}
		\item We let  $\mathbf{M}_0=\mathbf{M}$ and we let
                  $c_0:=c$, $m := \deg(v(c))$, where $v(c)$ is the
                  vertex adjacent to $c$ (the root vertex of $\mathbf{M}$);
		\item
			For each $i$ from $1$ to $m$, we let $P_i=(e^{(i)}_1,e^{(i)}_2,\dots,e^{(i)}_k)$ be the right-path of the corner $c_{i-1}$ in $\bM_{i-1}$. We let $\mathbf{M}_i:=\mathbf{M}_{i-1}\setminus P_i$. For $i<m$ we let $c_{i}$ be the oriented corner induced by $c_{i-1}$ in the map~$\bM_i$.
		 \item We let $L$ be the ordered list of edges 	$$
	L=(e^{(1)}_1,e^{(1)}_2, \dots, e^{(1)}_k, e^{(2)}_1, e^{(2)}_2, \dots, e^{(2)}_k,\dots,\dots,e^{(m)}_1, e^{(m)}_2,\dots,e^{(m)}_k).
	$$
We say that the weight $\rho(\bM;L)$ has been \emph{collected} by the algorithm.
\item If the map $\bM_m$ is empty, then stop. If it is not, let
  $\bM^{(1)}, \dots, \bM^{(i)}$ be its connected components. For $j
  \in [1..i]$ let $c^{(j)}$ be the last corner in $\bM^{(j)}$ from
  which an edge was deleted in the execution of the algorithm (in
  other terms after removing an edge attached to $c^{(j)}$ in the execution of the algorithm no other
  edge was removed from $\bM^{(j)}$). This corner naturally 
		inherits the orientation of the last right-path removed from
  it. We root $\bM^{(j)}$ in the first corner of color $0$ following
  the oriented corner $c^{(j)}$ (that is to say, we visit consecutive corners of the face of $\bM^{(j)}$ which contains
  $c^{(j)}$ starting from this corner and following its orientation, until a corner of colour $0$ is reached, and we take that corner as the root). Let $\tilde{\bM}^{(1)}, \dots, \tilde{\bM}^{(i)}$ be their dual maps. 
			\item Run recursively the algorithm on each of the maps $\tilde{\bM}^{(1)}, \dots, \tilde{\bM}^{(i)}$.
		\end{itemize}
              \end{definition}
              
\cref{fig:CombDec} shows the combinatorial decomposition
            of the $3$-constellation from \cref{subfig:constellation}. The need to alternate between primal and dual maps in the decomposition may seem unnatural to the reader. As we will see, it comes from the fact that the differential equations we use to make the connection with Jack polynomials mix two sets of differential variables ($\pp$ and $\qq$, see Proposition~\ref{prop:mixDiffEqs}). 

	\begin{definition}
		Fix a MON $\rho$, and let  $(\bM,c)$ be a rooted connected $k$-constellation. We define the \emph{weight} $\vec\rho(\bM,c)$ of $(\bM,c)$ as the product of all the weights collected during the combinatorial decomposition of $(\bM,c)$.

              \end{definition}

                \begin{example}
                  Let us compute $\vec\rho(\bM,c)$ for the rooted  map $(\bM,c)$
                  of \cref{subfig:constellation}. During the
                  combinatorial decomposition we first remove four
                  right-paths attached to the root vertex, until the
                  3-constellation $\bM_4$ has more than one connected
                  component (it has two: the
                  trivial one consisting of the root vertex and
                  $\bM^{(1)}$ shown in \cref{fig:CombDec}). The edges
                  appearing in the list $L$ produced by the
                  combinatorial decomposition are depicted in
                  \cref{fig:CombDec}. One can check that
                  $\rho(\bM_0,P_1) = X_1\cdot X_2 \cdot 1$,
                  $\rho(\bM_1,P_2) = 1\cdot 1 \cdot b$,
                  $\rho(\bM_2,P_3) = b\cdot 1\cdot 1$,
                  $\rho(\bM_3,P_4) = 1\cdot 1 \cdot 1$, where $X_1,X_2
                  \in \QQ[b]$ are fixed polynomials associated with
                  $\rho$ by Axiom 1b) of \cref{def:MON} (the $i$-th
                  term in the product expressing $\rho(\bM_{j-1},P_j)$
                  corresponds to removing $e^{(j)}_i$).
                  In particular $\rho(\bM,L) = b^2\cdot X_1\cdot X_2$.
                  Finally, $\rho(\tilde\bM^{(1)})=1\cdot 1\cdot 1$ so that
                  $\vec\rho(\bM,c) = b^2\cdot X_1\cdot X_2$.
                \end{example}

                \begin{figure}%
\includegraphics[width=0.8\linewidth]{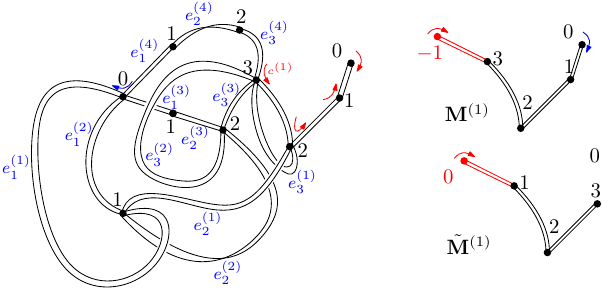}
\caption{Combinatorial decomposition of the $3$-constellation from
  \cref{subfig:constellation}. In the left we indicate which edges
  of $\bM$ are listed in the first part of the algorithm. Red oriented
  corners show how to root the connected component $\bM^{(1)}$ by first
  orienting the corner $c^{(1)}$ and following its orientation until
  we reach the first corner of color $0$. On the right hand side the
  red vertex and the red edge are not part of
  $\bM^{(1)}$ but they indicate the intermediate step of the
			construction of the dual map $\tilde\bM^{(1)}$, which also shows how
to orient the root of~$\tilde\bM^{(1)}$ (see Definition~\ref{def:duality} and the comment following it).}
\label{fig:CombDec}
\end{figure}
              
\begin{definition-lemma}\label{deflemma:nu}
	Let $\rho$ be an integral MON, and let $(M,c)$ be a connected rooted $k$-constellation. Then we have
	$$\vec\rho(\bM,c) = b^{\nu_{\rho}(\bM,c)},
	$$ 
	where $\nu_{\rho}(\bM,c)$ is a nonnegative integer, which is zero if and only if $\bM$ is orientable.
\end{definition-lemma}
\begin{proof}
	The fact that $\vec\rho(\bM,c)$ is a monomial is a direct consequence of the definitions. The fact that $\nu_{\rho}(\bM,c)$ is zero if and only if $\bM$ is orientable follows from the fact that it is orientable if and only if the weight $1$ (instead of $b$) is collected at each step of the combinatorial decomposition, which is clear by inspecting all cases in the definition of a MON.
\end{proof}

	\subsection{$b$-weights and $p_{\cdot}q_{\cdot}u_{\cdot}y_{\cdot}$-markings}

	In this paper we will consider generating  functions of constellations, and we will  be able to keep track of many parameters of these combinatorial objects in our formulas. In order to make our discussions as clear and readable as possible, we fix some terminology and notation now.

	To a constellation $\bM$ (possibly rooted, or labeled), we will associate several sorts of ``weights'':
	\begin{itemize}
		\item a \emph{$b$-weight}, which is a quantity in
                  $\mathbb{Q}[b]$, a priori dependent on the choice of
                  an underlying MON $\rho$.
			Example of $b$-weights
                  are the quantities $\vec\rho(\bM,c)$ or
			$b^{\nu_\rho(\bM,c)}$ defined above. We will restrict the word \emph{weight} to these quantities.
	  \item a monomial weight in the variables $p_i,q_i,u_i,y_i$, which serves as a marking keeping track of parameters of the map, such as the face or vertex degrees. To avoid confusion with the $b$-weights we will use the word \emph{marking} instead of \emph{weight} for these quantities.
	\end{itemize}

	For the rest of this paper we fix indeterminates $b$,
        $\pp=(p_i)_{i\geq 1}$, $\qq=(q_i)_{i\geq 1}$,
        $\yy=(y_i)_{i\geq 0}$, $\uu = (u_i)_{i\geq 1}$.
	If $\bM$ is a constellation we denote by $cc(\bM)$ its number
        of connected components, and by $F(\bM)$ the set of its
        faces. For $i\geq 0$ we denote by $V_i(\bM)$ the set of
        vertices of color $i$ and by $v_i(\bM)$ its cardinality.
Recall also that $|\bM|$ is the size, {\it i.e.} the number of corners of colour $0$, of $\bM$.
	\begin{definition}[Markings]
		Let $\bM$ be a $k$-constellation. The \emph{marking} of $\bM$ is the monomial
		\begin{align}\label{eq:marking}
		 \kappa(\bM):=\prod_{f\in F(\bM)} p_{\deg(f)}\prod_{v\in V_0(\bM)} q_{\deg(v)} \prod_{i=1}^k u_i^{v_i(\bM)}.
		\end{align}
		Let $(\bM,c)$ be a rooted $k$-constellation, and $f_c$ be its root face. The \emph{marking} of $(\bM,c)$ is the monomial
		$$
		 \vec\kappa(\bM,c):= y_{\deg (f_c)}\prod_{f\in F(\bM)\setminus\{f_c\}} p_{\deg(f)}\prod_{v\in V_0(\bM)} q_{\deg(v)} \prod_{i=1}^k u_i^{v_i(\bM)}
		 = \frac{y_{\deg (f_c)}}{p_{\deg (f_c)}}\kappa(\bM).
		$$
	\end{definition}
		In other words, our marking uses variables $p_i$ to
                record a non-root face of degree $i$, $q_i$ to record
                a vertex of colour $0$ and degree $i$, variables $u_i$
                to record a vertex of color $i$, and $y_i$ to record
                the fact that the root face has degree $i$. %

                \begin{example}
                  For the rooted $3$-constellation from
                  \cref{subfig:constellation} one has
                  $\vec\kappa(\bM,c)=y_5\cdot q_4\cdot q_1\cdot
                  u_1^4\cdot u_2^3\cdot u_3$.
                  \end{example}

\subsection{Generating functions of connected maps and the decomposition equations.}
\label{subsec:TutteEq}

Let $\rho$ be a MON.
We let \sloppy $\vec{H}_\rho(t;\pp,\qq, \yy,u_1,\dots,u_k)$ be the multivariate generating
function of rooted connected $k$-constellations given by the formula
\begin{equation}
\label{eq:MultivariateGeneratingMapsRootedConnected}
	\vec{H}_\rho(t;\pp,\qq, \yy,u_1,\dots,u_k) := \sum_{n\geq 1}\sum_{(\bM,c)} t^{n} \vec{\rho}(\bM,c) \vec\kappa(\bM,c),
\end{equation}
where the second sum is taken over rooted connected (unlabeled) $k$-constellations of size~$n$.
Formally $\vec{H}_\rho$ is viewed as formal power series
in $t$, with coefficients that are polynomials in the
variables $y_i,p_i,q_i,u_i$, with coefficients in $\mathbb{Q}(b)$, that is
\[ \vec{H}_\rho \in \QQ(b)[\yy,\pp,\qq,u_1,\dots,u_k][[t]].\]

For $m\geq 1$, we also denote by $\vec{H}_\rho^{[m]}$ the contribution to $\vec{H}_\rho$ of maps whose root vertex has degree $m$
\begin{equation}
  \label{eq:DefHRhoM}
	\vec{H}_\rho^{[m]}(t;\pp,\qq, \yy,u_1,\dots,u_k) := q_m^{-1}\sum_{n\geq 1}\sum_{(\bM,c)\atop \deg v_c=m} t^{n} \vec{\rho}(\bM,c) \vec\kappa(\bM,c),
\end{equation}
where the second sum is now taken over rooted connected (unlabeled) $k$-constellations of size~$n$ whose root vertex $v_c$ has degree $m$. Note that we do not count the root vertex in the marking (hence the factor $q_m^{-1}$). By definition one has:
\begin{align}
  \label{eq:HintoHm}
	\vec{H}_\rho = \sum_{m\geq 1} q_m \cdot \vec{H}_\rho^{[m]}.
\end{align}
We also consider the variant where the root face is marked with $\pp$-variables, namely
\begin{align*}
	H_\rho^{[m]} =\Theta_Y \vec H_\rho^{[m]}
	=q_m^{-1}\sum_{n\geq 1}\sum_{(\bM,c)\atop \deg v_c=m} t^{n} \vec{\rho}(\bM,c) \kappa(\bM),
\end{align*}
with
\begin{align}
     \Theta_Y &:= \sum_{i\geq 1}p_i\frac{\partial}{\partial y_i}.
\end{align}
Finally, we let $\pi$ be the operator that exchanges the sets
  of variables $\pp \leftrightarrow \qq$ and $u_{i}\leftrightarrow
  u_{k+1-i}$ for $1 \leq i \leq k$, and we let  
\begin{align}\label{eq:dualFunction}
\widetilde{H}_\rho^{[m]}:= \pi H_\rho^{[m]}.
\end{align}
Note that, by applying duality, we have
\begin{align}\label{eq:dualFunction2}
\widetilde{H}_\rho^{[m]}= 
	p_m^{-1}\sum_{n\geq 1}\sum_{(\bM,c)\atop \deg f_c =m} t^{n} \vec\rho {(\tilde{\bM}, \tilde{c})} \kappa(\bM),
\end{align}
where the second sum is now taken over rooted connected (unlabeled) $k$-constellations of size~$n$ whose root \emph{face} has degree $m$. Note that in this sum the $b$-weight is computed  on the dual rooted map $(\tilde{\bM}, \tilde{c})$ of $(\bM,c)$, and that we used that $\kappa(\tilde{\bM}) = \pi \kappa (\bM)$. 

We will now state a set of equations (which we call ``decomposition equations'') that characterizes these functions. We first need to define some operators:
\begin{align*}
\Lambda_{Y} &:= (1+b)\sum_{i,j\geq 1}y_{i+j-1}\frac{i\partial^2}{\partial
  p_i \partial y_{j-1}} +\sum_{i,j\geq
  1}y_{i-1}p_j\frac{\partial}{\partial y_{i+j-1}}
+b\cdot\sum_{i\geq
  0}y_{i}\frac{i\partial}{\partial y_i },
	\\
Y_+ &:= \sum_{i\geq 0}y_{i+1}\frac{\partial}{\partial y_i}.
\end{align*}

\begin{theorem}[Decomposition equations]\label{thm:TutteConnected}
	Let $\rho$ be any coherent MON.
Then	
	the family of generating series $\vec H_\rho^{[m]}\equiv \vec H_\rho^{[m]}(t;\pp,\qq,\yy,u_1,\dots,u_k)$
	satisfies the following set of equations, for $m\geq 1$:
	\begin{equation}\label{eq:TutteConnected}
\vec H_\rho^{[m]} =t^m  \cdot
		\Big(Y_+\prod_{l=1}^{k}\big(\Lambda_{Y}+u_l + \sum_{i,j\geq 1} y_{j+i-1}
		\widetilde{H}_\rho^{[i]} \frac{\partial}{\partial y_{j-1}}
		\big)\Big)^{m}  (y_0).
\end{equation}
\end{theorem}

\begin{corollary}\label{cor:TutteThetaConnected}
		Let $\rho$ be any coherent MON.
Then	
	the family of generating series $(H_\rho^{[m]})_{m\geq 1}$ 
	is fully characterized by the following set of equations, for $m\geq 1$:
	\begin{equation}\label{eq:TutteThetaConnected}
H_\rho^{[m]}(t;\pp,\qq,u_1,\dots,u_k) =t^m  \cdot
		\Theta_Y \Big(Y_+\prod_{l=1}^{k}\big(\Lambda_{Y}+u_l + \sum_{i,j\geq 1} y_{j+i-1}
		\widetilde {H}_\rho^{[i]} \frac{\partial}{\partial y_{j-1}}
		\big)\Big)^{m}  (y_0),
\end{equation}
	together with~\eqref{eq:dualFunction}.
\end{corollary}
\begin{proof}[Proof of the corollary]
	Equation~\eqref{eq:TutteConnected} implies \eqref{eq:TutteThetaConnected} by applying the operator $\Theta_Y$ to both sides. It is clear that this set of equations characterizes the functions since coefficients can be computed inductively from the equations, order by order in $t$.
\end{proof}

In particular, we observe that
\begin{corollary}\label{cor:independentOfMONPatched}
	The functions $\vec H_\rho$, $H_\rho^{[m]}$, $\widetilde{H}_\rho^{[m]}$ do not depend on the coherent MON $\rho$.
\end{corollary}

	\subsection{Unconnected functions}

        Let us consider the following antiderivative of $\Theta_Y \vec H_\rho$:
$$H_\rho :=
\sum_{n\geq 1} \frac{1}{n} \sum_{(\bM,c)} t^{n} \vec{\rho}(\bM,c) \kappa(\bM),
$$
where the second sum is taken over rooted connected (unlabeled)
$k$-constellations of size~$n$ so that
\[ t \frac{\partial}{\partial t} H_\rho = \Theta_Y \vec H_\rho.\]
Then the series $F_\rho$ is defined by 
$$F_\rho = \exp \frac{1}{1+b} H_\rho.$$
Note that by Corollary~\ref{cor:independentOfMONPatched} it does not
depend on the coherent MON $\rho$ and the following identity holds
\begin{align}\label{eq:newDefFrho}
(1+b) t \frac{\partial}{\partial t}\ln F_\rho =\Theta_Y \vec H_\rho.
\end{align}

We now want to give a combinatorial interpretation to coefficients of $F_\rho$. Because each connected rooted constellation has $n! 2^{n-1}$ different labellings, we can also write 
$$
H_\rho =
\sum_{n\geq 1} \frac{1}{n} \sum_{(\bM,c)} \frac{t^{n}}{2^{n-1} n!} \vec{\rho}(\bM,c) \kappa(\bM)
$$
where the second sum is now taken over \emph{labeled} and rooted connected  $k$-constellations of size~$n$. This can also be rewritten
$$
H_\rho =
\sum_{n\geq 1} \sum_{\bM} \frac{t^{n}}{2^{n-1} n!} \kappa(\bM) \cdot \mathbf{E}_{c\in \mathbf{M}} [\vec{\rho}(\bM,c)] ,
$$
where the second sum is now taken over \emph{labeled} (but no more rooted) connected  $k$-constellations of size~$n$, and where $\mathbf{E}_{c\in \mathbf{M}}$ now denotes expectation with respect to a corner $c$ of colour $0$ chosen uniformly at random among those of $\bM$.

\begin{definition}
	Define the weight  of a labeled connected constellation $\bM$
        as $\tilde \rho(\bM):=\mathbf{E}_{c\in \mathbf{M}}
        [\vec{\rho}(\bM,c)]$. Extend this definition multiplicatively
        to unconnected (but still labeled) constellations.
\end{definition}

Because $\tilde{\rho}(\bM)$ is multiplicative on connected components
by definition, and $\kappa(\bM)$ and  $2^{|\bM|-cc(\bM)}$ also are,
the generating function $F_\rho = \exp \frac{1}{1+b} H_\rho$ can be
directly interpreted as the generating function of unconnected objects with these markings. More precisely:
\begin{theorem}
	The generating function $F_\rho\equiv F_\rho (t;\pp,\qq,u_1,\dots,u_k)$ defined by~\eqref{eq:newDefFrho} is given by the expansion:
	\begin{align}\label{eq:expansionUnconnected}
	F_\rho = 1+\sum_{n \geq 1}\sum_{\bM} \frac{t^{n}}{2^{n-cc(\bM)}n!} \frac{\tilde{\rho}(\bM) \kappa(\bM)}{(1+b)^{cc(\bM)}},
	\end{align}
	where the second sum is taken over labeled $k$-constellations of size~$n$, connected or not.
      \end{theorem}

\begin{remark}
	We will prove (Theorem~\ref{thm:mainInSection5}) that $F_\rho$ is in fact equal to the function $\tau^{(k)}_b$ defined in~\eqref{eq:JackIntro}. Thus the last theorem gives an explicit interpretation of the coefficients of $\tau^{(k)}_b$. We will also show (Lemma~\ref{lem:MultiJack}) that $F_\rho=\tau^{(k)}_b$ satisfies the following equation
	\begin{align}\label{eq:TutteUnconnected}
		m \frac{q_m\partial}{\partial q_m} F_\rho (t;\pp,\qq,u_1,\dots,u_k) = \Theta_Y t^m \cdot q_m \cdot
	\big(Y_+\prod_{l=1}^{k} (\Lambda_{Y}+u_l)\big)^{m}  \frac{y_0}{1+b} F_\rho(t;\pp,\qq,u_1,\dots,u_k),
	\end{align}
	and similarly, in Corollary~\ref{cor:equalInSection5}, we will show that 
	$$
	 m \frac{\partial }{\partial q_m} H_\rho = H_\rho^{[m]}.
	$$
	This equation has the following interpretation. Fix a monomial $\mathbf{m}=p_\lambda q_\mu u_1^{v_1}\dots u_k ^{v_k}$ such that the number of parts in $\mu$ equal to $m$ is nonzero, and that there exists at least one connected $k$-constellation with marking $\mathbf{m}$. Then, if $\bM$ denotes a random connected $k$-constellation such that $\kappa(\bM)=\mathbf{m}$, chosen uniformly at random, one has
	$$
	\mathbf{E}_{\mathbf{M}} 
	\mathbf{E}_{c\in \mathbf{M}}  [\vec{\rho}(\bM,c)]
=
\mathbf{E}_{\mathbf{M}} 
	\mathbf{E}_{c\in \mathbf{M}}  [\vec{\rho}(\bM,c) | \deg v_c = m].
	$$
	This ``symmetry'' is not directly apparent on the combinatorial model.
\end{remark}

	\begin{remark}[Other relations and deformed Virasoro constraints]\label{rem:VirasoroPatched}
	In the special cases of $b\in \{0,1\}$ we can obtain other
        equations, which do not involve differential operators
	with respect to the variables $(q_i)_{i \geq 1}$. Indeed, it follows from our proofs that for $b\in\{0,1\}$,
		one has
	\begin{align}\label{eq:Virasoro3}
	\frac{\ell\partial}{\partial p_\ell}
          F_\rho =
          [y_\ell] \sum_{m \geq 1}t^m \cdot q_m \cdot
	\big(Y_+\prod_{l=1}^{k} (\Lambda_{Y}+u_l)\big)^{m}  \frac{y_0}{1+b} F_\rho, 
	\end{align}
	where $[y_\ell]$ is coefficient extraction.
		Indeed, our proofs show that the right-hand side is the generating function of rooted constellations in which the root face has degree $\ell$, without attributing any marking to that face (this follows from the combinatorial interpretation of operators given below, which in the case $b\in \{0,1\}$ can be applied also to the combinatorial decomposition of the root component in a rooted unconnected constellation).  Such maps can also be obtained by distinguishing a face of degree $\ell$ in an unrooted constellation and choosing one of the $\ell$ corners of colour $0$ it contains as the root, giving rise to the expression of  the left-hand side (note that for $b\not \in \{0,1\}$ this argument would not work, since the $b$-weight attributed to a map depends on the choice of the root corner).
		In the special
        case $q_i=\delta_{1,i}$ and $k=2$, Equation~\eqref{eq:Virasoro3} is precisely
        the $\ell$-th Virasoro constraint for dessins d'enfants for
        $b=0$ (or its non-orientable generalisation for $b=1$),
        see~\cite{KazarianZograf2015}. In the general case, possible
        $b$-analogues of the differential
        equations~\eqref{eq:Virasoro3} and their links with the
        Virasoro algebra and its central extensions deserve further
        interest and will be studied in further works.\footnote{Note
          added during revision: Virasoro constraints for certain 
		$b$-deformed models (including the case $q_i=\delta_{1,i}$ and $k=2$ of~\eqref{eq:Virasoro3}, for general $b$) were discussed and proved
          in~\cite{BonzomChapuyDolega2022}, and they were used in the subsequent
          paper~\cite{BonzomChapuyDolega2022a} to prove
		recursion formulas counting various non-orientable maps. However, we still cannot prove~\eqref{eq:Virasoro3} for general values of $b$.}
\end{remark}

\section{Differential operators: partial right-paths and commutation
  relations}
\label{sec:Operators}

\subsection{Interpretation of the operator $\Lambda_{Y}$ and proof of the decomposition equation}
\label{subsec:proofTutte}

\begin{figure}%
\center\includegraphics[width=\linewidth]{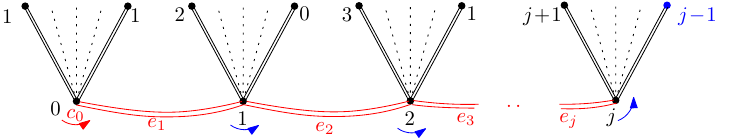}
	\caption{A partial right-path as in
          \cref{def:partialPath}. Black edges belong to a
          $k$-constellation $\bN$, and red edges form a partial
          right-path $P_j$. In the case $j=k$, the color $j+1$ should
          be replaced by $j-1=k-1$ on the illustration. The blue vertex of color $j-1$
          corresponds to the vertex described in the last axiom of the
        definition, and the blue corners corresponds to the tour of
	the face $f_0$ which determines this blue vertex. Note that, for $j<k$, if one replaced the edge $e_j$ by its twist, the configuration would not be a partial right-path anymore, as the vertex in the last axiom of Definition~\ref{def:partialPath} would now have colour $j+1$.}
\label{fig:partialRightPath}
\end{figure}

To prove the decomposition equation, we will show that each $k$-constellation can be constructed from smaller ones by adding edges one by one, thus working with intermediate objects that do not fully satisfy the constraints defining $k$-constellations.

\begin{definition}[Partial right-path; see \cref{fig:partialRightPath}]\label{def:partialPath}
	Let $\bN$ be a $k$-constellation and $j\in [0..k]$. 
	Assume that $\bM_j$ is a map formed by adding a sequence of
        new edges $P_j = (e_1,e_2, \dots, e_j)$ to $\bN$ (possibly
        also using some new vertices of color $i$ for $i\in [0..j]$;
        $P_0 := \emptyset$ by convention), and assume that $\bM_j$ is rooted at some oriented corner $c_0$ of color $0$, incident to $e_1$, such that:
	\begin{itemize}
		\item the edge	$e_i$ has color $(i-1,i)$, for $i\in[1..j]$;
		\item starting from $c_0$ and following the tour of the face $f_0$ containing it in $\bM_j$, we follow the sequence of edges $e_1,e_2,\dots,e_j$ in this order;
		\item for $i<j$, the vertex of color $i$ on $P_j$ satisfies the local constraints of a $k$-constellation in the map $\bM_j$; 
		\item if $j\neq0$, the vertex of $\bM_j$ that follows $e_j$ in the
                  tour of the face $f_0$ starting from $c_0$ and following the sequence of edges $e_1,e_2,\dots,e_j$ has color $j-1$. 
	\end{itemize}
	Then  we say that $P_j$ is a \emph{partial right-path of length $j$} for $\bN$.
\end{definition}

The following is clear from definitions:
\begin{lemma}
	If $P=(e_1,\dots,e_k)$  is a partial right-path of length $k$ for $\bN$, starting at some corner $c_0$, then $\bN\cup P$ is a $k$-constellation, and $P$ is the right-path of $c_0$ in $\bN \cup P$.
\end{lemma}

We will now construct operators that ``build'' partial right-paths, but for this we first have to define what markings we associate to the intermediate objects that are not exactly $k$-constellations.
\begin{definition}[Markings for partial right-paths]
	Let $\bN$ be a constellation and $P_j$ be a partial right-path of length $j$ for $\bN$, of root corner $c_0$ and root vertex $v_0$.
	We define the marking $\hat\kappa(\bN\cup  P_j,c_0)$ of the ``intermediate'' constellation $\bN\cup  P_j$ as for usual rooted constellations, except that when measuring the degree of faces, we do not count the root corner $c_0$, and that we do not count the factor $q_i$ corresponding to the root vertex. That is to say:
$$
	 \hat\kappa(\bN\cup  P_j,c_0):=
	 y_{\deg (f_0)-1}\prod_{f\in F(\bN\cup  P_j)\setminus\{f_0\}} p_{\deg(f)}\prod_{v\in V_0(M)\setminus\{v_0\}} q_{\deg(v)} \prod_{i=1}^k u_i^{v_i(\bN\cup  P_j)}
	 ,
$$
where $f_0$ (resp. $v_0$) is the face containing $c_0$
(resp. the vertex incident to $c_0$) in $\bN\cup
P_j$. 
\end{definition}
\noindent Note that when $c_0$ is the only corner of color $0$ in the face $f_0$, then $\hat\kappa(\bN\cup P_j,c_0) $ involves the variable $y_0$.

	The following proposition says that the effect of the operator $\Lambda_{Y}+u_j$ is to extend the length of a partial right-path by one unit.
\begin{proposition}[Interpretation of $\Lambda_{Y}$]\label{prop:interpretGammaPatched}
	Let $\bN$ be a $k$-constellation, $j\in[0..k-1]$, and $P_j=(e_1,\dots,e_j)$ be a partial right-path for $\bN$. 
Assume that $N\cup P_j$ is connected.  
	Let $\rho$ be a coherent MON.
	Then
	\begin{align*}
		(\Lambda_{Y}+u_{j+1})
	\hat\kappa(\bN \cup  P_j,c_0)
		= \sum_{e_{j+1}} \rho(\bN \cup  P_j \cup \{e_{j+1}\};e_{j+1})
	\hat\kappa(\bN\cup P_j \cup \{e_{j+1}\},c_0)
	,\end{align*}
	where the sum is taken over all possible additions of an edge $e_{j+1}$ (possibly using a new vertex of color $j+1$) such that $(e_1,\dots,e_j,e_{j+1})$ is a partial right-path of length $j+1$ for $\bN$.
\end{proposition}

\begin{proof}[Proof of \cref{prop:interpretGammaPatched}]
	In order to add the edge $e_{j+1}$ to the partial path $P_j$, we should connect the corner $c_j$ which follows $e_j$ on $P_j$, to some corner of color $(j+1)$, by some new edge  $e_{j+1}$. We can already note that after doing this, the colour constraints of $k$-constellations around the vertex of colour $j$ on $P_j$ will automatically be satisfied, from the last property of \cref{def:partialPath}.
	
	We first remark that if $c_{j+1}$ is a corner of color $j+1$ in $\bN$ incident to a vertex $v_{j+1}$, there are two different edges $e_{j+1}, \tilde e_{j+1}$ that can be added joining $c_j$ to $c_{j+1}$, where one is the twist of the other. 
We distinguish two cases:
	\begin{itemize}[itemsep=0pt, topsep=0pt,parsep=0pt, leftmargin=12pt]
		\item if $j+1=k$, then both $(e_1,\dots,e_{k-1}, e_k)$ and $(e_1,\dots,e_{k-1}, \tilde{e}_k)$ are (partial) right-paths for $\bN$. Indeed, in this case there are no nontrivial colour constraints to satisfy.
		\item if $j+1<k$, then exactly one choice of $e\in \{e_{j+1},\tilde e_{j+1}\} $ is such that $(e_1,\dots,e_{j}, e)$ is a partial right-path for $\bN$. Indeed, since $\bN$ is a $k$-constellation the corner $c_{j+1}$ in $\bN \cup P_j$ is incident to two edges of color $ \{j,j+1\}$ and $\{j+1,j+2\}$, and the last constraint in \cref{def:partialPath} requires that after following $e$ along the path one reaches the edge of color $\{j,j+1\}$, which forces the choice of the twist.  In this proof we will say that this choice of $e$ is the \emph{valid choice}.
	\end{itemize}

	\smallskip
	Then we observe that each vertex of $\bM_j:=\bN\cup\{P_j\}$ satisfies the local constraints of a $k$-constellation, except for the vertex of colour $j$ on $P_j$.
	This implies that, in the map $\bN\cup P_j$ each non-root face of degree $d$ contains exactly $2d$ corners of label $j+1$ if $j+1<k$, and $d$ corners of label $k$. The same is true for the root face provided we do not count the corner $c_0$ in the degree.
	Let $f$ be a face of $\bM_j$ of degree $d$ (or degree $d+1$ if
        $f$ is the root-face). Orient $f$ arbitrarily and assume that
        $j+1<k$. Let $u_1,\dots,u_{2d}$ be the list of corners of
        color $j+1$ in $f$, with respect to the chosen orientation. When following the tour of $f$, the labels of the two corners visited before and after $u_i$ are either $(j,j+2)$ or $(j+2,j)$, and moreover, corners of the two types alternate. 
	For each such corner $u_i$, if we want to create a new edge $e_{j+1}$ from $c_j$ to $u_i$, only one possible twist of that new edge is a valid choice, and moreover, the type of twist which is valid alternates with corners.
	This observation being recorded, let us proceed with the proof by distinguishing some cases. There are several ways to create the new edge $e_{j+1}$:
	\begin{itemize}[itemsep=0pt, topsep=0pt,parsep=0pt, leftmargin=20pt]
		\item[(i)] {\it we create a new isolated vertex of color $j+1$, linked by an edge to $c_j$.} This does not contribute to the $b$-weight, and the contribution to the marking is $u_{j+1}$. %
		\item[(ii)] {\it we connect $c_j$ to a corner of color $c_{j+1}$ in a non-root face $f$.}
			If this chosen face has degree $d$, the degree of the root face will increase by $d$.
If $j+1<k$, from the observation recorded above, the $2d$ corners of color $(j+1)$ in the chosen face give rise to $2d$ valid choices of edges whose twists alternate, and because $\rho$ is coherent, the total weight of these $2d$ possible additions sum up to
			\begin{align}\label{eq:contribution}
	(1+b)+(1+b)+\dots+(1+b) = d (1+b).
			\end{align}
			If $j+1=k$, there are  $d$ corners of color $(j+1)$ in the chosen face, and each of them corresponds to two possible choices of edges $e_{j+1}$ and $\tilde{e}_{j+1}$, which are both valid choices. By Axiom 1(b) of \cref{def:MON}, the sum of contributions to the $b$-weight of adding $e_{j+1}$ or $\tilde{e}_{j+1}$ is $(1+b)$. The total contribution in this case is thus again $d(1+b)$.
				
			Therefore both subcases of case $(ii)$ the contribution is $d(1+b)$ where $d$ is the degree of~$f$. We conclude that case (ii) is described by the operator:
			$$
(1+b)\sum_{a,d\geq 1}y_{a+d-1}\frac{d\partial^2}{\partial
  p_d \partial y_{a-1}} $$

		\item[(iii)] {\it we connect $c_j$ to a corner $c_{j+1}$ of color $j+1$ inside the root face, and we choose the twist of the new edge so that we do not create any new face.} Let $(d+1)$ be the degree of the root face. If $j+1<k$, from the observation recorded above, only half of the $2d$ such edges are valid choices. On the other hand if $j+1=k$, we have $d$ such possible edges and all of them are valid. Hence the number of possible choices is $d$ in both cases. Moreover the degree of the root face does not increase, and the corresponding $b$-weight for each such choice is $b$ by Axiom 1(c) of \cref{def:MON}. Therefore contribution for this case is:
	$$b\cdot\sum_{d\geq
			0}y_{d}\frac{d\partial}{\partial y_d }.$$

		\item[(iv)] {\it we connect $c_j$ to a corner $c_{j+1}$ of color $j+1$ inside the root face, and we twist the edge so that we create a new face in addition to the root face.}
			As before, if $j+1<k$ only half of the $2d$ corners of label $j+1$ in that face are valid choices from the observation above, and the $d$ valid choices alternate around the root face with the $d$ non-root corners of color $0$. Similarly if $j+1=k$, there are $d$ valid choices that alternate around the root face with the $d$ non-root corners of color $0$. 

			Therefore given $i,j \geq 1$, if the root face has degree $i+j$, there is exactly one choice of valid corner such that after adding the edge the new root face has degree $i$ (and the newly created face then has degree $j$). Moreover, the corresponding $b$-weight is $1$ (Axiom~1(c)), so the contribution for this case is:
			$$\sum_{i,j\geq
  1}y_{i-1}p_j\frac{\partial}{\partial y_{i+j-1}}.$$
	\end{itemize}
	By summing contributions of  cases (ii)-(iii)-(iv), we
        recognize the definition of the operator $\Lambda_{Y}$, so the
        contribution of the four cases (i)-(ii)-(iii)-(iv) is $\Lambda_{Y}+u_{j+1}$ and the proof is complete.
\end{proof}

\begin{remark}We required the MON to be \emph{coherent} so that the second case of the decomposition equation gave rise to the correct weight. One could imagine relaxing further the notion of coherent MON to require only~\eqref{eq:contribution} to hold, rather than the stronger property that edges can be grouped in pairs of weight $(1+b)$ each.
\end{remark}

	\begin{remark}The assumption that $\bN\cup P_j$ is connected is not strictly needed to give an interpretation of the operator $\Lambda_Y+u_i$ as adding an edge, but working with unconnected constellations would require to use a more sophisticated marking taking into account the number of connected components, in the spirit of~\eqref{eq:expansionUnconnected}. We will not need this discussion and prefer to avoid it. We leave to the reader the task of giving a direct interpretation of~\eqref{eq:TutteUnconnected} along these lines.
\end{remark}

	The following proposition is the counterpart of Proposition~\ref{prop:interpretGammaPatched} for the case when the next edge on the partial right-path joins to a new connected component. 
\begin{proposition}\label{prop:interpretGammaPatchedDisconnect}
	Let $\bN$ be a $k$-constellation, $j\in[0..k-1]$, and
        $P_j=(e_1,\dots,e_j)$ be a partial right-path for $\bN$. 
	Assume that $\bN \cup P_j$ is connected. 
	Let
        $\rho$ be a coherent MON.  
	Then
	\begin{align*}
		&\left(\sum_{i,j\geq 1}  y_{j+i-1}
		\widetilde{H}_\rho^{[i]} \frac{\partial}{\partial y_{j-1}}\right)
	\hat\kappa(\bN \cup  P_j,c_0)t^{|\bN|} \\
		&\hspace{2cm}	= \sum_{e_{j+1}, \bN'} \rho(\bN' \cup  P_j \cup \{e_{j+1}\};e_{j+1}) \vec\rho {(\tilde{\bN}'', \tilde{c})}
	\hat\kappa(\bN'\cup P_j \cup \{e_{j+1}\},c_0)t^{|\bN'|},
\end{align*}
	where the sum is taken over all $k$-constellations
	$\bN'$ such that $P_{j+1}=(e_1,\dots,e_j,e_{j+1})$ is a partial
	right-path of length $j+1$ for $\bN'$, and such that removing
        $e_{j+1}$ from the connected map $\bN'\cup P_{j+1}$ disconnects it into two components such that the component containing $P_j$ is $\bN\cup P_j$.
	In the sum, the other connected component is denoted by $\bN''$ and it is rooted at the first corner of colour $0$ following the corner from which $e_{j+1}$ was deleted,
	denoted by $c$. We denote by $(\tilde{\bN}'', \tilde{c})$ the dual of the rooted map $(\bN'',c)$.
\end{proposition}
\begin{proof}
	In order to build a map $\bN'$ as in the statement of the proposition, we should connect the corner $c_j$ which follows $e_j$ on $P_j$, to some corner of color $(j+1)$ in some new connected $k$-constellation $\bN''$, by adding some new edge  $e_{j+1}$ to a corner $c_{j+1}$ of $\bN''$. As in the proof of Proposition~\ref{prop:interpretGammaPatched},  after doing this the colour constraints of $k$-constellations around the vertex of colour $j$ on $P_j$ will automatically be satisfied, from the last property of \cref{def:partialPath}.

	Conversely, let $(\bN'',c)$ be a rooted connected $k$-constellation whose
        root face $f$ has degree $i$. There is a unique
        way of adding a valid edge from the corner $c_j$ to a
        corner of colour $(j+1)$ in $f$, such that the first corner of colour $0$
	following the edge $e_{j+1}$ is equal to $c$ (indeed, similarly as in the proof of Proposition~\ref{prop:interpretGammaPatched}, valid corners and corners of colour $0$ alternate along faces). Moreover, by Axiom
        1(a) of Definition~\ref{def:MON}, the contribution to the
        $b$-weight of this addition is equal to $1$. 

	The contribution for the choice of the map $\bN''$, with root
        face of degree $i$, is given by the generating function
        $\widetilde{H}_\rho^{[i]}$, where we note that we are
        computing it with the dual $b$-weight
        $\vec\rho(\tilde{\bN}'',\tilde{c})$ as in \eqref{eq:dualFunction2}. Moreover, the root face will increase by $i$ when connecting the edge $e_{j+1}$. Therefore the overall contribution for the choice of $\bN'$ and $e_{j+1}$ is equal to
$$\left(\sum_{i,j\geq 1}  y_{i+j-1}
		\widetilde{H}_\rho^{[i]} \frac{\partial}{\partial y_{j-1}}\right)
	\hat\kappa(\bN \cup  P_j,c_0)t^{|\bN|}.  ~ ~ ~ ~ ~~ ~ ~ ~ ~ ~ \qedhere
	$$
      \end{proof}

      We are now ready to prove the decomposition equation.
\begin{proof}[Proof of \cref{thm:TutteConnected}]
	We invert the combinatorial decomposition of the previous
        section. Any rooted connected (unlabeled) $k$-constellation $\bM$, with a root vertex of degree $m$, can be constructed as follows:
	\begin{enumerate}[itemsep=0pt, topsep=0pt,parsep=0pt, leftmargin=20pt]
		\item create a new isolated vertex of color $0$, with a marked corner $c_1$; 
		\item for $i$ from $1$ to $m$ do:
			\begin{enumerate}[itemsep=0pt, topsep=0pt,parsep=0pt, leftmargin=20pt]
				\item let $P^{(i)}_0$ be a new (empty) partial-right path of length $0$, rooted at $c_{i}$
				\item extend the partial right-path
                                  $P^{(i)}_0$ into partial right-paths
                                  $P^{(i)}_1, \dots, P^{(i)}_k$, by
                                  adding edges one by one (possibly
                                  adding new vertices, or connecting
                                  them with rooted connected
                                  (unlabeled) $k$-constellations along the way);
				\item once the right-path $P^{(i)}_k$ has been created, call $c_{i+1}$ the corner that follows $c_i$ around $v_0$, and reroot the current constellation at $c_{i+1}$; 
			\end{enumerate}
	\end{enumerate}
	The contribution of step (1) is simply $y_0$. By \cref{prop:interpretGammaPatched,prop:interpretGammaPatchedDisconnect}, for each $i
        \in [1..m]$ the contribution of Step 3(b) is given by the
        product of operators
        \[\left(\Lambda_{Y}+u_k+\sum_{i,j\geq 1}  y_{j+i-1}
		\widetilde{H}_\rho^{[i]}\frac{\partial}{\partial y_{j-1}})\right)\cdots \left(\Lambda_{Y}+u_1+\sum_{i,j\geq 1}  y_{j+i-1}
		\widetilde{H}_\rho^{[i]}\frac{\partial}{\partial
                  y_{j-1}})\right).\]
            Indeed, note that the dual $b$-weight
            $\rho(\tilde{\bN}'',\tilde{c})$ appearing in the R.H.S. of
            the equation given by
            \cref{prop:interpretGammaPatchedDisconnect} is coherent
            with the fact that to compute the $b$-weight in the
            combinatorial decomposition of $\bN'$, the $b$-weight of the smaller component $\bN''$ will be computed from its dual map $\tilde{\bN}''$. 
            After Step 3(b) the corner $c_i$ is no longer the
            root corner of the current map, so it has to be counted in
            the marking $\hat \kappa$. This is taken into account by
            the operator $Y_+$, so the overall contribution of Steps
            3(b) and 3(c) is given by the operator
            \[\Big(Y_+\prod_{l=1}^{k}\big(\Lambda_{Y}+u_l + \sum_{i,j\geq 1} y_{j+i-1}
		\widetilde {H}_\rho^{[i]} \frac{\partial}{\partial y_{j-1}}
		\big)\Big)^{m}.\]
        \\
        At the end of the process, the newly created vertex $v_0$
        contributes a monomial $q_m$ to the marking, but this
        contribution is killed by the factor $q_m^{-1}$ in front of
        the defining equation\eqref{eq:DefHRhoM}. Finally the fact
        that the size of the map increases by $m$ is taken into
        account by a factor $t^{m}$. 
	Overall, the contribution of steps (1)--(3) is thus equal to
            \[t^m\Big(Y_+\prod_{l=1}^{k}\big(\Lambda_{Y}+u_l + \sum_{i,j\geq 1} y_{j+i-1}
		\widetilde {H}_\rho^{[i]} \frac{\partial}{\partial y_{j-1}}
		\big)\Big)^{m}(y_0),\]
	which finishes the proof.
      \end{proof}

\subsection{Commutation relations and Lax pairs}
\label{subsec:Commutation}

The two theorems below are the keystone of this paper. They show that the operators that appear in the decomposition equations can be alternatively defined inductively by certain recurrence relations involving commutators. Their proof is the hardest part of the paper and will occupy much of the next sections.

These relations are the crucial link between Jack polynomials and
weighted generalized branched coverings (via constellations).  From now on we let $\alpha=1+b$ and we will use either $b$ or $\alpha$, or both, in our notation.

\begin{definition}
The Laplace-Beltrami operator $D_\alpha$ is the differential operator defined by
\begin{equation}
\label{eq:Laplace-Beltrami}
   D_\alpha = \frac{1}{2}\left((1+b)\sum_{i,j\geq 1}p_{i+j}\frac{ij\partial^2}{\partial
  p_i \partial p_{j}} +\sum_{i,j\geq
  1}p_{i}p_j\frac{(i+j)\partial}{\partial p_{i+j}}+b\cdot\sum_{i\geq
  1}p_{i}\frac{i(i-1)\partial}{\partial p_i}\right).
\end{equation}
\end{definition}
Here and below we let $\PPP:=\mathbb{Q}(b)[p_1,p_2,\dots]$. Moreover we let $[\cdot,\cdot]$ denote the algebra commutator, $[A,B]=AB-BA$. 
\begin{theorem}[First commutation relations]\label{thm:commut1}
Define the differential operators $(A_j)_{j\geq 1}$ on $\PPP$ by:
\begin{align}\label{eq:EjCombi}
 A_{j+1}:=\Theta_Y Y_+\Lambda_{Y}^{j} \frac{y_0}{1+b}, \ \ j\geq0.
\end{align}
Then these operators satisfy the recurrence formula
\begin{align}\label{eq:commut1}
	A_1 &= p_1/(1+b) \ \ \ ,\ \ A_{j+1} = [D_\a,A_j], \mbox { , for } j\geq 1.
\end{align}
These equalities hold between operators on $\PPP$.
\end{theorem}
We now define the operator on $\PPP$.
$$
\Omega^{(k)}_Y:=
\Theta_Y Y_+ \prod_{j=1}^k (\Lambda_{Y}+u_j)\Lambda_{Y}\frac{y_0}{1+b}.
$$
We have:
\begin{theorem}[Second commutation relations]\label{thm:commut2}
Define the differential operators $(B^{(k)}_m)_{m\geq 1}$  by:
\begin{align}\label{eq:FmCombi}
 B^{(k)}_{m}:= (m-1)! \Theta_Y \big(Y_+\prod_{i=1}^{k}(\Lambda_{Y}+u_i)\big)^m \frac{y_0}{1+b}, \ \ m\geq1.
\end{align}
Then these operators satisfy the recurrence formula, for $m\geq 1$
\begin{align}\label{eq:commut2}
B^{(k)}_1 =%
\Theta_Y Y_+ \prod_{j=1}^k (\Lambda_{Y}+u_j) \frac{y_0}{1+b}	
	\ \ \ , \ \ \  B^{(k)}_{m+1} = [\Omega^{(k)}_Y,B^{(k)}_m] \mbox { , for } m\geq 1.
\end{align}
These equalities hold between operators on $\PPP$.
\end{theorem}

\begin{remark}
	The equalities in \cref{thm:commut1,thm:commut2} hold between operators acting on $\PPP$. They do not hold on a larger space containing also the variables $y_i$. This simple fact makes the proof of these theorems difficult. The strategy we design in the next section will demand to promote these operators to such a larger space,  on which induction can be applied.  The fact that we encounter a difficulty here will hardly be a surprise for combinatorialists: the variables $y_i$ play the role of ``catalytic'' variables that enabled us to write combinatorial equations in the first place, but we then pay the price of having to eliminate them. 
\end{remark}

We observe that the commutation relations have an obvious
reformulation in terms of Lax pairs. Although we will not use this in
this paper, we believe that the following reformulation might be of an
independent interest, especially in view of a connection with integrability. 
\begin{proposition}[Lax equations]
  \label{prop:Lax}
	The formal power series of operators $A(s):=\sum_{j\geq 0}\frac{s^j}{j!} A_{j+1}$ and $B^{(k)}(s):=\sum_{j\geq 0}\frac{s^j}{j!} B^{(k)}_{j+1}$  each satisfies a Lax equation with respective Lax pairs $(A(s),D_\alpha)$ and $(B^{(k)}(s),\Omega^{(k)}_Y)$. Namely
	$$
	\frac{d}{ds} A(s) = [D_\alpha,A(s)]\ \ \ \mbox{and} \ \ \ 
	\frac{d}{ds} B^{(k)}(s) = [\Omega^{(k)}_Y,B^{(k)}(s)], 
	$$
with solutions
$$
	A(s) = e^{s D_\alpha} A_1 e^{-s D_\alpha}\ \ \ \mbox{and} \ \ \
	B^{(k)}(s) = e^{s \Omega^{(k)}_Y} B^{(k)}_1 e^{-s \Omega^{(k)}_Y}.
$$

\end{proposition}

\subsection{Heuristic: a simple combinatorial proof for $b=0$ or $1$}
\label{subsec:heuristic}

It is tempting to prove the commutation relations of \cref{thm:commut1,thm:commut2} by giving them a combinatorial interpretation. This turns out to be possible for $b=0$ or $1$. In this section we quickly sketch this idea because it is the inspiration for the algebraic proof we design in the next sections that works for all $b$.

\begin{proof}[Sketch of the proof of \cref{thm:commut1} for $b\in \{0,1\}$]
Similarly as in the proof of the decomposition equation and
\cref{prop:interpretGammaPatched}, the operator $A_{j+1}:=\Theta_Y
	\Lambda_{Y}^{j} \frac{y_0}{1+b}$  can be interpreted as follows. First,
create a new isolated vertex of color $0$, considered as the root
vertex and counted by the factor $y_0/(1+b)$. Then create a partial right-path of length $j$ from this vertex, using only existing vertices (operator $\Lambda_{Y}^{j}$). Finally restore the marking of the root face from the $y$ to the $p$ variable (operator $\Theta_Y$). Thus this operator has the effect of creating a root and a partial path of length $j$ from this root, at the level of the $p$ variables.
	
	Moreover, it can also be shown with a bit of work that for any $j$ the operator $D_\alpha$ can be interpreted as adding an edge of color $\{j,j+1\}$ at an arbitrary position in the map. Similarly as in the proof of the decomposition equation, if $j+1<k$ only half of the possible edges are valid choices for this construction, while for $k=j+1$ all of them are.

By composing these operators, the product $D_\alpha A_{j+1}$ has the
effect of adding a new right path of length $j$, and an edge of color
$\{j,j+1\}$ somewhere in the map. Changing the order of the action of
these operators $A_{j+1} D_\alpha$ has the effect of adding an edge of
color $\{j,j+1\}$, and then a new right path. This is almost the same,
except that it does not include the case when the edge is added from the very last corner of the newly created right path, or equivalently when this creates a right path of length $j+1$. We conclude that the commutator $D_\alpha A_{j+1}-A_{j+1} D_\alpha$ has the effect of creating a right-path of length $j+1$, i.e. it is equal to $A_{j+2}$. This is precisely the first commutation relation. 
\end{proof}

\begin{proof}[Sketch of the proof of \cref{thm:commut2} for $b\in \{0,1\}$]
	Similarly as in the proof of \cref{thm:TutteConnected}, the operator $B^{(k)}_{m}=(m-1)! \Theta_Y \big(Y_+\prod_{i=1}^{k}(\Lambda_{Y}+u_i)\big)^m y_0$ can be interpreted as creating a new vertex of colour~$0$ and degree $m$, with an ordering of the edges incident to it, at the level of $p$ variables.
	Moreover it can be shown that $\Omega^{(k)}_Y$ has the effect
        of adding a right-path (of length $k$, possibly using new
        vertices along the way) to a $k$-constellation, while
        $B^{(k)}_1$ has the effect of adding a new vertex of color $0$
        and a right-path starting from it. Therefore, the operators
        $\Omega^{(k)}_Y B^{(k)}_{m}$ and
        $B^{(k)}_{m}\Omega^{(k)}_Y $ both have the effect of adding
        a new vertex of color $0$ and degree $m$, and a new right-path
        in the map, except that the second one does not include the
        case when the new right path is incident to the new vertex. This corresponds precisely to creating a new vertex of degree $m+1$ with an ordering of its edges, i.e. to $B^{(k)}_{m+1}$, which gives the second commutation relation.
      \end{proof}

The sketch of the proof we just gave can be made fully rigorous in the
case $b=0$ or $b=1$ with a bit of work (what remains to be done is the
proof of the fact that the operators $D_\alpha$ and $\Omega^{(k)}_Y$
can be interpreted as we claimed, with appropriate marking
conventions). However, these proofs do \emph{not} work for general
$b$. Indeed, they are based on the idea of constructing the same map
by adding the same edges in different orders, but in the general case there is no reason {\it a priori} that different orders give the same contribution to the $b$-weight. Our whole strategy is designed to overcome this difficulty, see \cref{rem:promotion} below.

\subsection{Proof of the first commutation relations (\cref{thm:commut1})}
\label{subsec:proofCommut1}

The idea of the proof of \cref{thm:commut1} is the following: we
``promote'' the operator $D_\alpha$ to an operator (noted
$D_\alpha+D_\alpha'$) acting on a larger space $\PPP_Y$ such that
$\Theta_Y(D_\alpha+D_\alpha') = D_\alpha \Theta_Y$. This promoted
operator commutes with $\Lambda_{Y}$ and its commutator with
$Y_+$ is given by $Y_+\Lambda_Y$, see
\cref{lemma:relations}. This enables us to perform simple algebraic
manipulations leading to the proof of \cref{thm:commut1} by projecting
the operators acting on $\PPP_Y$ on the subspace
$\PPP$. \cref{rem:promotion} describes the (combinatorial) origin of this proof.

We let $\PPP_Y$ be the set of polynomials in the variables $y_i$ and $p_j$ that are at most linear in the variables $y_i$, that is
\begin{align}
\PPP_{Y} &:= \Span_{\mathbb{Q}(b)}\{p_\lambda, y_i p_\lambda\}_{ i \in \mathbb{N}, \lambda \in \Y}.
\end{align}
Note that $\PPP \subset \PPP_{Y}$ and that differential operators in the variables $p_i$, such as $D_\alpha$, naturally act on $\PPP_Y$. 
We now define the operator $D_\alpha'$ on $\PPP_Y$ by
\begin{align*}
D_\alpha' := 1/2 \left(   (1+b)\sum_{i,j\geq 1} 2y_{i+j} ij \frac{\partial^2 }{\partial y_i \partial p_j }+\sum_{i,j\geq 1} 2i y_{i}p_j \frac{\partial }{ \partial y_{i+j} }+b\sum_{i,j\geq1} i(i-1)y_i\frac{\partial }{\partial y_i} \right).
\end{align*}

\begin{lemma}\label{lemma:relations}
	We have the following commutation relations, as operators on $\PPP_Y$.
	\begin{subequations}
	\begin{align}
		[D_\alpha+D_\alpha',\Lambda_{Y}]&=0 \label{eq:rela1}\\
		[D_\alpha+D_\alpha',Y_+]&=Y_+\Lambda_{Y} \label{eq:rela2}\\
		\Theta_Y (D_\alpha+D_\alpha')&= D_\alpha\Theta_Y. \label{eq:rela3}
	\end{align}
	\end{subequations}
\end{lemma}
\begin{proof}
The proof of these equations presents no difficulty since all
operators have finite order. We refer the interested reader to \cref{sec:AppendixComputations}.
\end{proof}

\begin{proof}[Proof of \cref{thm:commut1}]
	The first two relations of \cref{lemma:relations} imply that
        $[D_\alpha+D_\alpha',Y_+\Lambda_{Y}^j]=Y_+\Lambda_{Y}^{j+1}$
        by induction on $j\geq 0$. Applying $\Theta_Y$ to this
        identity we get:
$$\Theta_Y  (D_\alpha+D_\alpha')Y_+\Lambda_{Y}^j-\Theta_Y Y_+\Lambda_{Y}^j (D_\alpha+D_\alpha') =\Theta_Y Y_+\Lambda_{Y}^{j+1}.$$
Using the third relation of the lemma we obtain
$$D_\alpha \Theta_Y Y_+\Lambda_{Y}^j-\Theta_Y Y_+\Lambda_{Y}^j (D_\alpha+D_\alpha') =\Theta_Y Y_+\Lambda_{Y}^{j+1} .$$
	Now we multiply by $y_0$ on the right, and notice that
        $D_\alpha'y_0$ annihilates the space $\PPP$. 
	Therefore, as operators on $\PPP$ we have
$$D_\alpha \Theta_Y Y_+\Lambda_{Y}^j y_0-\Theta_Y Y_+\Lambda_{Y}^j D_\alpha y_0=\Theta_Y Y_+\Lambda_{Y}^{j+1} y_0.$$ 
Using that $D_\alpha y_0=y_0D_\alpha$ this shows that we have the
following equality between operators on $\PPP$:
$$
	\left[ D_\alpha,\Theta_Y Y_+\Lambda_{Y}^j y_0\right] =\Theta_Y Y_+\Lambda_{Y}^{j+1} y_0.
$$
Since for $j=0$ we have $\Theta_Y Y_+\Lambda_{Y}^{j} y_0= \Theta_Y Y_+
y_0 = p_1$, we obtain \eqref{eq:commut1} by induction on $j$, which
finishes the proof.
\end{proof}

\begin{remark}[Origin of this proof]\label{rem:promotion}
Let us quickly explain the origin of this proof and of the operator $D_\alpha'$.
	The  idea of the combinatorial proof of \cref{subsec:heuristic} is that a given map can be obtained from a smaller one by adding the missing edges in several different orders. It fails in the context of $b$-weights because these different orders may give different contributions. To overcome this, it is natural to look for an ``exchange lemma'' that would say that in fact, the contributions are the same. More precisely, we would need to say that the operation of adding an edge $e$
to a partial right-path in a rooted map $\bM$, and of adding an edge
$f$ of a given colour not incident to this path, ``commute'' with respect to the $b$-weight. For example, for any map $\bM$, one could look for an involution $(e,f)\mapsto (e',f')$ on the set of such pairs of edges that preserves the rooted marking of the final map  and such that $\rho(\bM\cup \{e,f\},e,f) =\rho(\bM\cup \{e',f'\},f',e')$. 

If such a proof exists, it should be represented algebraically by a
simple commutation relation between the operator $\Lambda_{Y}$ and the
operator that adds an edge of a given colour somewhere in the map. This operator needs to be a ``promoted'' version of $D_\alpha$ acting on the space $\PPP_Y$, that needs to take into account the case when the edge $f$ is incident to the root face. This is precisely what the operator $D_\alpha+D_\alpha'$ does.

	In fact, such a combinatorial proof can be given and we found it before the algebraic proof given here. We were able to make it work by using the coherent MON $\rho_{SYM}$, see \cref{rem:rhoSYM}.	
	However writing all the details turns out to be tedious and we
        decided to give only the algebraic proof, leaving this remark for the interested reader.
\end{remark}

\subsection{Proof of the second commutation relations (\cref{thm:commut2})}
\label{subsec:proofCommut2}

In order to prove \cref{thm:commut1}, we promoted our operators from
$\PPP$ to the larger space $\PPP_Y$, where we were able to control
commutation relations.
 In order to prove \cref{thm:commut2} we follow a similar approach, but we need to use the larger space $\PPP_{\tilde{Y},\tilde{Z}}$ defined below.
 The rest of this section is dedicated to this proof.

We introduce three new families of indeterminates $\bm{y'} = (y_i')_{i\geq 0}$ ,$\bm{z}=(z_i)_{i\geq 0}$,
$\bm{z'} = (z_i')_{i\geq 0}$. We let $\PPP_{\tilde{Y},\tilde{Z}}$ be the space of polynomials in $\bm{y},\bm{y'},\bm{z},\bm{z'},\bm{p}$ which are at most linear in each of the families $\bm{y},\bm{y'},\bm{z},\bm{z'}$, and that do not involve simultaneously prime and non-prime variables in these families. Namely:
\begin{align}
\PPP_{\tilde{Y},\tilde{Z}} &:= \Span_{\mathbb{Q}(b)}\{y_i z_j p_\lambda, y'_i z'_j p_\lambda, y_i p_\lambda, z_j p_\lambda, y'_i p_\lambda, z'_j p_\lambda, p_\lambda\}_{i,j \in \mathbb{N}, \lambda \in \Y}.
\end{align}
Clearly
\[ \PPP \subset \PPP_{Y} \subset \PPP_{\tilde{Y},\tilde{Z}}.\]

\begin{remark}[Origin of the space $\PPP_{\tilde{Y},\tilde{Z}}$]\label{rem:promotion2}
	In the spirit of \cref{rem:promotion}, in order to make the
        heuristic proof of \cref{subsec:heuristic} work for the second
        commutation relation, one would need an exchange lemma that
        enables one to add two different right-paths to the same map,
        with different roots, in two different orders, in such a way
        that the contributions to the $b$-weights of both additions
        are the same. But since the construction of right-paths only
        applies to rooted objects, the proof of this lemma, which
        would have to be inductive and work with partial right-paths,
        would need  to keep track of \emph{two} root face degrees. It
        is thus natural to use new variables $z_j$ to mark the size of
        this second root face. However, one should not forget to
        consider the case where, at some point of the construction,
        both roots lie in the \emph{same} face of the map, thus
	splitting it into two intervals (say of length $i$ and $j$). 
	For this case, we use the variables $y'_i z'_j$, hence
        the need of working with the big space
        $\mathcal{P}_{\tilde{Y},\tilde{Z}}$. One needs to promote the
        various operators we consider to this bigger setting, and
        understand their commutators. For example, the operators
        $\Lambda_{\tilde{Y}}$ and $\Lambda_{\tilde{Z}}$ defined below
        are the promoted versions of the operator $\Lambda_{Y}$ (and
        its $\mathbf{z}$-analogue $\Lambda_Z$), and they have the
        effect of extending the first and second partial right path by
        one unit, respectively. The key ``commutation relations''
        between these operators are presented in \cref{lemma:relations3}. 
\end{remark}

We first define variants of the operator $\Lambda_{Y}$ for other alphabets. Since $\La_{Y}$ is acting also on the bigger space $\PPP_{\tilde{Y},\tilde{Z}}$ we can define operators $\La_{Y'},\La_{Z}$ and $\La_{Z'}$ acting on $\PPP_{\tilde{Y},\tilde{Z}}$ by analogy, that is $\La_A$ is obtained from the formula for $\La_{Y}$ by replacing each occurrence of $y_i$ and $\frac{\partial}{\partial y_i}$ by $a_i$ and $\frac{\partial}{\partial a_i}$ for each symbol $a\in\{z,y',z'\}$ of capital symbol $A\in\{Z,Y',Z'\}$. Using the same analogy, we define the operators $\Theta_{Y'},\Theta_{Z}, \Theta_{Z'}$,  $Y'_+,Z_+,Z'_+$, and $\Omega^{(k)}_Z$.
Next, we define 
\begin{align*}
	\La_{Z,Z'}^{Y,Y'} :=   (1\splus{}b)\ssum_{i,j,k\geq 1} \frac{y'_{i\splus{}j\sminus{}1} z'_{k\sminus{}1}\cdot\partial^2 }{\partial y_{i\splus{}k\sminus{}1} \partial z_{j\sminus{}1} }
\splus{}\ssum_{i,j,k\geq 1}\frac{y_{i\splus{}j\sminus{}1} z_{k\sminus{}1}\cdot\partial^2 }{\partial y'_{i\splus{}k\sminus{}1} \partial z'_{j\sminus{}1} }\splus{}b\ssum_{i,j,k\geq 1} \frac{y'_{i\splus{}j\sminus{}1} z'_{k\sminus{}1}\cdot\partial^2 }{\partial y'_{i\splus{}k\sminus{}1} \partial z'_{j\sminus{}1} },
\end{align*}
and its version $\La_{Y,Y'}^{Z,Z'}$ with appropriately exchanged variables:
\begin{align*}
\La_{Y,Y'}^{Z,Z'} :=  (1\splus{}b)\ssum_{i,j,k\geq 1} \frac{z'_{i\splus{}j\sminus{}1} y'_{k\sminus{}1}\cdot\partial^2 }{\partial z_{i\splus{}k\sminus{}1} \partial y_{j\sminus{}1} }
\splus{}\ssum_{i,j,k\geq 1}\frac{z_{i\splus{}j\sminus{}1} y_{k\sminus{}1}\cdot\partial^2 }{\partial z'_{i\splus{}k\sminus{}1} \partial y'_{j\sminus{}1} }\splus{}b\ssum_{i,j,k\geq 1} \frac{z'_{i\splus{}j\sminus{}1} y'_{k\sminus{}1}\cdot\partial^2 }{\partial z'_{i\splus{}k\sminus{}1} \partial y'_{j\sminus{}1} }.
\end{align*}
We also define
\begin{align*}
\Theta_{\tilde{Z}} &:=  \sum_{i\geq 0}p_i\frac{\partial}{\partial z_i}
                     + \sum_{i,j\geq
                     0}y_{i+j}\frac{\partial^2}{\partial y'_i \partial
                     z'_j},
\end{align*}
with the convention that $p_0 = 1$, and 
$\Theta_{\tilde{Y}}$ by analogy.
Finally, we define
\begin{align*}
	\La_{\tilde{Y}} &:=  \La_{Y} + \La_{Y'} + \La_{Y,Y'}^{Z,Z'}\ , \ \ 
\tilde{Y}_+ := Y_+ + Y'_+,\\
	\La_{\tilde{Z}} &:=  \La_{Z} + \La_{Z'} + \La_{Z,Z'}^{Y,Y'}\ , \ \ 
\tilde{Z}_+ := Z_+ + Z'_+.
\end{align*}

The following lemma is the analogue of \cref{lemma:relations} and it easily implies \cref{thm:commut2}.
\begin{lemma}\label{lemma:relations2}
	We have the following relations between operators acting on $\PPP_Y$.
	\begin{align*}
		\left[\Omega^{(k)}_Z + \square^{(k)}, Y_+\prod_{i=1}^{k}(\Lambda_{Y}+u_i)\right]&=\left(Y_+\prod_{i=1}^{k}(\Lambda_{Y}+u_i)\right)^2, \\
		\Theta_Y (\Omega^{(k)}_Z + \square^{(k)}) &= \Omega^{(k)}_Z\Theta_{Y},\\
		y_0\cdot\PPP &\subset \ker\square^{(k)},
	\end{align*}
	where 
	\begin{align*}
		(1+b)\cdot\square^{(k)} &= \Theta_{\tilde{Z}}\tilde{Z}_+\prod_{1 \leq i \leq k}(\La_{\tilde{Z}}+u_i)\La_{\tilde{Z}}z_0 - \Theta_{Z}Z_+\prod_{1 \leq i \leq k}(\La_{Z}+u_i)\La_{Z}z_0 \\
		&= \Theta_{\tilde{Z}}\tilde{Z}_+\prod_{1 \leq i \leq k}(\La_{\tilde{Z}}+u_i)\La_{\tilde{Z}}z_0 - (1+b)\Omega^{(k)}_Z.
	\end{align*}
\end{lemma}

\begin{proof}[Proof of \cref{thm:commut2}]
The first relation of \cref{lemma:relations2} and a direct induction imply that for all $m \geq 1$ one has
\[ [\Omega^{(k)}_Z + \square^{(k)}, \big(Y_+\prod_{i=1}^{k}(\Lambda_{Y}+u_i)\big)^m]=m \big(Y_+\prod_{i=1}^{k}(\Lambda_{Y}+u_i)\big)^{m+1}.\]
Acting by $y_0$ on the right and by $\Theta_Y$ on the left, we obtain
the following identity between operators on $\PPP$:
	\begin{align}\label{eq:interm11}
\Theta_Y[\Omega^{(k)}_Z + \square^{(k)}, \big(Y_+\prod_{i=1}^{k}(\Lambda_{Y}+u_i)\big)^m]y_0=\Theta_Ym \big(Y_+\prod_{i=1}^{k}(\Lambda_{Y}+u_i)\big)^{m+1}y_0.
\end{align}

	Now, we know by \cref{lemma:relations2} that $y_0\cdot\PPP
        \subset \ker\square^{(k)}$. Moreover the operator $[\Omega^{(k)}_Z,y_0] =
        0$ obviously annihilates $\PPP$. Thus we get the following identity between operators on $\PPP$:
\begin{align*} &\Theta_Y
  \big(Y_+\prod_{i=1}^{k}(\Lambda_{Y}+u_i)\big)^m(\Omega^{(k)}_Z +
  \square^{(k)})y_0 = \Theta_Y
  \big(Y_+\prod_{i=1}^{k}(\Lambda_{Y}+u_i)\big)^m\Omega^{(k)}_Z y_0
                      =\\
  &=
  \Theta_Y \big(Y_+\prod_{i=1}^{k}(\Lambda_{Y}+u_i)\big)^m
  y_0\Omega^{(k)}_Z=
  \Theta_Y \big(Y_+\prod_{i=1}^{k}(\Lambda_{Y}+u_i)\big)^m
  y_0\Omega^{(k)}_Y.
  \end{align*}
Moreover, the second relation in \cref{lemma:relations2} gives 
\[\Theta_Y(\Omega^{(k)}_Z + \square^{(k)}) \big(Y_+\prod_{i=1}^{k}(\Lambda_{Y}+u_i)\big)^m y_0 = \Omega^{(k)}_Z\Theta_Y \big(Y_+\prod_{i=1}^{k}(\Lambda_{Y}+u_i)\big)^m y_0.\]
Thus
\[ \Theta_Y[\Omega^{(k)}_Z + \square^{(k)}, \big(Y_+\prod_{i=1}^{k}(\Lambda_{Y}+u_i)\big)^m]y_0 = [\Omega^{(k)}_Y,\Theta_Y \big(Y_+\prod_{i=1}^{k}(\Lambda_{Y}+u_i)\big)^m y_0],\]
	which together with~\eqref{eq:interm11} implies the second relation of~\eqref{eq:commut2} and concludes the proof.
\end{proof}

We prove \cref{lemma:relations2} using an equality between operators acting on $\PPP_{\tilde{Y},\tilde{Z}}$ (\cref{lemma:relations3}) which we then project to $\PPP_Y$. In order to do this, we first need \cref{lemma:relations4}. 
We let 
\[ \Delta := (1+b)\sum_{i,j \geq 0}z'_i y'_j \frac{\partial^2}{\partial y_i\partial z_j} + \sum_{i,j \geq 0}z_i y_j \frac{\partial^2}{\partial y'_i\partial z'_j}+b\sum_{i,j \geq 0}z'_i y'_j \frac{\partial^2}{\partial y'_i\partial z'_j}.\]

\begin{lemma}\label{lemma:relations4}
	We have the following relations between operators acting on $\PPP_{\tilde{Y},\tilde{Z}}$.
	\begin{subequations}
	\begin{align}
		\La_{\tilde{Z}}\Delta &= \Delta \La_{\tilde{Y}}, \label{eq:rel1} \\ 
		 (\La_{\tilde{Z}}+\Delta)\La_{\tilde{Y}}&=\La_{\tilde{Y}}(\La_{\tilde{Z}}+\Delta)
		=(\La_{\tilde{Y}}+\Delta)\La_{\tilde{Z}}=\La_{\tilde{Z}}(\La_{\tilde{Y}}+\Delta)
		\label{eq:rel2bis}\\
		[\La_{\ZZ},\YY_+] &= \YY_+ \Delta, \mbox{ i.e. } \La_{\ZZ}\YY_+ = \YY_+ (\La_{\ZZ} + \Delta) ,\label{eq:rel3} \\
				[\La_{\YY},\ZZ_+] &= \ZZ_+ \Delta, \mbox{ i.e. } \La_{\YY}\ZZ_+ = \ZZ_+ (\La_{\YY} + \Delta) ,\label{eq:rel3bis} \\
		\Theta_{\tilde{Z}}\La_{\tilde{Y}} &=
                                                    \La_{Y}\Theta_{\tilde{Z}}  \mbox{ and }\Theta_{\tilde{Y}}\La_{\tilde{Z}} = \La_{Z}\Theta_{\tilde{Y}}\label{eq:rel4}\\
          		\Theta_{\tilde{Z}}\tilde{Y}_+ &= Y_+\Theta_{\tilde{Z}} \mbox{ and }\Theta_{\tilde{Y}}\tilde{Z}_+= Z_+\Theta_{\tilde{Y}}\label{eq:rel5}.%
	\end{align}
	\end{subequations}
\end{lemma}

\begin{proof}
The proof of these equations presents no conceptual difficulty since all
operators have finite order. We refer the interested reader to \cref{sec:AppendixComputations}.
\end{proof}

\begin{lemma}\label{lemma:relations3}
	For $n,m\geq 0$, we have the following equality between operators acting on $\PPP_{\tilde{Y},\tilde{Z}}$:
	\begin{align}\label{eq:crazySymmetrizedRelation}
		[\tilde{Z}_+\La_{\tilde{Z}}^{m+1},\tilde{Y}_+\La_{\tilde{Y}}^{n}] + [\tilde{Z}_+\La_{\tilde{Z}}^{n+1},\tilde{Y}_+\La_{\tilde{Y}}^{m}] &=\tilde{Y}_+\La_{\tilde{Y}}^n\tilde{Z}_+\La_{\tilde{Z}}^m\Delta + \tilde{Y}_+\La_{\tilde{Y}}^m\tilde{Z}_+\La_{\tilde{Z}}^n\Delta.
	\end{align}
\end{lemma}

\begin{proof}
Without loss of generality we can assume that $m\geq n$, that is $m = n+i$ for $i \geq 0$.
We first claim that it suffices to prove the formula
\begin{align}\label{eq:explicitCommutator}
	[\tilde{Z}_+\La_{\tilde{Z}}^m,\tilde{Y}_+\La_{\tilde{Y}}^{n}] = \sum_{1\leq j \leq i}\YY_+\La_{\YY}^{n+i-j}\ZZ_+\La_{\tilde{Z}}^{n+j-1}\Delta = [\tilde{Y}_+\La_{\tilde{Y}}^m,\tilde{Z}_+\La_{\tilde{Z}}^{n}].
\end{align}
	Indeed, assuming~\eqref{eq:explicitCommutator}, the L.H.S. of~\eqref{eq:crazySymmetrizedRelation} is equal to
\begin{align*} 
	&	[\tilde{Z}_+\La_{\tilde{Z}}^{m+1},\tilde{Y}_+\La_{\tilde{Y}}^{n}]-[\tilde{Z}_+\La_{\tilde{Z}}^{m},\tilde{Y}_+\La_{\tilde{Y}}^{n+1}] \\
	&=\sum_{1\leq j \leq i+1}\YY_+\La_{\YY}^{n+i+1-j}\ZZ_+\La_{\tilde{Z}}^{n+j-1}\Delta - \sum_{1\leq j \leq i-1}\YY_+\La_{\YY}^{n+i-j}\ZZ_+\La_{\tilde{Z}}^{n+j}\Delta 
\\ 
	&=\tilde{Y}_+\La_{\tilde{Y}}^n\tilde{Z}_+\La_{\tilde{Z}}^m\Delta +
\tilde{Y}_+\La_{\tilde{Y}}^m\tilde{Z}_+\La_{\tilde{Z}}^n\Delta,
\end{align*}
	since all  terms in the first sum cancel with the second sum, except $j\in\{1,i+1\}$. This is the desired equality.

	We now prove~\eqref{eq:explicitCommutator} by a repeated use of the relations of \cref{lemma:relations4}.
	We rewrite
        \begin{multline}
          \label{eq:takietam}
	[\tilde{Z}_+\La_{\tilde{Z}}^m,\tilde{Y}_+\La_{\tilde{Y}}^{n}] = \tilde{Z}_+\La_{\tilde{Z}}^{n+i}\tilde{Y}_+\La_{\tilde{Y}}^{n}-\tilde{Y}_+\La_{\tilde{Y}}^{n}\tilde{Z}_+\La_{\tilde{Z}}^{n+i} \\
	=
\tilde{Z}_+\tilde{Y}_+\big((\La_{\tilde{Z}}+\Delta)^{i}-\La_{\tilde{Z}}^i\big)(\La_{\tilde{Z}}+\Delta)^n\La_{\tilde{Y}}^{n},
\end{multline}
	by applying first the relations \eqref{eq:rel3}--\eqref{eq:rel3bis} to move the operators $\YY_+$ and $\ZZ_+$ to the left, and then rearranging with \eqref{eq:rel2bis}.

We now expand $(\La_{\tilde{Z}}+\Delta)^{i}$ according to the position of the leftmost $\Delta$, and we get:
\begin{align*} 
\tilde{Z}_+\tilde{Y}_+\big((\La_{\tilde{Z}}+\Delta)^{i}-\La_{\tilde{Z}}^i\big)(\La_{\tilde{Z}}+\Delta)^n\La_{\tilde{Y}}^{n}
	&=\tilde{Z}_+\tilde{Y}_+ \sum_{1\leq j \leq i}\La_{\tilde{Z}}^{j-1}\Delta
(\La_{\tilde{Z}}+\Delta)^{i-j}(\La_{\tilde{Z}}+\Delta)^n\La_{\tilde{Y}}^{n}
 \\
	&=\tilde{Z}_+\tilde{Y}_+ \sum_{1\leq j \leq i}(\La_{\tilde{Y}}+\Delta)^{n+i-j}\La_{\tilde{Z}}^{n+j-1}
\Delta,
\end{align*}
	where for the second equality we first used the relation~\eqref{eq:rel1} to move the isolated operator $\Delta$ to the right, and then rearrange with~\eqref{eq:rel2bis}. Using the relation \eqref{eq:rel3bis} we can move the operator $\ZZ_+$ inside the sum, and we obtain the first equality in~\eqref{eq:explicitCommutator}. 
	
	If we expand $(\La_{\tilde{Z}}+\Delta)^{i}$
        in~\eqref{eq:takietam} according to the position of the rightmost $\Delta$, we get
	\begin{align*} 
\tilde{Z}_+\tilde{Y}_+ \big((\La_{\tilde{Z}}+\Delta)^{i}-\La_{\tilde{Z}}^i\big)(\La_{\tilde{Z}}+\Delta)^n\La_{\tilde{Y}}^{n}
		&=\tilde{Z}_+\tilde{Y}_+ \sum_{1\leq j \leq i}(\La_{\tilde{Z}}+\Delta)^{j-1}\Delta
\La_{\tilde{Z}}^{i-j}\La_{\tilde{Z}}^n(\La_{\tilde{Y}}+\Delta)^{n} \\
		&=\tilde{Z}_+\tilde{Y}_+ \sum_{1\leq j \leq i}(\La_{\tilde{Z}}+\Delta)^{n+j-1}\La_{\tilde{Y}}^{n+i-j}\Delta,   
	\end{align*}
using the same relations as before. Applying \eqref{eq:rel3} to move
the operator $\YY_+$ yields the second equality in~\eqref{eq:explicitCommutator}.
\end{proof}

\begin{proof}[Proof of~\cref{lemma:relations2}]
We have three statements to prove:

\noindent $\bullet$
We first prove that $y_0\cdot\PPP \subset \ker\square^{(k)}$.
To see this, we replace in the formula
\[ (1+b)\cdot\square^{(k)} = \Theta_{\tilde{Z}}\tilde{Z}_+\prod_{1 \leq i \leq k}(\La_{\tilde{Z}}+u_i)\La_{\tilde{Z}}z_0 - \Theta_{Z}Z_+\prod_{1 \leq i \leq k}(\La_{Z}+u_i)\La_{Z}z_0 \]
 the operator  $\La_{\tilde{Z}}$ by its definition $\La_{Z} + \La_{Z'} + \La_{Z,Z'}^{Y,Y'}$ and we expand the first product. We notice that the monomials in the expansion that involve only the operator $\La_Z$ cancel with the second product, therefore all remaining monomials involve one of the operators  $\La_{Z'}$ or $\La_{Z,Z'}^{Y,Y'}$ at least once. Therefore each term in the expansion involves either a derivative with respect to a prime variable, or a derivative $\frac{\partial}{\partial y_k}$ with $k\geq 1$, and the statement follows.

	\medskip

\noindent $\bullet$
We now prove that
	$[\Omega^{(k)}_Z + \square^{(k)}, \big(Y_+\prod_{i=1}^{k}(\Lambda_{Y}+u_i)\big)]=\big(Y_+\prod_{i=1}^{k}(\Lambda_{Y}+u_i)\big)^2.$ 
Since
	\begin{align}\label{eq:Omega+square}
	\Omega^{(k)}_Z + \square^{(k)} = \Theta_{\tilde{Z}}\tilde{Z}_+\prod_{1 \leq i \leq k}(\La_{\tilde{Z}}+u_i)\La_{\tilde{Z}}\frac{z_0}{1+b},   	\end{align}
	we can rewrite the desired identity (multiplied by $(1+b)$) as
	\begin{align}\label{eq:quadratic}
		[\Theta_{\tilde{Z}}\tilde{Z}_+P\big(\La_{\tilde{Z}}\big)\cdot \La_{\tilde{Z}}z_0, \La_{\tilde{Z}},Y_+P\big(\La_{Y}\big)] &= (1+b)\bigg(Y_+P\big(\La_{Y}\big)\bigg)^2,
	\end{align}
	where $P(x):=\prod_{1 \leq i \leq k}(x+u_i)$.
	To prove this quadratic identity on polynomials, it is sufficient to prove the corresponding symmetrized bilinear identity\footnote{It would of course be sufficient to prove the non-symmetrized version of this bilinear identity, namely that the first terms on each side of~\eqref{eq:commutatorYZ} are equal, but this is not true!}
\begin{multline}
\label{eq:commutatorYZ}
[\Theta_{\tilde{Z}}\tilde{Z}_+\La_{\tilde{Z}}^{m+1}z_0,Y_+\La_{Y}^n] + [\Theta_{\tilde{Z}}\tilde{Z}_+\La_{\tilde{Z}}^{n+1}z_0,Y_+\La_{Y}^m]
= (1+b)\big(Y_+\La_{Y}^m Y_+\La_{Y}^n+Y_+\La_{Y}^n Y_+\La_{Y}^m\big),
\end{multline}
for $m,n\geq 0$.
	Indeed, assuming~\eqref{eq:commutatorYZ} and  writing $P(x)=\sum_{0\leq m \leq k} a_m x^m$, the L.H.S. of~\eqref{eq:quadratic} rewrites
\begin{multline*}
	\sum_{0 \leq m < n\leq k} a_m a_n\Big([\Theta_{\tilde{Z}}\tilde{Z}_+\La_{\tilde{Z}}^{m+1}z_0,Y_+\La_{Y}^n] + [\Theta_{\tilde{Z}}\tilde{Z}_+\La_{\tilde{Z}}^{n+1}z_0,Y_+\La_{Y}^m] \Big) 
        \\
        + \sum_{0 \leq m \leq k} a_m^2[\Theta_{\tilde{Z}}\tilde{Z}_+\La_{\tilde{Z}}^{m+1}z_0,Y_+\La_{Y}^m] 
	= (1+b)\bigg(\sum_{0 \leq m < n\leq k} a_m a_n\Big(Y_+\La_{Y}^m Y_+\La_{Y}^n+Y_+\La_{Y}^n Y_+\La_{Y}^m\Big) 
        \\
       + \sum_{0 \leq m \leq k} a_m^2 Y_+\La_{Y}^m Y_+\La_{Y}^m\bigg) = (1+b)\bigg(Y_+P\big(\La_{Y}\big)\bigg)^2.
\end{multline*}
	Now, by acting with the operators $\Theta_{\tilde{Z}}$ and
        $z_0$, respectively on the left and right of the relation of
        \cref{lemma:relations3} we have the following identity between
        operators acting on $\PPP_Y$:
\begin{equation}
\label{eq:commutatorYZagain}
\Theta_{\tilde{Z}}\Bigg([\tilde{Z}_+\La_{\tilde{Z}}^{m+1},\tilde{Y}_+\La_{\tilde{Y}}^{n}] + [\tilde{Z}_+\La_{\tilde{Z}}^{n+1},\tilde{Y}_+\La_{\tilde{Y}}^{m}]\Bigg)z_0 =\Theta_{\tilde{Z}}\Bigg(\tilde{Y}_+\La_{\tilde{Y}}^n\tilde{Z}_+\La_{\tilde{Z}}^m\Delta + \tilde{Y}_+\La_{\tilde{Y}}^m\tilde{Z}_+\La_{\tilde{Z}}^n\Delta\Bigg)z_0.
\end{equation}
	
	It thus suffices to prove that the left (right, resp.) hand side of \eqref{eq:commutatorYZ} and \eqref{eq:commutatorYZagain} coincide. 
We start with the left hand side. First, we claim that
	\begin{align}\label{eq:halfRelation}
		\Theta_{\tilde{Z}}[\tilde{Z}_+\La_{\tilde{Z}}^{m+1},\tilde{Y}_+\La_{\tilde{Y}}^{n}]z_0 = [\Theta_{\tilde{Z}}\tilde{Z}_+\La_{\tilde{Z}}^{m+1}z_0,Y_+\La_{Y}^n].
	\end{align}
Indeed, 
\[
\Theta_{\tilde{Z}}[\tilde{Z}_+\La_{\tilde{Z}}^{m+1},\tilde{Y}_+\La_{\tilde{Y}}^{n}]z_0 = \Theta_{\tilde{Z}}\tilde{Z}_+\La_{\tilde{Z}}^{m+1}\tilde{Y}_+\La_{\tilde{Y}}^{n}z_0 - \Theta_{\tilde{Z}}\tilde{Y}_+\La_{\tilde{Y}}^{n}\tilde{Z}_+\La_{\tilde{Z}}^{m+1}z_0.\]
Using the fact that $[z_0,\La_{Y}] = 0$ annihilates $\PPP_Y$, and the
fact that $\La_{Y'}$ and $\La_{Y,Y'}^{Z,Z'}$ annihilate $z_0\PPP_Y$ we
substitute $\La_{\tilde{Y}} = \La_{Y}+\La_{Y'}+\La_{Y,Y'}^{Z,Z'}$ and obtain
\[ \Theta_{\tilde{Z}}\tilde{Z}_+\La_{\tilde{Z}}^{m+1}\tilde{Y}_+\La_{\tilde{Y}}^{n}z_0 = \Theta_{\tilde{Z}}\tilde{Z}_+\La_{\tilde{Z}}^{m+1}z_0 Y_+\La_{Y}^{n}.\]
Similarly, using the relations~\eqref{eq:rel4}-\eqref{eq:rel5}, we obtain
\[ \Theta_{\tilde{Z}}\tilde{Y}_+\La_{\tilde{Y}}^{n}\tilde{Z}_+\La_{\tilde{Z}}^{m+1}z_0 = Y_+\La_{Y}^{n}\Theta_{\tilde{Z}}\tilde{Z}_+\La_{\tilde{Z}}^{m+1}z_0,\]
which together with the previously proved equation, implies~\eqref{eq:halfRelation}. 
Since the same equation with $m$ and $n$ exchanged also holds, this proves that the left hand sides of \eqref{eq:commutatorYZ} and \eqref{eq:commutatorYZagain} coincide. 
	
	We now turn to the right-hand sides of \eqref{eq:commutatorYZ} and \eqref{eq:commutatorYZagain}. First, we have
\[ \Theta_{\tilde{Z}}\tilde{Y}_+\La_{\tilde{Y}}^m\tilde{Z}_+\La_{\tilde{Z}}^n\Delta z_0 = Y_+\La_{Y}^m \Theta_{\tilde{Z}}\tilde{Z}_+\La_{\tilde{Z}}^n\Delta z_0\]
by relations~\eqref{eq:rel4}-\eqref{eq:rel5}. Moreover
\begin{multline*} 
	\Theta_{\tilde{Z}}\tilde{Z}_+\La_{\tilde{Z}}^n\Delta z_0 y_i
        p_\lambda =
        (1+b)\Theta_{\tilde{Z}}\tilde{Z}_+\La_{\tilde{Z}}^n y_0'z'_i
        p_\lambda=
        (1+b)\Theta_{\tilde{Z}}y_0'\tilde{Z}_+\La_{\tilde{Z}}^n z'_i
        p_\lambda\\
        =(1+b)\Theta_{\tilde{Z}}y_0'Z'_+\La_{Z'}^n z'_i p_\lambda=
        (1+b) Y_+ \La_{Y}^n y_i p_\lambda
      \end{multline*}
by direct inspection of the definitions of operators. This implies
that the action of
$\Theta_{\tilde{Z}}\tilde{Y}_+\La_{\tilde{Y}}^m\tilde{Z}_+\La_{\tilde{Z}}^n\Delta
z_0$ and $(1+b) Y_+ \La_{Y}^m Y_+ \La_{Y}^n$ on $\PPP_Y$ coincides.
Therefore the right hand sides of \eqref{eq:commutatorYZ} and \eqref{eq:commutatorYZagain} coincide, which finally implies that \eqref{eq:commutatorYZ} holds true. This concludes the proof of the desired identity.

\medskip

\noindent $\bullet$
	It only remains to prove that $\Theta_Y (\Omega^{(k)}_Z + \square^{(k)}) = \Omega^{(k)}_Z\Theta_{Y}$. From~\eqref{eq:Omega+square} and from the relations~\eqref{eq:rel4}-\eqref{eq:rel5}, we directly obtain $\Theta_{\tilde Y} (\Omega^{(k)}_Z + \square^{(k)}) = \Omega^{(k)}_Z\Theta_{\tilde{Y}}$. But $\Theta_{\tilde{Y}}$  and $\Theta_{Y}$ have the same action on $\PPP_Y$, therefore we get that $\Theta_{Y} (\Omega^{(k)}_Z + \square^{(k)}) = \Omega^{(k)}_Z\Theta_{Y}$, as operators on $\PPP_Y$, which is the desired relation.
	\end{proof}

\section{$b$-deformation of the tau-function}
\label{sec:Jack}

In this section we study the $b$-deformed tau-function $\MultiJack$, defined in~\eqref{eq:JackIntro} using Jack symmetric functions. We show that it is annihilated by the operators defined in the previous section, which makes the connection with the generating function of coverings $F_\rho$ and prove our main result, \cref{thm:mainInSection5}.

\subsection{Jack symmetric functions}

\subsubsection{Partitions and symmetric functions}
The group
$\Sym{\infty}$ of permutations of $\mathbb{N}_{\geq 1}$ with a finite support
acts naturally on the set of 
sequences of nonnegative integers with finite support $\A = \bigoplus_{i = 1}^{\infty}\N$, and
partitions represent orbits of this action. We can rephrase this
observation as follows. Let $\Symm_n :=
\QQ[x_1,\dots,x_n]^{\Sym{n}}$ be the algebra of symmetric polynomials,
that is polynomials in $x_1,\dots,x_n$ invariant by the natural action
of $\Sym{n}$ permuting their variables. Let $\Symm :=
\underleftarrow{\Symm_n}$ be the projective limit with respect to the
natural morphism $\Symm_{n+1} \ni f(x_1,\dots,x_n,x_{n+1}) \mapsto
f(x_1,\dots,x_n,0)$. The algebra of symmetric functions $\Symm$ has a natural homogenous basis
indexed by partitions and obtained by symmetrizing monomials:
\[m_\la = \sum_{\alpha \in \Sym{\infty} \la} \xx^\alpha,\]
where $\Sym{\infty} \la$ is the orbit of the partition $\la$ by the
action of the permutation group $\Sym{\infty}$ on $\A$, and
$\xx^\alpha$ is the monomial $\xx^\alpha = \prod_i
x_i^{\alpha_i}$. In particular $\Symm = \bigoplus \Symm^n$ has a natural gradation by
degree, and $\Symm^n$ is a finite-dimensional vector space,
whose dimension is given by the number of partitions of size
$n$.

There is another base of $\Symm^n$ of a great importance in this
paper, which is called \emph{power-sum} basis, and is given by 
    \[ p_\lambda = \prod_{i=1}^{\ell(\lambda)} p_i; \quad p_i = \sum_j x_j^i
      \text{ for } i>0.\]
An immediate consequence of this fact is that $\Symm$ is a polynomial
algebra, $\Symm = \QQ[p_1,p_2,\dots]$.

\subsubsection{Laplace-Beltrami operator and Jack symmetric function}

In order to define Jack symmetric functions, which are the main
characters of this section, we need to introduce some simple
combinatorial statistics of partitions.

We let $\PPP_n$ denote the set of partitions of size $n$. There is an important poset structure on $\PPP_n$
given by the \emph{dominance order:}
\[ \lambda \leq \mu \iff \sum_{i\leq j}\lambda_i \leq \sum_{i\leq
    j}\mu_i \text{ for any positive integer } j.\]
To any partition $\lambda \in \PPP_n$ we can associate
a conjugate partition $\la^t = (\la^t_1,\dots,\la^t_{\ell'})$, where
$\ell' = \la_1$, and for any $1 \leq i \leq \ell'$
\[ \la^t_i = \#\{j: \la_j \geq i\}.\]
The concept of conjugate partition is very natural from a geometric
point of view. Indeed we can represent a partition $\la$ by drawing its
\emph{Young diagram}, which consists of the set 
\[ \la = \{(i, j):1 \leq i \leq \la_j, 1 \leq j \leq \ell(\la) \}\]
and then conjugating $\la$ corresponds to reflecting its Young diagram
through the line $x=y$.  For any
 box $\square := (i,j) \in \la$ from Young diagram we define its
 \emph{arm-length} by $a(\square) := \la_j-i$ and its
 \emph{leg-length} by $\ell(\square) := \la_i^t-j$. These definitions follow~\cite[Chapter I]{Macdonald1995}.

Let $\alpha = 1+b$ be an indeterminate. There are
several natural statistics on the set of partitions (or,
equivalently, Young diagrams) that we need:
\begin{align}
\label{eq:HookProduct}
\hook_\a(\la) &:= \prod_{\square \in \la}\left(\a\ a(\square) + \ell(\square) +1 \right),\\
\label{eq:HookProduct2}
\hook'_\a(\la) &:= \prod_{\square \in \la}\left(\a\ a(\square) +
                 \ell(\square) +\a \right),\\
\label{eq:NumericalFactor}
z_\la &:= \prod_{i \geq 1}i^{m_i(\la)}m_i(\la)!,
\end{align}
where $m_i(\la)$ denotes the number of parts of $\la$ equal to $i$
(therefore $n!z_\lambda^{-1}$ is the number of permutations from the
conjugacy class of type $\lambda$). We
also recall that for a box $\square = (x,y) \in \la$ its
content is equal to $x-y = (x-1) - (y-1)$. We define its
$\a$-deformation by
\[c_\a(\square) := \a(x-1)-(y-1).\]

Let $\Symm_\alpha$ denote the algebra of symmetric functions over the
field $\QQ(\alpha)$ of rational functions in $\alpha$ with rational
coefficients. Since $\Symm_\alpha = \QQ(\a)[p_1,p_2,\dots]$ then clearly
the Laplace-Beltrami operator $D_\a$ given by
\eqref{eq:Laplace-Beltrami} acts on the symmetric function
algebra. Its importance is reflected in the following result.

\begin{defprop}
  \label{defprop:Jack}
There is a unique family of symmetric functions
$\{J^{(\alpha)}_\lambda\}$ such that for each partition $\lambda$,
\begin{itemize}
\item $D_\alpha J^{(\alpha)}_\lambda = \left(\sum_{\square \in \la}c_\a(\square)\right)J^{(\alpha)}_\lambda$;
\item
$ J_\la^{(\a)} = \hook_\a(\la) m_\la + \sum_{\nu < \la}a^{\la}_\nu m_\nu,  \text{ where } a^{\la}_\nu \in \QQ(\a).$
  \end{itemize}
  We call them \emph{Jack symmetric functions}.
\end{defprop}

\begin{remark}
  Jack symmetric functions
  are usually defined by three conditions: orthogonality,
	normalization, and triangularity (which is the second property in our definition). However,
  \cref{defprop:Jack} is the core of the proof that the classical definition
  makes sense. Therefore we are
  going to treat \cref{defprop:Jack} as a definition of Jack symmetric functions in
  this paper and we refer to \cite{Stanley1989,Macdonald1995} for completeness.
\end{remark}

We can endow $\Symm_\alpha$ with a scalar product by defining it on
the basis of power-sums
\begin{equation}
  \label{eq:product}
  \langle p_\mu,p_\nu\rangle_\a = \a^{\ell(\mu)}z_\mu\delta_{\mu,\nu},
\end{equation}
where $\delta_{\mu,\nu}$ is the Kronecker delta. It turns out that Jack
symmetric functions are also orthogonal with the following squared norm:
\[\langle J_\la^{(\a)},J_\la^{(\a)}\rangle_\a = \hook_\a(\la)
  \hook'_\a(\la) =: j_\la^{(\alpha)}.\]
Note that this is a one-parameter deformation of the factor
$(\frac{f_\la}{n!})^2$ appearing in the definition~\eqref{eq:SchurIntro}
of the classical
tau-function. Indeed, for  $\alpha = 1$, $\hook_1(\la) = \hook'_1(\la)$
coincides with the classical hook-length appearing in the hook-length formula for $f_\lambda$ (see e.g.~\cite{Stanley:EC2}), thus \[ \frac{n!^2}{f_\la^2} = j_\la^{(1)}.\]

For any linear operator $D \in \End(\Symm_\alpha)$ we can
define its adjoint $D^\perp$ with respect to
$\langle\cdot,\cdot\rangle_\alpha$ so that
\[\langle Df,g\rangle_\alpha = \langle
f,D^\perp g\rangle_\alpha\] for all symmetric functions $f,g \in
\Symm_\alpha$.
For instance
\begin{equation}  
  \left(p_j/\a\right)^\perp = j\partial/\partial p_j,
\end{equation}
which is a direct consequence of \eqref{eq:product}.

From now on, we will use the notation $J_\la^{(\a)}(\pp)$ to indicate
that we are treating Jack polynomials as polynomials in the ``power-sum'' variables
$p_1,p_2,\dots$, considered as indeterminates.

\medskip

We are now going to show that the operators $G_j := j!\partial/\partial p_j$
are determined by a similar recursion as the operators $A_j$ from
\cref{thm:commut1}:

\begin{lemma}\label{lem:commutSimple}
Define the differential operators $(G_j)_{j\geq 1}$ on $\PPP$ by:
\begin{align}
  \label{eq:GjCombi}
 G_{j}:= j!\partial/\partial p_j \ \ j\geq 1.
\end{align}
Then these operators satisfy the recurrence formula
\begin{align}\label{eq:commutSimple}
	G_1 &= \partial/\partial p_1 \ \ \ ,\ \ G_{j+1} = [G_j,A_2^\perp], \mbox { , for } j\geq 1.
\end{align}
\end{lemma}

\begin{proof}
  The proof is made by induction on $j$ and it is an easy
  computation. Note that
  \[ A_2^\perp = \left(\Theta_Y Y_+\Lambda_{Y} \frac{y_0}{\alpha}\right)^\perp = \left(\sum_{i
      \geq 1} p_{i+1}\cdot \left(p_{i}/\a\right)^\perp\right)^\perp
	= \sum_{i
        \geq 1} p_{i} \cdot \left(p_{i+1}/\a\right)^\perp = \sum_{i
        \geq 1} \left(p_{i+1}/\a\right)^\perp\cdot p_{i}.\]
    Since
    \[j![    \partial/\partial p_j, \left(p_{i+1}/\a\right)^\perp\cdot p_{i}] =
    j!\delta_{i,j}\left(p_{i+1}/\a\right)^\perp\]
  we obtain 
  \[ G_{j+1} = j! \left(p_{j+1}/\a\right)^\perp = j! \sum_{i \geq 1}[
    \partial/\partial p_j, \left(p_{i+1}/\a\right)^\perp\cdot
    p_{i}] = [G_j,A_2^\perp],\]
  and we conclude the proof.
\end{proof}

Stanley obtained in his seminal paper \cite{Stanley1989} some results concerning Jack symmetric functions, which are of special
interest for us. These results can be seen as $\a$--deformations of
a classical product formula and a special case of
Pieri rule for Schur polynomials. 
Moreover, for two partitions $\la,\mu \in \PPP$ we write $\la
\nearrow \mu$ if $|\mu|-|\la| = 1$ and the Young diagram  of $\la$ is
contained in the one of $\mu$.

\begin{theorem}[\cite{Stanley1989}]
  For any $\lambda \in \PPP_n$ one has
  \begin{align}
    \label{eq:JackEvaluationProdForm}
     J_\la^{(\a)}(\underline{u}) &= \prod_{\square \in
                         \la}(u+c_\a(\square)),\\
    \label{eq:JackPieriRule}
    p_1 J_\la^{(\a)}(\pp) &= \sum_{\la \nearrow \mu}c_{\la \nearrow \mu}
                       J_\mu^{(\a)}(\pp),
    \end{align}
where $c_{\la \nearrow \mu} \in \N[\a]$ is a (explicit) polynomial in
$\a$ with nonnegative integer coefficients.

\end{theorem}

\begin{corollary}
  \label{cor:ActionOfEOnJack}
  For any $i \geq 1$ the following identity holds true
  \[ A_i J_\la^{(\a)}(\pp) = \sum_{\la \nearrow \mu}c_\a(\mu\setminus\la)^{i-1}\frac{c_{\la \nearrow \mu}}{\a}
                       J_\mu^{(\a)}(\pp).\]
                   \end{corollary}

                   \begin{proof}
                     We use induction on $i$. For
                     $i=1$ one has $A_1 = p_1/\a$, so this is simply
                     \eqref{eq:JackPieriRule}.
                     We recall that
			   \begin{align}\label{eq:jackEigenvector}
				   D_\a  J_\la^{(\a)}(\pp) = \left(\sum_{\square \in \la} c_{\a}(\square)\right)  J_\la^{(\a)}(\pp).
			   \end{align}
                     Thus
                     \begin{multline*}
                       A_{i+1} J_\la^{(\a)}(\pp) =
                         [D_\a,A_i]J_\la^{(\a)}(\pp)  = \sum_{\la \nearrow \mu}\left(\sum_{\square \in\mu}c_\a(\square)- \sum_{\square \in\la}c_\a(\square)\right)c_\a(\mu\setminus\la)^{i-1}\frac{c_{\la \nearrow \mu}}{\a}
                         J_\mu^{(\a)}(\pp)\\
                         =\sum_{\la \nearrow
                           \mu}c_\a(\mu\setminus\la)^i \frac{c_{\la \nearrow \mu}}{\a}
                       J_\mu^{(\a)}(\pp). \ \ \  \qedhere
                       \end{multline*}
                     \end{proof}

\subsection{The $b$-deformation of the tau-function}

In this section we prove our main theorem.
In the proof, the differential operators defined in previous sections with respect to the variables $\pp$ will also be used with respect to the variables $\qq$, and for this we introduce the following more precise notation. We denote by $A_i(\pp)$, $B^{(k)}_i(\pp)$, $G_i(\pp)$, respectively,
 the operators defined by \eqref{eq:EjCombi},
\eqref{eq:FmCombi} and \eqref{eq:GjCombi}. We
denote by $A_i(\qq)$, $B^{(k)}_i(\qq)$, $G_i(\qq$), respectively,
the operators obtained from $A_i(\pp)$ $B^{(k)}_i(\pp)$, $G_i(\pp)$
 by replacing each occurence of the indeterminate $p_i$ in
their definition by the indeterminate $q_i$ for each $i>0$.
Moreover we recall that, everywhere, $\a = 1+b$.

Define $\MultiJack(t;\pp,\qq,u_1,\dots,u_k) \in
\QQ(b)[\pp,\qq,u_1,\dots,u_k][[t]]$ by
\begin{align}
  \label{eq:MultiJackDef}
  \MultiJack(t;\pp,\qq,u_1,\dots,u_k) := \sum_{n \geq 0}t^n\sum_{\la
    \vdash n}\frac{J_\la^{(1+b)}(\pp)J_\la^{(1+b)}(\qq)J_\la^{(1+b)}(\underline{u_1})\cdots J_\la^{(1+b)}(\underline{u_k})}{j_\la^{(1+b)}},
\end{align}
with the convention that the constant term in $t$ is equal to $1$.

        \begin{lemma}
      \label{lem:MultiJack1}
      The generating series $\MultiJack$ satisfies the following
      equation:
      \begin{equation}
        \label{eq:VirasoroJack1}
        G_1(\qq) \MultiJack(t;\pp,\qq,u_1,\dots,u_k) =
          tB^{(k)}_1(\pp) \MultiJack(t;\pp,\qq,u_1,\dots,u_k).
          \end{equation}
      \end{lemma}

\begin{proof}
        We fix an integer $n \geq 0$ and we look at the coefficient of $t^{n+1}$ in the %
L.H.S.
        of \cref{eq:VirasoroJack1}, which is given by the formula:
        \begin{multline*}
          \sum_{\la
    \vdash
    n+1}\frac{J_\la^{(1+b)}(\pp)\JJ_\la^{(1+b)}(u_1,\dots,u_k)}{j_\la^{(1+b)}}G_1(\qq)
J_\la^{(1+b)}(\qq) = \\
  \sum_{\la
    \vdash
    n+1}\frac{J_\la^{(1+b)}(\pp)\JJ_\la^{(1+b)}(u_1,\dots,u_k)}{j_\la^{(1+b)}}\frac{\partial}{\partial
    q_1}   J_\la^{(1+b)}(\qq),
\end{multline*}
where
\[\JJ_\la^{(1+b)}(u_1,\dots,u_k) := J_\la^{(1+b)}(\underline{u_1})\cdots J_\la^{(1+b)}(\underline{u_k}).\]
Applying the Pieri rule \eqref{eq:JackPieriRule} we obtain the following
expression
\[ \sum_{\la
    \vdash
    n+1} J_\la^{(1+b)}(\pp)\JJ_\la^{(1+b)}(u_1,\dots,u_k)\sum_{\mu \nearrow \la}\frac{c_{\mu \nearrow \la}}{1+b}\frac{J_\mu^{(1+b)}(\qq)}{j_\mu^{(1+b)}}.\]
It is straightforward from \eqref{eq:JackEvaluationProdForm} that for
any $\mu \nearrow \la$ we have
	\begin{align}\label{eq:contentForm}
		J_\la^{(1+b)}(\underline{u}) =
\big(u+c_{\a}(\la\setminus\mu)\big)\cdot
J_\mu^{(1+b)}(\underline{u}),
	\end{align}
which gives the following identity:
\[ \JJ_\la^{(1+b)}(u_1,\dots,u_k) = \big(\sum_{1 \leq i \leq
    k+1}e_{k+1-i}(u_1,\dots,u_k) c_{\a}(\la\setminus\mu)^{i-1}\big)
  \JJ_\mu^{(1+b)}(u_1,\dots,u_k).\]
Plugging it into the expression of the L.H.S. of \cref{eq:VirasoroJack1}
and changing the order of summation we obtain the following formula:
\[ \sum_{\mu
    \vdash
    n} \frac{J_\mu^{(1+b)}(\qq)\JJ_\mu^{(1+b)}(u_1,\dots,u_k)}{j_\mu^{(1+b)}}\big(\sum_{1 \leq i \leq
    k+1}e_{k+1-i}(u_1,\dots,u_k) \sum_{\mu \nearrow
    \la}c_{\a}(\la\setminus\mu)^{i-1}\frac{c_{\mu \nearrow \la}}{1+b} J_\la^{(1+b)}(\pp)\big).\]
Finally, \cref{cor:ActionOfEOnJack} implies that this expression is
equal to
\[ \sum_{\mu
    \vdash
    n}
  \frac{J_\mu^{(1+b)}(\qq)\JJ_\mu^{(1+b)}(u_1,\dots,u_k)}{j_\mu^{(1+b)}}B^{(k)}_1(\pp)J_\mu^{(1+b)}
  = [t^{n+1}]\RHS,\]
which leads to the desired identity:
\[ G_1(\qq)\MultiJack(t;\pp,\qq,u_1,\dots,u_k) =
  tB^{(k)}_{1}(\pp) \MultiJack(t;\pp,\qq,u_1,\dots,u_k). \qedhere\]
\end{proof}

\begin{lemma}
      \label{lem:MultiJack2}
      The generating series $\MultiJack$ satisfies the following
      equation:
      \begin{equation}
        \label{eq:VirasoroJack2}
        A_2^\perp(\qq) \MultiJack(t;\pp,\qq,u_1,\dots,u_k) =
          t\Omega_Y^{(k)}(\pp) \MultiJack(t;\pp,\qq,u_1,\dots,u_k).
          \end{equation}
      \end{lemma}

\begin{proof}
The proof is very similar to the proof of the previous lemma. We fix an integer $n \geq 0$ and we look at the coefficient of $t^{n+1}$ in the L.H.S.
        of \cref{eq:VirasoroJack2}:
        \begin{multline*}
          \sum_{\la
    \vdash
    n+1}\frac{J_\la^{(1+b)}(\pp)\JJ_\la^{(1+b)}(u_1,\dots,u_k)}{j_\la^{(1+b)}} A_2^\perp(\qq)
  J_\la^{(1+b)}(\qq) \\
  = \sum_{\la
    \vdash
    n+1}J_\la^{(1+b)}(\pp)\JJ_\la^{(1+b)}(u_1,\dots,u_k)
  \sum_{\mu \nearrow \la} c_\a(\la\setminus\mu)\frac{c_{\mu \nearrow \la}}{1+b}\frac{J_\mu^{(1+b)}(\qq)}{j_\mu^{(1+b)}},
\end{multline*}
by \cref{cor:ActionOfEOnJack}. Applying the same substitutions as in the
proof of \cref{lem:MultiJack2} we transform the coefficient of
$t^{n+1}$ on the L.H.S. of \cref{eq:VirasoroJack2} to the following form:
\[ \sum_{\mu
    \vdash
    n} \frac{J_\mu^{(1+b)}(\qq)\JJ_\mu^{(1+b)}(u_1,\dots,u_k)}{j_\mu^{(1+b)}}\big(\sum_{1 \leq i \leq
    k+1}e_{k+1-i}(u_1,\dots,u_k) \sum_{\mu \nearrow
    \la}c_{\a}(\la\setminus\mu)^{i}\frac{c_{\mu \nearrow \la}}{1+b} J_\la^{(1+b)}(\pp)\big).\]
Finally, \cref{cor:ActionOfEOnJack} gives that this expression is
equal to
\[ \sum_{\mu
    \vdash
    n}
  \frac{J_\mu^{(1+b)}(\qq)\JJ_\mu^{(1+b)}(u_1,\dots,u_k)}{j_\mu^{(1+b)}}\Omega_Y^{(k)}(\pp)
  J_\mu^{(1+b)}(\pp)
  = [t^{n+1}]\RHS,\]
which finishes the proof.
\end{proof}

These two lemmas composed together give us the following equations, 
which are a keystone of this paper.

    \begin{lemma}
      \label{lem:MultiJack}
      For any $m \geq 1$ the generating series $\MultiJack$ satisfies the following
      equation:
      \begin{equation}
        \label{eq:VirasoroJack}
        G_m(\qq) \MultiJack(t;\pp,\qq,u_1,\dots,u_k) =
          t^mB_m^{(k)}(\pp) \MultiJack(t;\pp,\qq,u_1,\dots,u_k).
          \end{equation}
      \end{lemma}

      \begin{proof}
        We use induction on $m$. For $m=1$ the
        statement is precisely 
        \cref{lem:MultiJack1}. We fix a nonnegative integer $l >1$ and
        suppose that \cref{eq:VirasoroJack}
        holds true for any $m<l$. Then 
      \[  G_l(\qq) \MultiJack(t;\pp,\qq,u_1,\dots,u_k)
        = [G_{l-1}(\qq), A_2^\perp(\qq)]\MultiJack(t;\pp,\qq,u_1,\dots,u_k),\]
      which is equal to
      \begin{align*}
        \big(&G_{l-1}(\qq)\cdot\Omega_Y^{(k)}(\pp) t-
        A_2^\perp(\qq) \cdot B^{(k)}_{l-1}(\pp)
                                                t^{l-1}\big)\MultiJack(t;\pp,\qq,u_1,\dots,u_k) 
        = \big(\Omega_Y^{(k)}(\pp)\cdot G_{l-1}(\qq) t-\\
     -   &B^{(k)}_{l-1}(\pp)\cdot A_2^\perp(\qq) 
                                                t^{l-1}\big)\MultiJack(t;\pp,\qq,u_1,\dots,u_k)
                                                  = [\Omega_Y^{(k)}(\pp),
        B^{(k)}_{l-1}(\pp)]t^l\MultiJack(t;\pp,\qq,u_1,\dots,u_k)
        \end{align*}
      by \cref{lem:MultiJack2}, induction hypothesis, and the fact
      that operators $A_2^\perp(\qq), G_{l-1}(\qq)$ commute with $\Omega_Y^{(k)}(\pp), B^{(k)}_{l-1}(\pp)$. This
      finishes the proof since
      \[ [\Omega_Y^{(k)}(\pp),
        B^{(k)}_{l-1}(\pp)] = B^{(k)}_{l}(\pp). \qedhere\]
    \end{proof}

\subsection{Proof of the main results}

We are now ready to make the connection between $k$-constellations and the function $\tau^{(k)}_b$.
For $m\geq 1$ we introduce the functions
$$
U_m := (1+b) m \frac{\partial}{\partial q_m} \log \tau^{(k)}_b \ \ 
, \ \ V_m =(1+b) m \frac{\partial}{\partial p_m} \log \tau^{(k)}_b.
$$
Recall that $\pi$, defined in Section~\ref{sec:MonTutte} is
  the operator exchanging the sets of variables
  $\pp\leftrightarrow\qq$ and $u_i\leftrightarrow u_{k+1-i}$ for $1
  \leq i \leq k$. 
\begin{proposition}\label{prop:mixDiffEqs}
	For $m\geq 1$ one has:
	\begin{align*}
		U_m =t^m  \cdot
		\Theta_Y \Big(Y_+\prod_{l=1}^{k}\big(\Lambda_{Y}+u_l + \sum_{i,j\geq 1}V_i  y_{j+i-1}
		 \frac{\partial}{\partial y_{j-1}}
		\big)\Big)^{m}  (y_0),
	\end{align*}
	and moreover $V_m = \pi U_m$. 
\end{proposition}
\begin{proof}
	From the previous lemma and from the definition~\eqref{eq:FmCombi} of $B_m^{(k)}$, we have for $m\geq 1$,
	\begin{align}\label{eq:patch1}
	m\frac{\partial}{\partial q_m} \tau^{(k)}_b =
	t^m\cdot \Theta_Y \big(Y_+\prod_{l=1}^{k}(\Lambda_{Y}+u_l)\big)^m \frac{y_0}{1+b} \tau^{(k)}_b.
	\end{align}
	Now, for any series $A(\yy,\pp)$ depending on variables $\yy$ and $\pp$ (and possibly other parameters), we have by definition of the operator $\Lambda_Y$:
	\begin{align*}%
		(\Lambda_{Y}+u_l) A(\yy,\pp) \tau^{(k)}_b =
		\Big((\Lambda_{Y}+u_l+ 
		\sum_{i,j\geq 1}  V_i y_{j+i-1}
		\frac{\partial}{\partial y_{j-1}} )A(\yy,\pp) \Big) \tau^{(k)}_b,
	\end{align*}
	where we used that $(1+b) i\frac{\partial}{\partial p_{i}}\tau^{(k)}_b = V_i \tau^{(k)}_b$.
	Applying this identity repeated times to~\eqref{eq:patch1}, and using also $m\frac{\partial}{\partial q_m} \tau^{(k)}_b=\frac{1}{1+b} U_m \tau^{(k)}_b$ , we get
	\begin{align*}
		\frac{1}{1+b} U_m =
		t^m\cdot \Theta_Y \big(Y_+\prod_{l=1}^{k}
		\big(\Lambda_{Y}+u_l+ 
		\sum_{i,j\geq 1}  V_i y_{j+i-1}
		\frac{\partial}{\partial y_{j-1}} \big)
		\Big)^m \left(\frac{y_0}{1+b}\right),
	\end{align*}
	which is the desired identity upon multiplying by $(1+b)$.

	The fact that $V_m = \pi U_m$ is clear since $\tau^{(k)}_b = \pi \tau^{(k)}_b$.
\end{proof}

From Corollary~\ref{cor:TutteThetaConnected} we immediately deduce
\begin{corollary}\label{cor:equalInSection5}
Let $\rho$ be a coherent MON. 
Then for any $m\geq 1$, one has $U_m = H^{[m]}$, where $H^{[m]}$ is the generating function of constellations defined by~\eqref{eq:DefHRhoM}.
\end{corollary}
Recall that for each positive integer $n$ the function $[t^n]\tau^{(k)}_b$ treated as a polynomial in
$\qq$ is homogenous of degree $n$ (where $\deg(q_i) := i$), therefore
$\sum_{m \geq 1} m q_m \frac{\partial}{\partial q_m}$ and $ t
\frac{\partial}{\partial t}$ act similarly on $\tau^{(k)}_b$. In
particular, using definition of $U_m$, \cref{cor:equalInSection5}, and identity \eqref{eq:HintoHm}
we obtain the following.
\begin{theorem}\label{thm:mainInSection5}
Let $\rho$ be a coherent MON. Then one has
$$
	(1+b) t \frac{\partial}{\partial t} \ln \tau^{(k)}_b = \Theta_Y \vec H_\rho.
$$
In particular, if $\rho$ is integral then we have
	\begin{align}\label{eq:mainInSection5}
(1+b) t \frac{\partial}{\partial t} \ln \tau^{(k)}_b =
	\sum_{n \geq 1}\sum_{(\bM,c)} t^{n} b^{\nu_\rho(\bM,c)} \kappa(\bM),
	\end{align}
	where the second sum is taken over rooted connected $k$-constellations $(\bM,c)$ of size $n$.
\end{theorem}
Note that, up to the correspondence between constellations and
generalized branched coverings, the last statement is precisely our main result stated in the introduction (Theorem~\ref{thm:abbreviated}).

  \section{Weighted Hurwitz numbers and projective limits}
\label{sec:infinite}

In previous sections we have considered $k$-constellations for an arbitrary fixed $k$, which correspond to coverings with $k+2$ ramification points. However in the literature concerning the orientable case, much interest has been given to cases where the number of ramification points is unbounded, which is the case for (weighted) Hurwitz numbers.
In this section we explain how to obtain this case as a limit of the previous one by allowing $k$ to become, in some sense, infinite.

\begin{quote}
	\centering {\it In this section, a coherent MON $\rho$ is fixed.}
\end{quote}

\subsection{Infinite constellations and projective limits}

We let $\mathcal{C}^{(k)}$ be the set of $k$-constellations, and $\mathcal{C}^{(k)}_n$ the subset formed by the ones of size $n$. 
We equip these objects with the \emph{modified marking} 
$$
	\hat \kappa(\bM) := \prod_{f\in F(\bM)} p_{\deg(f)}\prod_{v\in V_0(\bM)} q_{\deg(v)} \prod_{i=1}^k u_i^{n-v_i(\bM)},
$$
which differs from the marking~\eqref{eq:marking} used in the previous
section only by the exponent of $u_i$. We let $\hat F_k$ denote the corresponding
``modified'' generating function, given by the substitution:
\begin{align}\label{eq:hatFk}
	\hat F^{(k)}(t,\mathbf{p},\mathbf{q},u_1,\dots,u_k) :=F_\rho(t
  \cdot u_1\cdots u_k,\mathbf{p},\mathbf{q},u_1^{-1},\dots,u_k^{-1}).
\end{align}
We define the modified marking in the same way for rooted objects.
In the rest of \cref{sec:infinite} we will work with modified markings. To avoid confusion, all the new generating functions we define in this section will be also denoted with a hat symbol`$~\widehat{~}~$'.

\begin{remark}
	In previous sections, the parameter $k$ was fixed, and it could be deduced from the notation for the function $F_\rho$ only by its number of arguments. Since in this section it is crucial to let $k$ vary, we indicate it explicitly in the notation $\hat F^{(k)}$. On the other hand, now that we have proved that our generating functions do not depend on the choice of a coherent MON, we drop the dependency in $\rho$ from the notation.
	The two sides of Equation~\eqref{eq:hatFk} reflect these differences of notation.
\end{remark}

We observe that the addition of a leaf of colour $(k+1)$ to every
corner of colour $k$ gives an inclusion 
\begin{align}\label{eq:inclusion}
	\mathcal{C}_n^{(k)} \hookrightarrow \mathcal{C}_n^{(k+1)}.
\end{align}
This inclusion preserves the modified marking, since the number of
corners of color $k$ and of color $0$ are equal.
The same inclusion holds at the level of rooted objects and it preserves the rooted modified marking.
Moreover, by an appropriate choice of the integral coherent MON $\rho$, it is possible\footnote{Indeed, the leaves are only spectators in the combinatorial decomposition. So it suffices to choose a MON $\rho$ such that $\rho(\bM,e)$ depends only on the position of $e$ in the 2-core (or skeleton) of $\bM$, which is the map obtained from $\bM$ by successively removing all leaves.} to define the $b$-weights $\tilde\rho(\bM)$, $\vec\rho(\bM,c)$, and $b^{\nu_\rho(\bM,c)}$ in a way that respects the inclusion~\eqref{eq:inclusion}.

This implies the following equation (that can also be deduced from~\eqref{eq:TutteUnconnected})
\begin{align}\label{eq:projectiveCondition}
	\hat F^{(k)}(t,\mathbf{p},\mathbf{q},u_1,\dots,u_k) =\hat F^{(k+1)}(t,\mathbf{p},\mathbf{q},u_1,\dots,u_k,0).
\end{align}

These inclusions enable us to define the projective limit 
$$\mathcal{C}_n^{(\infty)}:=\underleftarrow{\mathcal{C}_n^{(k)}}.$$ 
Elements of $\mathcal{C}^{(\infty)} := \bigcup \mathcal{C}_n^{(\infty)}$ can be viewed as ``constellations
with arbitrarily many colors'',
and for short will be called \emph{infinite constellations} in what follows. They are equipped with a well-defined (modified) marking and $b$-weight.
Their generating function
	$$
	\hat F^{(\infty)}(t,\mathbf{p},\mathbf{q},u_1,\dots ):=
	1+\sum_{n \geq 1}\sum_{\bM \in \mathcal{C}^{(\infty)}_n} \frac{t^{n}}{2^{n-cc(\bM)}n!} \frac{\tilde{\rho}(\bM)}{(1+b)^{cc(\bM)}} \hat \kappa (\bM)
	$$
is given by the projective limit
$$
\hat F^{(\infty)}(t,\mathbf{p},\mathbf{q},u_1,\dots ) = \underleftarrow{\hat F^{(k)}}(t,\mathbf{p},\mathbf{q},u_1,\dots,u_k).
$$

\begin{remark}[Infinite constellations are generalized branched coverings]
	The interpretation of infinite constellations in terms of
        generalized branched coverings is straightforward. Elements of
        $\mathcal{C}^{(\infty)}_n$ correspond to generalized branched
        coverings of the sphere $\mathcal{S}\rightarrow \Sphere_+$ of
        degree~$n$, whose number of non-trivial ramification points is
	(finite but) not bounded {\it a priori}. The first ramification point (allowed to be trivial) is numbered $-1$ and corresponds to faces of the constellation. The other (nontrivial) ramification points are numbered with some (non necessarily consecutive) nonnegative integers, corresponding to the integers $i$ such that $n-v_i(\bM)$ is nonzero in the constellation.
\end{remark}

We denote the set of \emph{rooted} infinite constellations of size $n$ by
$\mathcal{C}^{(\infty)}_{n,\bullet}$. We have the following theorem.
\begin{theorem}[Main results reformulated for infinite constellations]\label{thm:infinite}
	The generating function $\hat F^{(\infty)} \equiv \hat F^{(\infty)} (t;\pp,\qq,u_1,\dots)$ of infinite constellations satisfies the equation
\begin{equation}\label{eq:TutteThetaInfinite}
	\frac{t\partial}{\partial t} \hat F^{(\infty)} = \Theta_Y \sum_{m \geq 1}t^m \cdot q_m \cdot
	\left(Y_+  \prod_{i=1}^\infty (1+u_i \Lambda_{Y})\right)^m  \frac{y_0}{1+b} \hat F^{(\infty)}.
\end{equation}
It is equal to the Jack polynomial expansion
\begin{equation}\label{eq:jackInfinite}
	\hat F^{(\infty)}(t,\mathbf{p},\mathbf{q},u_1,\dots ) = \sum_{n \geq 0}t^n
	\sum_{\la
	\vdash n}\frac{J_\la^{(1+b)}(\pp)J_\la^{(1+b)}(\qq) \prod_{i\geq 1} u_i^{|\lambda|}J_\la^{(1+b)}(\underline{1/u_i})  }{j_\la^{(1+b)}}.
\end{equation}
	Moreover $(1+b) \frac{ t \partial} {\partial t} \log \hat  F^{(\infty)}$ has coefficients which are polynomials in $b$ with non negative integer coefficients, explicitly given by
\begin{equation}\label{eq:jackInfiniteConnected}
(1+b) \frac{ t \partial} {\partial t} \log
	\hat F^{(\infty)} = 
	\sum_{n \geq 1}\sum_{(\bM,c) \in \mathcal{C}^{(\infty)}_{n,\bullet}} t^{n} \hat \kappa(\bM) b^{\nu_\rho(\bM,c)}.
\end{equation}
\end{theorem}
\begin{proof}
	This is a direct consequence of the results of previous sections.
\end{proof}

We recall that $J_\la^{(1+b)}$ is a polynomial in the power-sum variables of homogeneous graded degree $|\lambda|$, normalized in such a way that the coefficient of $p_1^{|\la|}$ is $1$.  Therefore $u_i^{|\lambda|}J_\la^{(1+b)}(\underline{1/u_i})$ is a polynomial in $u_i$ with constant term $1$. Thus, despite the infinite product in its definition, the expression~\eqref{eq:jackInfinite} is a well defined formal power series in $t$ and the $u_i$.
In fact, from~\eqref{eq:contentForm}, we have the explicit formula 
$u^{|\la|}J_\la^{(1+b)}(\underline{1/u}) = \prod_{\square \in \la}\big(1+u c_{\a}(\square)\big)$, where the product is taken over all boxes of $\la$, and we can write
\begin{align}\label{eq:contentProductForm}
\hat F^{(\infty)}(t,\mathbf{p},\mathbf{q},u_1,\dots ) = \sum_{n \geq 0}t^n
	\sum_{\la
	\vdash n}\frac{J_\la^{(1+b)}(\pp)J_\la^{(1+b)}(\qq) \prod_{i\geq 1}  \prod_{\square \in \la}\big(1+u_i c_{\a}(\square)\big)  }{j_\la^{(1+b)}}.
\end{align}

We also observe that the case of finite $k$ considered in previous
sections can be recovered from the infinite case by considering the
case where $u_i$ is equal to zero for all $i>k$. 

\subsection{Weighted $b$-Hurwitz numbers}

We now introduce $b$-weighted analogues of the weighted Hurwitz numbers of~\cite{Guay-PaquetHarnad2017}. In order to avoid the terminology ``$b$-weighted weighted Hurwitz numbers'', we will use ``weighted $b$-Hurwitz numbers'' instead.

We note that the parameters $u_i$ appear in the equations of \cref{thm:infinite} only through the generating function $z \mapsto \prod_{i=1}^\infty (1+u_i z)$. Therefore it can be interesting to use the coefficients of this power series as parameters, rather than the $u_i$ themselves. Let 
$$
G(z) = 1+\sum_{k\geq 1} g_k z^k
$$
be a formal power series, where the $g_k$ are indeterminates. 
Define the generating function $\hat F^G \equiv \hat F^G(t,\mathbf{p},\mathbf{q},g_1,\dots)$ by the following variant of Equation~\eqref{eq:TutteThetaInfinite}
\begin{equation}\label{eq:TutteThetaInfiniteG}
	\frac{t\partial}{\partial t} \hat F^{G}  = \Theta_Y \sum_{m \geq 1}t^m \cdot q_m \cdot
	\Big(Y_+  G\big( \Lambda_{Y}\big)\Big)^m  \frac{y_0}{1+b} \hat F^{G}.
\end{equation}

When the variables $(g_k)_{k\geq 1}$ and $(u_i)_{i\geq 1}$ are related by the equation $G(z)=\prod_{i=1}^\infty (1+z u_i)$, then we have
\begin{align}\label{eq:F=F}
	\hat F^G(t,\mathbf{p},\mathbf{q},g_1,\dots)=\hat F^{(\infty)}(t,\mathbf{p},\mathbf{q},u_1,\dots).
\end{align}
Therefore in this case $\hat F^G$ can be seen as the generating function of infinite constellations, with weights $u_i$ related to $g_k$ by $$g_k =  e_k (u_1,\dots)$$ where $e_k$ denote the elementary symmetric functions. 
Call an infinite constellation $\bM$ of size~$n$ \emph{normal} if the sequence 
$$\mathfrak{v}(\bM):=(n-v_i(\bM))_{i\geq 1}$$ is nonincreasing, i.e. if it is an integer partition. Then $\hat F^{(\infty)}$ can be viewed as the generating function of normal infinite constellations with their usual marking in the $\mathbf{p}$-$\mathbf{q}$-$t$ variables, and a marking  $m_{\mathfrak{v}(\bM)}(u_1,u_2,\dots)$ giving the contribution of vertices of color $i\geq 1$, where $m_{\cdot}$ denotes monomial symmetric functions (see e.g.~\cite{Stanley:EC2}). Indeed, the summation hidden in the definition of monomial symmetric functions consists in summing over all possible reorderings of the sequence $\mathfrak{v}(\bM)$, accounting also for constellations which are not normal.  This observation enables us to give an interpretation of $\hat F^{G}$ in terms of the indeterminates $g_k$ with no reference to underlying variables $u_i$, as follows.
\begin{theorem}\label{thm:weighted}
	Let $G(z)=1+\sum_{k\geq 1}g_k z^k$. Then the series $\hat F^G$ is the generating function of normal infinite constellations with their $b$-weight, and with a marking $p_i$ (resp $q_i$) per face (resp. vertex of colour $0$) of degree $i$, and a marking $f_{\mathfrak{v}(M)}(g_1,g_2,\dots)$ where  for an integer partition $\iota$, $f_{\iota}$ is the polynomial that expresses the monomial symmetric function of index $\iota$ in terms of the elementary symmetric functions. 
	At the level of rooted connected objects, we have
	$$
(1+b)\frac{t\partial}{\partial t} \log	\hat F^G =
		\sum_{n \geq 1}\sum_{(\bM,c) \in \mathcal{C}^{(\infty)}_{n,\bullet}\atop \bM~\mbox{\tiny normal}} 
		t^{n} b^{\nu_\rho(\bM,c)} \prod_{f\in F(\bM)} p_{\deg(f)}\prod_{v\in V_0(\bM)} q_{\deg(v)}\ f_{\mathfrak{v}(\bM)}(g_1,\dots). 
	$$
\end{theorem}

The polynomials $f_\iota$ can be constructed as follows. For an integer partition $\mu$ of size $n$, we have $e_\mu = \sum_{\iota \vdash n} R_{\mu,\iota} m_\iota$ where the square matrix $R$ has its rows and columns indexed by integer partitions of size $n$, and $R_{\mu,\iota}$ is equal to the number of $0$-$1$ matrices with row-sum (resp. column-sum) equal to $\mu$ (resp. $ \iota$), see~\cite{Stanley:EC2}. Then  
$f_\iota(g_1,\dots)  = \sum_{\mu \vdash n} (R^{-1})_{\iota,\mu} \prod_{i=1}^{\ell(\mu)} g_{\mu_i}$, where $R^{-1}$ denotes the inverse matrix of $R$. Since $R$ is triangular for the dominance order with a diagonal of $1$, the matrix $R^{-1}$ has integer coefficients, that we will not try to make explicit here.

\begin{remark}\label{rem:genus}
	An extra
        parameter\footnote{The parameter $\hbar$ is noted $\beta$
          in~\cite{AlexandrovChapuyEynardHarnad2020,Okounkov2000a} but
          we prefer to avoid the confusion with $\beta$-ensembles.}
        $\hbar$ can be added to the series $\hat F^G$ (or $\hat
        F_{\infty}$) via the scaling $g_k \mapsto \hbar ^k g_k$ (or
        $u_i\mapsto \hbar u_i$). The exponent $d$ of $\hbar$ is directly related to the Euler
	characteristic of the underlying constellation (or covering) via $\chi = \ell(\lambda)+\ell(\mu)-d$, where $\lambda$ and $\mu$ represent respectively the degrees of faces and of vertices of colour $0$ (or the profiles of the first two ramification points). In particular, the $\hbar$-expansion of $\log \hat F^G$ is a topological expansion, with powers $\hbar^{\ell(\lambda)+\ell(\mu)-\chi}$ appearing in the coefficient of $p_\lambda q_\mu$ for each $\lambda, \mu \vdash n$. We will not need this viewpoint here in general, but we will introduce the variable $\hbar$ in the special cases of the next sections. 
\end{remark}

\begin{remark}
	In the case $b=0$, the interpretation of the series $\hat F^G$ that we gave, in terms of normal constellations, is \emph{not} the standard one. Indeed, in that case it is possible to give an interpretation in terms of factorizations of permutations using \emph{transpositions} (which at the level of coverings correspond to simple branchpoints), with a weighting system that can be made explicit in terms of the $g_k$, which gives rise to the so-called \emph{weighted Hurwitz numbers}, see~\cite{Guay-PaquetHarnad2017}.
	The connection between the two interpretations goes through the group algebra of the symmetric group and the Jucys-Murphy elements, which are specific to $b=0$. 	We leave as an open problem to give a similar interpretation of the series $\hat F^G$ in full generality. In the next section, we will address, however, the case of $b$-weighted analogues of \emph{classical} Hurwitz numbers.

\end{remark}

We observe that the function $\hat F^G$ has the following expression
\begin{align}\label{eq:GProductForm}
\hat F^{G}(t,\mathbf{p},\mathbf{q},g_1,\dots ) = \sum_{n \geq 0}t^n
	\sum_{\la
	\vdash n}\frac{J_\la^{(1+b)}(\pp)J_\la^{(1+b)}(\qq) \prod_{\square \in \la} G(c_{\a}(\square))  }{j_\la^{(1+b)}}.
\end{align}
Indeed, when $g_k=e_k(u_1,\dots)$ this is a consequence of \eqref{eq:contentProductForm} and~\eqref{eq:F=F}. But the fact that elementary symmetric functions are a basis of the space of symmetric functions implies that this is true with the $g_k$ being independent indeterminates.

\subsection{$b$-Hurwitz numbers, $G(z)=\exp(\hbar z)$.}
\label{subsec:bHurwitz}

\begin{definition}
	For $n, \ell \geq 1$ and $\lambda, \mu$ two partitions of $n$, the (connected) $b$-Hurwitz number $H^\ell(\lambda,\mu)(b)$ is defined as the
	polynomial
	$$
	H^\ell(\lambda,\mu)(b):= \sum_{(\bM,c)} b^{\nu_\rho(\bM,c)},
	$$
	where the sum is taken over rooted connected $\ell$-constellations $\bM$ of size $n$ with full profile 
	\begin{align}\label{eq:profileHurwitz}
		(\lambda, \mu, \underbrace{[2,1^{n-2}], \dots, [2,1^{n-2}]}_{\ell \mbox{ times}}).
	\end{align}
	Equivalently, this sum is taken over rooted generalized branched coverings  of degree $n$ of the sphere  by a connected surface, with  $\ell+2$  numbered ramification points, the first two with profiles $\lambda$, $\mu$, and all the other ones being simple.
\end{definition}

Since each rooted connected constellation of size $n$ has $(n-1)$
non-root corners, we note that
$\frac{(n-1)!}{n!}H^\ell(\lambda,\mu)(0) = \frac{1}{n}H^\ell(\lambda,\mu)(0)$ is nothing but the usual (orientable) double Hurwitz number of parameters $\lambda, \mu, \ell$. For $\mu=[1^n]$, we recover single Hurwitz numbers. Also note that $H^\ell(\lambda,\mu)(b)-H^\ell(\lambda,\mu)(0)$ gives the contribution of coverings from non-orientable surfaces. In both cases, the Euler characteristic of the underlying surface can be recovered by~\eqref{eq:EulerCharacteristic}, namely	$$
	\chi = \ell(\mu)+\ell(\lambda)-\ell.
	$$

We now form the corresponding generating function for non-necessarily connected objects,
$$
\hat F_{\mathrm{Hurwitz}}\equiv \hat F_{\mathrm{Hurwitz}}
(t,\mathbf{p},\mathbf{q}, \hbar) := 1 + \sum_{n\geq 1} \frac{t^n}{2^{n}n!} \sum_{\lambda, \mu \vdash n\atop \ell \geq 0 } \sum_{\bM}\frac{\tilde\rho(\bM)}{2^{-cc(\bM)}(1+b)^{cc(\bM)}} \frac{\hbar^\ell}{\ell!} p_\lambda q_\mu,
$$
where $\hbar$ is a new indeterminate and where the last sum is taken over \emph{labeled}, connected or not, $\ell$-constellations $\bM$ of size $n$ with full profile~\eqref{eq:profileHurwitz}. 
We have
\begin{theorem}\label{thm:FHurwitz}
	The generating function $\hat F_{\mathrm{Hurwitz}}$ is equal to the function $\hat F^G$ with $G(z) = \exp_\hbar (z):= \exp \hbar z$, {\it i.e.}
	\begin{align}\label{eq:bHurwitzGen}
	\hat F_{\mathrm{Hurwitz}}  = \hat F^{\exp_\hbar}.
	\end{align}
	The function $(1+b)\frac{t \partial}{\partial t} \log \hat F_{\mathrm{Hurwitz}}$ is the generating function of rooted connected coverings, explicitly given by
	\begin{align}\label{eq:bHurwitzGenConn}
	(1+b)\frac{t \partial}{\partial t} \log \hat F_{\mathrm{Hurwitz}}
		=\sum_{n\geq 1} t^n \sum_{\lambda, \mu \vdash n\atop \ell \geq 0 } H^\ell(\lambda,\mu)(b) \frac{\hbar^\ell}{\ell!} p_\lambda q_\mu
		.
	\end{align}

\end{theorem}
\begin{theorem}[$b$-Cut and Join equation]\label{thm:cutAndJoin}
The series $\hat F_{\mathrm{Hurwitz}} \equiv \hat F_{\mathrm{Hurwitz}} (t,\mathbf{p},\mathbf{q}, \hbar)$ satisfies the equation
	\begin{align}\label{eq:cutAndJoin}
	\frac{\partial}{\partial \hbar}\hat F_{\mathrm{Hurwitz}} = D_\alpha \hat F_{\mathrm{Hurwitz}},
	\end{align}
	where $\alpha=1+b$ and $D_\alpha$ is the Laplace-Beltrami operator~\eqref{eq:Laplace-Beltrami}.
\end{theorem}
We observe that the function $\hat F^{\exp_\hbar}$ is well defined as a formal power series, since the $\hbar$-grading makes the substitution $g_k \mapsto \frac{\hbar^{k}}{k!}$ well defined.
\begin{proof}[Proof of \cref{thm:FHurwitz}]
	We will obtain the generating function of Hurwitz numbers as a suitable limit of $\hat F^{(k)}$ for $k\rightarrow \infty$. This idea is inspired from the orientable case where a similar argument is classical~\cite{BousquetSchaeffer2000,GouldenJackson2008}.

	Let us study the expansion of $\hat F^{(k)}\equiv \hat F^{(k)} (t,\mathbf{p}, \mathbf{q},u_1,\dots,u_k)$ for $u_1=\dots=u_k = \frac{\hbar}{k}$, by grouping monomials $u_1^{n_1}\dots u_k^{n_k}$ according to the number $p$ of nonzero exponents.
	It will be convenient to use (only) in this proof the
        following notation: $[\dots]$ denotes the coefficient extraction with respect to the $u_i$-variables only, as if the variables $t, \hbar, p_i, q_j$ were constants.
	We have,
	\begin{align}
	\left. \hat F^{(k)} \right|_{u_i = \frac{\hbar}{k} }
		&= \sum_{\ell \geq 0} \frac{\hbar^\ell}{k^\ell} \sum_{p\geq 0}  \sum_{n_{1}+\dots+n_{k}=\ell, n_{i} \geq 0 \atop |\{i,n_i > 0\}|=p}
	[u_{1}^{n_1}\dots u_k^{n_k}] \hat F^{(k)}
		\nonumber
		\\ 
		&=\sum_{\ell \geq 0} \frac{\hbar^\ell}{k^\ell}
		\sum_{0\leq p \leq \ell}  \sum_{n_{1}+\dots+n_{p}=\ell\atop n_{i} > 0 }
		{k \choose p} [u_{1}^{n_1}\dots u_p^{n_p}] \hat F^{(k)},
		\label{eq:int23}
	\end{align}
where the second inequality uses the fact that coefficients of $\hat F^{(k)}$ are symmetric functions in the $u_i$.
	We now consider the limit of~\eqref{eq:int23} when $k$ goes to infinity.
	First remark that for fixed $\ell$ and $p\leq \ell$, we have when $k$ goes to infinity:
$$
	\frac{1}{k^\ell} {k \choose p} \longrightarrow \mathbf{1}_{p=\ell} \frac{1}{\ell!}.
$$
	Therefore, for each fixed coefficient in $\hbar$, when $k$ goes to infinity, only the term $p=\ell$ contributes to~\eqref{eq:int23} at the first order. Further, for $p=\ell$ only the term $n_1=\dots=n_\ell=1$ is possible in~\eqref{eq:int23}. We thus get the (coefficientwise) limit
$$
\lim_k \left. \hat F^{(k)} \right|_{u_i = \frac{\hbar}{k} }
		= \sum_{\ell\geq 0}  
	\frac{\hbar^\ell}{\ell!} [u_{1}^{1}\dots u_\ell^{1}] \hat F^{(\infty)},
$$
which, by definition, is equal to $\hat F_{\mathrm{Hurwitz}}$.

Now, when $k$ goes to infinity, we also have the limit: 
$$\prod_{i=1}^k (1+ z u_i) \big|_{u_i = \frac{\hbar}{k}}  = \left(1+\frac{z \hbar}{k}\right)^k\longrightarrow \exp (\hbar z).$$
	Because the series $\hat F^{(k)}$ satisfies the decomposition equation~\eqref{eq:TutteThetaInfiniteG} with $G(z)=\prod_{i=1}^k (1+u_i z)$, and because all coefficients (in $t$) of the series $\hat F^{(k)}$ are symmetric polynomials in the $u_i$, this implies that the series $\lim_k \left. \hat F^{(k)} \right|_{u_i = \frac{\hbar}{k} }$ satisfies the decomposition equation~\eqref{eq:TutteThetaInfiniteG} with $G(z)=\exp(\hbar z)$. Therefore $\lim_k \left. \hat F^{(k)} \right|_{u_i = \frac{\hbar}{k} } = \hat F^{\exp_\hbar}$ and the proof is finished.
\end{proof}
\begin{proof}[Proof of \cref{thm:cutAndJoin}]
	This is a direct consequence of~\eqref{eq:GProductForm} with $G=\exp_\hbar$ and of the fact~\eqref{eq:jackEigenvector} that $J_\la^{(\alpha)}(\mathbf{p})$ is an eigenvector of $D_\alpha$ with eigenvalue $\sum_{\square \in \la} c_\alpha(\square)$.
      \end{proof}

      We conclude this subsection by proving piecewise
      polynomiality of the $b$-weighted double Hurwitz numbers. 
      In the case $b=0$ our result  is slightly weaker than the 
      result of Goulden, Jackson and Vakil~\cite{GouldenJacksonVakil2005}
     which states that the classical double Hurwitz number 
      $\frac{1}{|\la|}H^\ell(\lambda,\mu)(0)$ is a piecewise polynomial.
      The extra factor of $|\la|$ comes from the fact that we need to average over the choice of the root to define the $b$-weight, and we do not know if the stronger result holds for general $b$.
      \begin{theorem}[Piecewise polynomiality]
        Let us fix $\ell, m, n$ and consider $H^\ell(\lambda,\mu)(b)$
        as a function on partitions $\lambda,\mu$ whose number of
        parts is equal to $m$ and $n$ respectively. Then
        $H^\ell(\lambda,\mu)(b)$ is a polynomial in $b$, whose
        coefficients are piecewise polynomial functions in $\lambda_1,\dots,\la_m,\mu_1,\dots,\mu_n$.
      \end{theorem}

	    \begin{proof}
Our proof is inspired by the combinatorial proof of Goulden, Jackson, and Vakil for the case $b=0$, but instead of working directly at the level of combinatorial objects it works inductively at the level of generating functions, using the decomposition equation.

        We will prove the result by induction on $\ell$.
        Let $\hat{H}^{[d]}_\rho$ and $\hat{\widetilde{H}}^{[d]}_\rho$
        denote the projective limit (over $\ell$) of $H^{[d]}_\rho(\frac{1}{u_1\cdots
                                   u_\ell};\pp,\qq,u_1^{-1},\dots,u_\ell^{-1})$
                                 and $\widetilde
            {H}_\rho^{[d]}(\frac{1}{u_1\cdots u_\ell};\pp,\qq,
            u_1^{-1},\dots,u_\ell^{-1})$, respectively. We define
            $\hat{\vec H}_\rho$ similarly. Note that for any function
		    $H \in\{ \hat{H}^{[d]}_\rho, \hat{\widetilde{H}}^{[d]}_\rho,
		    \hat{\vec H}_\rho\}$ and for any sequence $i_1<\dots <i_\ell$
            one has $[u_{i_1}\cdots u_{i_\ell}]H = [u_{1}\cdots
            u_{\ell}]H = [e_i(\uu)]H$ since $H$ is symmetric in
            $\uu$. In particular
        \begin{align*}
          H^\ell(\lambda,\mu)(b) &= [u_1\cdots u_\ell]\Theta_Y \hat{\vec
          H}_\rho= \sum_{d \geq
            1}[u_1\cdots
          u_\ell]q_d\cdot \hat{H}_\rho^{[d]}.
            \end{align*}
        Let us define $h_\ell: \N^{m+n}\to \N[b]$ as the coefficient
		    $$h_\ell(\lambda_1,\dots,\lambda_m,\mu_1,\dots,\mu_{n-1},d) = 
		    [p_{\lambda_1}\dots p_{\lambda_m} q_{\mu_1}\dots q_{\mu_{n-1}}] [u_1\cdots u_\ell] \hat{H}_\rho^{[d]}.$$ 
        We will prove by induction on $\ell$ that for any $n$ and $m$,  $h_\ell: \N^{m+n}\to
        \N[b]$ is polynomial in $b$ with coefficients which are
        piecewise polynomials, which implies the result we want. 

        \cref{cor:TutteThetaConnected} implies that
		    \[\hat{H}_\rho^{[d]} = t^d p_d +\text{(monomials involving
		    variables $\uu$)},\]
        therefore 
		    $h_0(\lambda_1,\dots,\lambda_m,\mu_1,\dots,\mu_{n-1},d)$ is equal to $1$ if $d=\lambda_i\geq 1$ for some $i$ and all other variables are equal to zero, and it is equal to zero otherwise. This is a piecewise
      polynomial function. %

Moreover \cref{cor:TutteThetaConnected} implies that
\begin{equation}
  \label{eq:pomocnicze}
  [u_1\cdots u_\ell]\hat{H}_\rho^{[d]}=
		\Theta_Y [u_1\cdots
                u_\ell]\Big(Y_+\prod_{i}\big(1+u_i(\Lambda^{[1]}_{Y}+\Lambda^{[2]}_{Y}+\Lambda^{[3]}_{Y}
                + \R)\big)\Big)^{d}  (y_0),
                \end{equation}
                where
\begin{align*}
	\Lambda^{[1]}_{Y} &=        (1+b)\sum_{i,j\geq 1}
                     y_{j+i-1}\frac{i\partial}{\partial
		     p_{i}}\frac{\partial}{\partial y_{j-1}},                       &\Lambda^{[2]}_{Y} =        \sum_{i,j\geq
                                          1}
                                          y_{j-1}p_i\frac{\partial}{\partial
                                          y_{i+j-1}},\\
                                           \Lambda^{[3]}_{Y} &=
                                                               b\sum_{i\geq
                                                               1}
                                                               y_{i}\frac{i\partial}{\partial
                                                               y_{i}},
							      &\R = \sum_{i,j\geq 1} y_{j+i-1}\hat{\widetilde
            {H}}_\rho^{[i]}\frac{\partial}{\partial y_{j-1}}.
                                                               \end{align*}
		    For $\ell'\geq 0$ let $\R_{\ell'}:=[u_1\cdots u_{\ell'}]\R  = [e_{\ell'}(\uu)]\R$. Expanding the
            R.H.S. of~\eqref{eq:pomocnicze} we can express it as the
            following linear combination:
            \begin{align*}
              \Theta_Y\sum_{s \geq 1}\sum_{I_1,\dots,I_d \subset
              [1..\ell]}\Bigg[\prod_{i \in
              \left(\bigcup_{i=1}^d I_i\right)^c}u_i\Bigg]\prod_{i=1}^dY_+(\Lambda^{[1]}_{Y}+\Lambda^{[2]}_{Y}+\Lambda^{[3]}_{Y}+\R)^{|I_i|}(y_0),
            \end{align*}
            where we sum over all pairwise disjoint (possibly empty)
            subsets of $[1..\ell]$, whose union is non-empty.
            Expanding the product we find a linear combination of quantities of the form
            \begin{equation}
              \label{eq:kombinacja}
              \Theta_YY_+^{k_1}w_1\cdots Y_+^{k_s}w_s(y_0),
              \end{equation}
            for some $1 \leq s \leq \ell$, and positive integers $k_1,\dots,k_s$
            whose sum is equal to $d$, where
            \begin{itemize}
              \item $w_k$ is a non-empty word in
                $\Lambda^{[1]}_{Y},\Lambda^{[2]}_{Y},\Lambda^{[3]}_{Y},\R_{0},\dots,\R_{\ell-1}$,
                \item the total length of words satisfies
                  $\ell(w_1)+\cdots+\ell(w_s) \leq \ell$,
                  \item the total sum of indices of the variables
                    $\R_{0},\dots,\R_{\ell-1}$ appearing in $w_1,\dots,w_s$ is equal to $\ell-\sum_{i=1}^s\ell(w_i)$.
                  \end{itemize}
                  For a fixed sequence of words $w_1,\dots,w_s$, the
                  element of the form~\eqref{eq:kombinacja} appears in the
                  R.H.S. of~\eqref{eq:pomocnicze} with 
                  coefficient
                  $\binom{\ell}{\ell(w_1),\dots,\ell(w_s)}$, which does not depend on
                  $k_1,\dots,k_s$.
		  Therefore, since the number of choices for $s$ and  $(w_i)_{1\leq i \leq s}$ is finite, it is enough to show that for each such choice, the quantity		  
		  \begin{align}\label{eq:fixedWord}
			  [p_{\lambda_1} p_{\lambda_2}\dots p_{\lambda_m} q_{\mu_1}q_{\mu_2}\dots q_{\mu_{n-1}}] \sum_{k_1+\dots+k_s=d} \Theta_YY_+^{k_1}w_1\cdots Y_+^{k_s}w_s(y_0)
		  \end{align}
		  is a piecewise polynomial in $(\lambda_1,\dots,\lambda_m,\mu_1,\dots, \mu_{n-1},d)$.

To see this, we notice that to compute the wanted coefficient  in~\eqref{eq:fixedWord}, one needs to sum over all ``trajectories'' of the monomials in $p_i,q_j, y_k$ appearing from right to left along the product of operators. Such a trajectory consists in a tuple  of monomials
$$(p_{\lambda^{(0)}}q_{\mu^{(0)}}y_{i_{0}},p_{\lambda^{(1)}}q_{\mu^{(1)}}y_{i_{1}}, \dots, p_{\lambda^{(2s)}}q_{\mu^{(2s)}}y_{i_{2s}}),$$
where $p_{\lambda^{(2r)}}q_{\mu^{(2r)}}y_{i_{2r}}$ (resp. $p_{\lambda^{(2r-1)}}q_{\mu^{(2r-1)}}y_{i_{2r-1}}$) is the monomial appearing to the right (resp. left) of the operator $w_r$ for $1\leq r \leq s$, with 
$p_{\lambda^{(2s)}}q_{\mu^{(2s)}}y_{i_{2s}}=y_0$ and
$p_{\lambda^{(0)}}q_{\mu^{(0)}}p_{i_{0}}=p_{\lambda_1} p_{\lambda_2}\dots p_{\lambda_m} q_{\mu_1}q_{\mu_2}\dots q_{\mu_{n-1}}.
$

Note that the nature of the operators $\Lambda^{[i]}_{Y}$ and $\R_{i}$ implies that the possible values of the indices of successive monomials appearing along this tuple are constrained by a finite number of linear equalities and inequalities. In other words, the set of valid trajectories is parametrized by tuples of integers
\begin{align}\label{eq:tuples}
(\lambda^{(i)}_j, 1\leq j \leq \ell(\lambda^{(i)});
\mu^{(i)}_j, 1\leq j \leq \ell(\mu^{(i)});
 i_j;
	k_j)_{1\leq j \leq s}
\end{align}
subject to linear constraints, {\textit i.e.} by integer points in a polytope. Moreover, note that for fixed $\ell, n, m$, the maximum number of parts appearing in any monomial is bounded, so this polytope is finite-dimensional. We include the equality
$$
k_1+\dots +k_s =d,
$$
which involves the parameter $d$, in the linear constraints defining this polytope.

Finally, each $r \in [1..s]$ such that $w_r$ is equal to $\R_{\ell'}$ for some $\ell'\in [0..\ell-1]$ corresponds to the fact that one passes from the monomial $p_{\lambda^{(2r)}}q_{\mu^{(2r)}}y_{i_{2r}}$ to $p_{\lambda^{(2r-1)}}q_{\mu^{(2r-1)}}y_{i_{2r-1}}$ by using the operator 
$$\R_{\ell'}=[u_1\cdots u_{\ell'}]  y_{j+i-1}\hat{\widetilde
            {H}}_\rho^{[i]}\frac{\partial}{\partial y_{j-1}}. 
$$
Therefore, assuming that the trajectory is valid, the coefficient 
\begin{align}\label{eq:piecewisePolyStep}
	[u_1\cdots u_{\ell'}][p_{\lambda^{(2r-1)}\setminus \lambda^{(2r)}}] [q_{\mu^{(2r-1)}\setminus \mu^{(2r)}}] \hat{\widetilde {H}}_\rho^{[i_{2r-1}-i_{2r}]} 
\end{align}
is collected along the way. By the induction hypothesis, and since
$\ell'<\ell$, the quantity~\eqref{eq:piecewisePolyStep} is a
polynomial, whose coefficients are piecewise polynomials in all parameters $\lambda^{(2r-1)}_j, \mu^{(2r-1)}_j, i_{2r-1}, \lambda^{(2r)}_j, \mu^{(2r)}_j, i_{2r}$ involved.

In conclusion, we have proved the following fact: the coefficients of
the quantity~\eqref{eq:fixedWord} are the sums over integer
trajectories of the form~\eqref{eq:tuples}, constrained to live in a
finite-dimensional polytope, of products of quantities of the
form~\eqref{eq:piecewisePolyStep}, which are piecewise polynomials in
the coordinates. Therefore the coefficients of~\eqref{eq:fixedWord} are piecewise polynomials, which is what we wanted to prove.
\end{proof}

      \begin{remark}
	      In the case $b=0$, piecewise polynomiality-type
        results were an important motivation to look for a hidden geometric explanation
        in the spirit of the
        ELSV-formula~\cite{EkedahlLandoShapiroVainshtein2001}. Our
        result gives one more motivation to explore the underlying, yet to be found,
        geometric structure describing the $b$-deformation.
        \end{remark}

\subsection{Dessins d'enfants, the $b$-conjecture, and $\beta$-ensembles}
\label{subsec:betaEnsembles}

The case $k=1$ of our main results, or equivalently the case $G(z)=(1+u_1z)$ of $b$-weighted Hurwitz numbers, corresponds to coverings with three ramification points, or combinatorially, to $1$-constellations. 
In the orientable case, $1$-constellations are called \emph{dessins d'enfants}, \emph{bipartite maps}, or \emph{hypermaps}. See~\cite{LandoZvonkin2004}.

Bipartite maps on non-orientable surfaces have been considered before. They  can be encoded combinatorially by matchings, or equivalently by products in the double coset algebra of the Gelfand pair $(\mathfrak{S}_{2n}, H_n)$, see~\cite{HanlonStanleyStembridge1992, GouldenJackson1996}.
In~\cite{GouldenJackson1996} the following formal power series is introduced:
\begin{align}\label{eq:bFunction}
  B(t;\pp,\qq,\rr) :&=
	\sum_{n \geq 0}t^n\sum_{\la
	\vdash n}\frac{1}{j_\la^{(\alpha)}} J_\la^{(\alpha)}(\pp)J_\la^{(\alpha)}(\qq) J_\la^{(\alpha)}(\rr),
   \end{align}
where $\rr=(r_i)_{i\geq 1}$ is an infinite family of variables. Of course, this function becomes equal to our function $\tau^{(1)}_b(t;\pp,\qq,u_1)$ under the specialization $\rr=\underline{u_1}$.
In the same paper~\cite{GouldenJackson1996}, Goulden and Jackson
state the $b$-conjecture\footnote{In the same paper Goulden and
  Jackson also state the closely related \emph{Matching-Jack
    conjecture}, which deals with non-necessarily connected bipartite
  maps. Our statements for non-necessarily connected
  $1$-constellations are related to this conjecture in the same way as
  our statements for rooted connected $1$-constellations are related
  to the $b$-conjecture. We choose to focus the discussion on the
  latter.
  }, namely that one can write
\begin{align}\label{eq:bConjecture}
	(1+b)\frac{t \partial}{\partial t} B(t;\pp,\qq,\rr) :&=
                                                               \sum_{n
                                                               \geq 1}\sum_{\lambda, \mu, \iota \vdash n} \sum_{(\bM,c)} t^n  p_\lambda q_\mu r_\iota b^{s(\bM,c)},
\end{align}
where the last sum is taken over rooted $1$-constellations (rooted bipartite maps in the language of~\cite{GouldenJackson1996}) on a connected surface (orientable or not), of size $n$ and of full profile $(\lambda, \mu, \iota)$, and where $s(\bM,c)$ is an (unspecified) integer parameter attached to $(\bM,c)$ which is zero if and only if the surface is orientable.
As they show, this statement is true for $b\in\{0,1\}$,  which is
proved using the connection between Schur (or zonal, respectively)
polynomials and representation theory of $\mathfrak{S}_n$ (or of the
Gelfand pair $(\mathfrak{S}_{2n},
H_n)$). See~\cite{Macdonald1995,HanlonStanleyStembridge1992,
  GouldenJackson1996}. For general values of $b$, no suitable connection to
representation theory is known, and this conjecture is still open due
to the lack of tools to attack it. However, it is interesting to
remark that both the $b$-conjecture and our main result (say, in the
formulation of \cref{thm:infinite}) involve \emph{three}
infinite families of parameters, respectively $\{\pp,\qq,\rr\}$ and
$\{\pp,\qq,(u_i)_{i\geq 1}\}$. As we already mentioned in the
introduction, our results are fundamental in a three-step program to attack to the $b$-conjecture. The only missing ingredient in this program is a
certain multiplicativity property of the function $\MultiJack$, that
we are unable to prove in full generality, but that we can
however prove in the special cases
$b=0,1$. In particular this leads us to a new proof of the $b$-conjecture in
these special cases, which relies only on the developments of this paper, and not
on the representation theoretical interpretations described above -- which are available only in
these special cases. We plan to continue our research in this direction in the
future and we hope that further studies on generalized branched
coverings will be crucial in finalizing this 
program to prove Goulden and Jackson's conjectures.

Polynomiality over the \emph{rationals} in the $b$- and Matching-Jack conjectures was proved in \cite{DolegaFeray2016,DolegaFeray2017}, but integrality and positivity of coefficients had been proved until now only in a few very special cases. The
case $k=1$ of our main result coincides with the $b$-conjecture in
the case where $\rr=\underline{u_1}$, if we identify the unknown
parameter $s(\bM,c)$ with our parameter $\nu_\rho(\bM,c)$.  In other
words, the $b$-conjecture is now proved when one keeps \emph{two} full
sets of indeterminate variables ($\pp$ and $\qq$). This is by far the best progress towards it and, in particular,
it covers all the special cases proved in the literature so far, as we now quickly explain. 

The most general case proven so far was the case of the simultaneous specialisations $\rr=\underline{u_1}$ and $q_i=\mathbf{1}_{i=2}$ (thus keeping \emph{one} full set of indeterminate variables). This was done by Lacroix~\cite{LaCroix2009}.
This case corresponds to $1$-constellations in which all vertices of colour $0$ have degree $2$, which are in bijection (by lifting these vertices into single edges)  with general, uncoloured, maps on surfaces (the only remaining vertices have colour~$1$ and can be thought of as uncoloured, they are counted with a modified weight $u_1$; vertices of colour $0$ have become edges, with weight $t$; faces of degree $i$ are counted with a weight $p_i$).
Lacroix used a connection between the function $B$ under this specialization and the $\beta$-ensembles of random matrix theory, together with a connection due to Okounkov~\cite{Okounkov1997} between $\beta$-ensembles and Jack polynomials -- for $1+b = 2/\beta$. The paper~\cite{Okounkov1997} enables one to work with certain \emph{linear} combinations of Jack polynomials, which is what~\cite{LaCroix2009} uses.
In our work, we crucially need to consider \emph{bilinear} sums instead (with Jack polynomials in two set of variables $\pp$ and $\qq$) which makes it inaccessible by this method.
Similarly, the equations of~\cite{AdlervanMoerbeke2001} in the context of $\beta$-ensembles deal with a function of \emph{one} infinite family of variables.
It is reasonable to expect that our work could be related to multi-matrix analogues of the $\beta$-ensembles.

Some special cases of the $b$-conjecture had been established for
bipartite maps in some other very restricted cases, for example for
bipartite maps with a unique face and of genus at most $2$,
see~\cite{Dolega2017a}. See also~\cite{KanunnikovVassilieva2016,KanunnikovPromyslovVassilieva2018} for
other partial results, concerning mostly the coefficients of $B$
itself (as in the Matching-Jack conjecture). All these cases are fully covered by the case $k=1$ of our result by taking specializations or extracting coefficients.

\subsection{$b$-Monotone Hurwitz numbers and $\beta$-HCIZ integral}
\label{subsec:HCIZ}

In the case $b=0$, the particular choice of function $G(z)=Z(z):=\frac{1}{1-\hbar z}$ is known to generate the monotone Hurwitz numbers, see~\cite{GouldenGuayPaquetNovak2014}. That is, $\hat F^{Z}\Big|_{b=0}$ has an explicit interpretation as a generating function of factorizations of a product of two permutations (whose cycles lengths are marked by variables $p_i$ and $q_i$) into a product of transpositions whose maximal transposed elements are weakly increasing.

It is also known~\cite{GouldenGuayPaquetNovak2014} that this same function is a $1/N$-expansion of the Harish-Chandra-Itzykson-Zuber (HCIZ) integral, 
$$I_N (t) = \int _{U(N)} e^{t N \mathop{Tr} (A_N U B_N U^{-1})} dU,$$
where $dU$ is the Haar measure over the unitary group $U(N)$.
Here the variables $p_i$ and $q_j$ are the power-sum symmetric functions in the eigenvalues of the diagonal matrices $A_N$ and $B_N$, and $\hbar$ plays the role of $-\frac{1}{N}$ in a double expansion in $t$ and $\frac{1}{N}$ of $I_N(t)$. See~\cite{GouldenGuayPaquetNovak2014} again for the precise meaning of this statement.

The function $\hat F^{Z}$ provides a natural $b$-deformation of the generating function of monotone Hurwitz numbers. This deformation satisfies the equation: 
\begin{equation}
	\frac{t\partial}{\partial t} \hat F^{Z}  = \Theta_Y \sum_{m \geq 1}t^m \cdot q_m \cdot
	\Big(Y_+  \frac{1}{1-\hbar \Lambda_{Y}}\Big)^m  \frac{y_0}{1+b} \hat F^{Z},
\end{equation}
and our main theorem implies in particular that  $(1+b)\frac{t\partial}{\partial t} \log \hat F^Z$ has coefficients which are polynomials in $b$ with an explicit combinatorial interpretation.

Moreover, deformations $I^{\beta}_N(t)$ of the HCIZ integral obtained by deforming the Laplace-Beltrami operator have been considered before, at least since Brézin and Hikami~\cite{BrezinHikami2003} (there are many other references, see e.g.~\cite{BergereEynard2009}). It would be interesting to study if they admit $1/N$ expansions, and if they are related to the $b$-deformed monotone Hurwitz numbers which we introduce here. In other words, can the following diagram be made commutative?
$$
\begin{array}{rcccl}
	&I_N(t) & \stackrel{\frac{1}{N}-\mbox{expansion, \cite{GouldenGuayPaquetNovak2014}}}{\xrightarrow{\hspace*{4cm}}} & \hat F^{Z}(t,\hbar) \Big|_{b=0}& \\
	\stackrel{\mbox{$\beta$-deformation}}{\mbox{e.g. \cite{BrezinHikami2003}}}	&\xdownarrow{6mm} & &\xdownarrow{6mm}& 	\stackrel{\mbox{$b$-deformation}}{\mbox{(this paper)}} \\
	&I^{\beta}_N(t) & \stackrel{\frac{1}{N}-\mbox{expansion ?} }{\xrightarrow{\hspace*{4cm}}} & \hat F^{Z}(t,\hbar)&
\end{array}
$$
This question is beyond the scope of our paper and is left as an open
problem\footnote{Note added during revision: this questions turned out
  to
  have an affirmative answer as described
  in~\cite{BonzomChapuyDolega2022}, where we were studying properties
  of the $b$-monotone Hurwitz numbers}.

	\medskip
	\noindent {\bf Acknowledgements.} We thank Houcine Ben Dali for pointing a mistake in the interpretation of the decomposition equation in the first version of this paper.

\newpage
\appendix

\newcommand{\dop}[1]{\includegraphics{diag#1.pdf}}
\newcommand{\sdop}[1]{\includegraphics[scale=0.7]{diag#1.pdf}}
\newcommand{\bdop}[1]{\includegraphics[scale=1.3]{diag#1.pdf}}

\section{Relations between operators of finite order}
\label{sec:AppendixComputations}

In this appendix, for completeness, we prove 
\cref{lemma:relations} and
\cref{lemma:relations4}.
All the equalities of operators appearing there can be proved by hand, either by computing commutators or by checking the action on a basis. These computations are lengthy but present no difficulty.
Here we present them in a diagrammatic way, which has the advantage of grouping terms together in a way which makes them easier to check -- however each reader may prefer to rely on their own technique to group terms, depending on their taste. The diagrams are especially useful to record the calculations to prove~\eqref{eq:rela1} and~\eqref{eq:rel2bis}. Other relations of the lemmas are quickly checked by any mean.

\medskip
{\bf Diagrammatic conventions.} Operators are represented by tree-like diagrams, read from top to bottom. Each top or bottom vertex of the diagram represents a variable. Variables from the families $\pp, \yy, \yy', \zz,  \zz'$ are represented by white, black, grey circles  and black, grey squares, respectively ($\circ,\bullet, \bdop{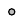},\bdop{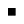},\bdop{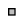}$). Inner nodes of the diagram have the effect of merging or splitting variables in all possible ways, conserving the sum of their indices. For example, the diagrams  \dop{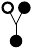} and \dop{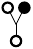}  represent respectively the operators 
$$\sum_{i,j}  y_{i+j} \frac{\partial^2}{\partial y_j \partial p_i}
\mbox{ and }
\sum_{i,j}  p_{i+j} \frac{\partial^2}{\partial y_j \partial p_i}.
$$
It is implicit that all sums are taken over combinatorially meaningful values of the parameters, {\it i.e.} indices of variables $\pp$ are positive while indices of other variables are nonnegative.
A fat edge in the diagram means that the contribution is weighted by the index (say $j$) of the variable appearing at the top of this edge, for example \dop{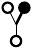} represents  
$\displaystyle\sum_{i,j} j p_{i+j} \frac{\partial^2}{\partial y_j \partial p_i}.$
Similarly, we use a fat dotted edge to weight the contribution by a factor of $j(j-1)$.

{\bf Composition.}
Composition of operators corresponds to concatenation of diagrams from top to bottom, respecting the type of variable at the gluing points.
For example the composition \dop{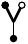}$\circ$\dop{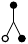} is equal to \dop{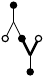}+\dop{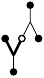}+ \dop{4}~\dop{5}+\dop{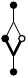}. Note that since we work on monomials which are at most linear in $\yy$, the second and third term are the null operator on the spaces we consider.

{\bf Example.}
The  operator \dop{6} has two top vertices of types $\yy, \pp$ and two bottom vertices of types $\yy,\pp$, it 
is equal to\footnote{Indeed  for a given choice of $i,k,i',k'$ such that $i+k=i'+k'$, if one assigns variables $p_i,y_k$ and $p_{i'},y_{k'}$ to top and bottom vertices according to their type, 
it is possible to distribute indices at each node if and only if $i'<k$ (as one sees by considering the topmost inner node) and in this case there is a unique way to do it. Moreover the two fat edges have respective weights $k-i'$ and $i$, from left to right.}
$\displaystyle
\sum_{i+k=i'+k'} \bI{i'<k} i (k-i')  y_{k'} p_{i'} \frac{\partial^2}{\partial y_k \partial p_i} .
$

\begin{proof}[Proof of~\eqref{eq:rela1}]
	We compute the commutator $[D_\alpha+D_\alpha',\Lambda_Y]$ by collecting its coefficients as a polynomial in $\alpha$ and $b$, viewed as independent variables. We have
$$
	D_\alpha+D_\alpha'=\sdop{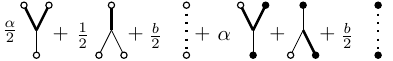} \ , \ 
	\Lambda_Y =\sdop{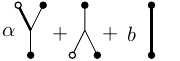}. 
$$
We need to evaluate the coefficients of $\alpha^2, \alpha b,  b^2, \alpha, b$, and $1$.
We start with the coefficient of $\alpha$, which is equal to\footnote{In equations below we expand commutators from left to right, e.g. $[a+b,c+d]=[a,c]+[a,d]+[b,c]+[b,d]$ in this order. We make an exception to this rule when, as in~\eqref{eq:commalphab} below, it is convenient to group terms of the R.H.S. in blocks according to the nature of incoming and outgoing variables -- in that case, the left-to-right order of expansion is preserved inside each block. Also, note that unconnected diagrams such as \sdop{4}~\sdop{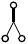} do not appear in commutators, since the two ways to obtain them cancel.}
	\begin{align}
		\sdop{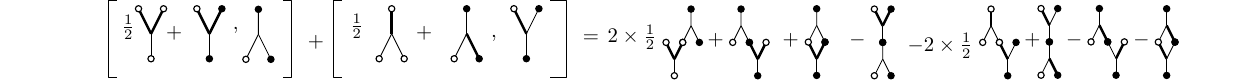}. \label{eq:commIalpha1}
	\end{align}
The third and last terms cancel, and by checking the types of top of bottom vertices, we see that what remains is a linear combination of operators of the form $\bI{i'+k'=i+k} p_{i'}y_{k'}\frac{\partial^2}{\partial y_k \partial p_i}$, 
	with coefficients that can be read from the diagrams and are equal to\footnote{In the following equality we assume the equation $i'+k'=i+k$ which is implicit from context. We will do similar assumptions in other computations without recalling it.} 
	\begin{align*}
		&i(k\sminus{}k')\bI{i'>i}
\splus{}i(k\sminus{}i')\bI{i'<k}
\sminus{}ik
\sminus{}i(i\sminus{}i')\bI{i'<i}
\splus{}ik'
		\sminus{}i(k\sminus{}i')\bI{i'<k}\\=&
		i(k\sminus{}k'\splus{}i\sminus{}i')\bI{i'>i}
		\splus{}i(k'\sminus{}k\sminus{}i\splus{}i')
			=0,
	\end{align*}
	where we used $(i\sminus{}i')\bI{i'<i}=(i\sminus{}i')(1\sminus{}\bI{i'>i})$.
Similarly the coefficients $\alpha^2$, $\alpha b$, $b$, and $1$, are respectively equal to 
\begin{align*}
	\sdop{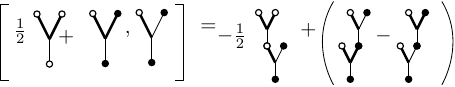}, \sdop{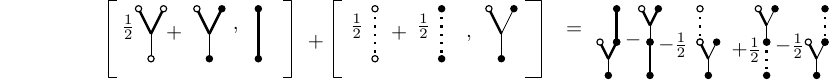},\\
	\sdop{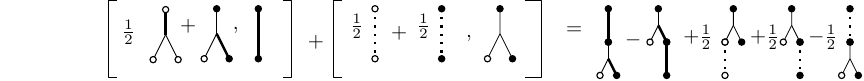}, \sdop{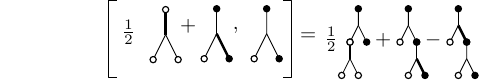}.
\end{align*}
These are respectively linear combinations of operators of the form
$$
	\frac{ y_{i+j+k}\partial^3}{\partial y_k \partial p_i \partial p_j} \ , \
	\frac{ y_{i+k}\partial^2}{\partial y_k \partial p_i} \ , \
	\frac{ y_{k}p_i\partial}{\partial y_{i+k}} \ , \
	\frac{ y_k p_i p_j\partial}{\partial y_{i+j+k}}.$$
	The coefficients can be read on the diagrams and are equal\footnote{Note that for the contribution of the second diagram in the coefficient of $\alpha^2$, we have symmetrized in $(i,j)$, namely we substituted $ij (i+k) \rightarrow ij\frac{ (i+k) + (j+k)}{2}$, since we know that the sum we take over $i$ and $j$ is symmetric. We do the same for the very last diagram appearing in the coefficient of $1$.} to, respectively
	\begin{align*}
		\sminus{}ij\tfrac{i\splus{}j}{2} \splus{} ij\tfrac{ (i\splus{}k) \splus{} (j\splus{}k)}{2}\sminus{}ijk = 0 
		\ &, \  
		k^2 i \sminus{}ik(i\splus{}k)\sminus{}\tfrac{i(i\sminus{}1)i}{2}\splus{}\tfrac{i (i\splus{}k)(i\splus{}k\sminus{}1)}{2}\sminus{}\tfrac{ik(k\sminus{}1)}{2}=0,
		\\
		k(k\splus{}i)\sminus{}k^2\splus{}\tfrac{i(i\sminus{}1)}{2} \splus{}\tfrac{k(k\sminus{}1)}{2}\sminus{}
		\tfrac{(k\splus{}i)(k\splus{}i\sminus{}1)}{2}=0
		\ &, \
		\tfrac{1}{2}(i\splus{}j) \splus{} k \sminus{}  \tfrac{(i\splus{}k)\splus{}(j\splus{}k)}{2}=0.
	\end{align*}
	This concludes the proof that $[D_\alpha+D_\alpha',\Lambda_Y]=0$. 
\end{proof}

Before completing the proof of \cref{lemma:relations}, we now prove the main commutation relation of \cref{lemma:relations4}.
\begin{proof}[Proof of~\eqref{eq:rel2bis}]
Assuming other relations it is enough to prove the first equality, namely
	$[\La_{\tilde{Z}}+\Delta,\La_{\tilde{Y}}]=0$.
	We have
	\begin{align*}
		\Lambda_{\tilde{Z}}+\Delta&=\sdop{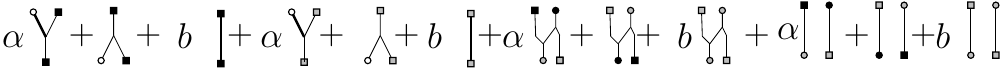}, \\
		\Lambda_{\tilde{Y}}&=\sdop{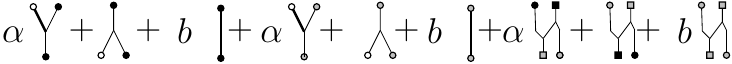}.
	\end{align*}
	We compute the commutator $[\La_{\tilde{Z}}+\Delta,\La_{\tilde{Y}}]$ as in the proof of~\eqref{eq:rel1}, by collecting its coefficients as a polynomial in $\alpha$ and $b$. Thus we need to evaluate the coefficients of $\alpha^2, \alpha b,  b^2, \alpha, b$, and $1$.
	The coefficient of $\alpha^2$ is equal to
	\begin{align}
		\sdop{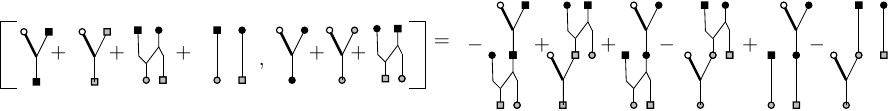}, \label{eq:commalpha2}
	\end{align}
	which is a linear combination of operators of the form $\bI{j'+k'=i+j+k}\frac{y_{j'}z_{k'}\partial^3}{\partial p_{i} \partial y_j \partial z_k}$, with coefficients that can be read from the diagrams\footnote{For example, the contribution of the leftmost diagram is computed as follows. Call $p_i,y_j,z_k$ and $y_{j'} z_{k'}$ the bottom and top variables, respectively. The fat edge gives a factor of $i$. Moreoever, for the indices of variables to be distributed at the inner nodes coherently, one sees that the index $i+k$ of the inner node inherited from the concatenation of the two smaller diagrams, has to be larger than the index $j'$ of the bottom $\yy'$-node. It is easy to see that this is the only constraint, hence a total contribution of $i\bI{j'<i+k}$ for that diagram. Other diagrams are treated similarly.}:
	\begin{align}\label{eq:luckyAgain}
i (-\bI{j'<i+k}+\bI{j'<k}+\bI{j'> k}-\bI{j'> i+k}+\bI{j'=k}-\bI{j'=i+k})
	= i(1-1) = 0.
	\end{align}
	Similarly, the coefficient of $\alpha b$ is equal to
	\begin{align}
		\sdop{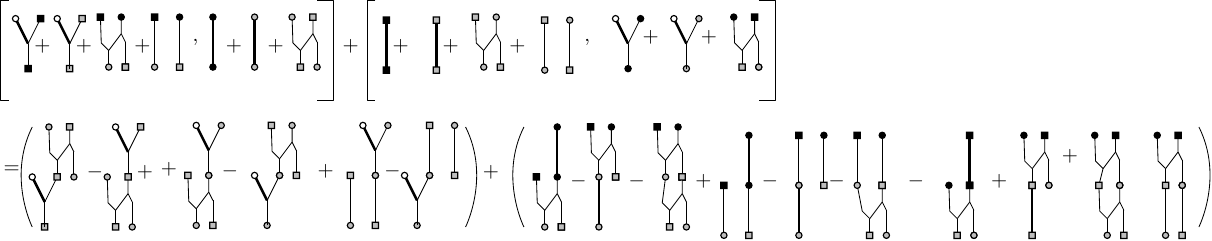},\!\! \label{eq:commalphab}
	\end{align}
	Note that the first term of~\eqref{eq:commalphab} is the same as~\eqref{eq:commalpha2} with permuted variables, so it is equal to zero. The second term of~\eqref{eq:commalphab} is a linear combination of operators $\bI{j'+k'=j+k}\frac{y_{j'}z_{k'}\partial^2}{\partial y_{j} \partial z_k}$, with coefficient%
	\footnote{In the second term of~\eqref{eq:commalphab}, the coefficient of the third diagram is computed as follows. Call $y_j,z_k$ and $y_{j'}, z_{k'}$ the top and bottom variables, respectively. Then the index (say $u$) of the ``square'' inner node inherited from the concatenation of the two smaller diagrams has to satisfy $u<j$ (from the top part of the diagram) and $u>j'$ (for the bottom part). This requires that $j>j'$, and in this case there are $j\sminus{}j'\sminus{}1$ possible choices for $u$. Similar constraints hold for the ``circle'' inner node but thay are equivalent to the previous ones provided that $j'+k'=j+k$. Therefore the contribution of this diagram is $(j\sminus{}j'\sminus{}1)\bI{j>j'}$. Other cases of the same sort appear in the computations and are treated similarly.}
	\begin{align}
		&	j\bI{j'>k}
	\sminus{}j'\bI{j'>k}
		\sminus{}(j\sminus{}j'\sminus{}1)\bI{j>j'}
	\splus{}j\bI{j'=k}
	\sminus{}j'\bI{j'=k} 
	\sminus{}\bI{j'<j}
	\sminus{}k\bI{j'<k}
	\splus{}k'\bI{j'<k}\label{eq:lucky}
		\\&\hspace{10cm} \nonumber
		\splus{}(k\sminus{}k'\sminus{}1)\bI{j'>j}
		\splus{}\bI{j'>j}   \\
		&=
		(j\sminus{}j')\bI{j'\geq k}
		\splus{}(k'\sminus{}k)(1\sminus{}\bI{j'\geq k})
		\splus{}(j'\sminus{}j)\bI{j>j'}
		\splus{}(k\sminus{}k')(1\sminus{}\bI{j\geq j'}) \nonumber
		\\
		&=0\cdot \bI{j'\geq k}
		\splus{}0\cdot \bI{j\geq j'}
		\splus{}k'\sminus{}k\splus{}k\sminus{}k' =0, \nonumber
	\end{align}
	Similarly, the coefficient of $b^2$ is equal to
	\begin{align}
		\sdop{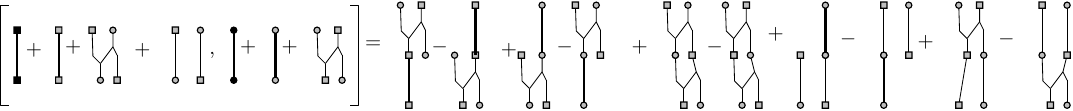}, \label{eq:commb2}	
	\end{align}
	which is a linear combination of operators $\bI{j'\splus{}k'=j\splus{}k}\frac{y'_{j'}z'_{k'}\partial^2}{\partial y'_{j} \partial z'_k}$, with coefficient
	\begin{align*}
		&k'\bI{j'<k}
	\sminus{}k\bI{j'<k}
	\splus{}j\bI{j'>k}
	\sminus{}j'\bI{j'>k}
		\splus{}(k'\sminus{}k\sminus{}1)\bI{j'<j}
		\sminus{}(j'\sminus{}j\sminus{}1)\bI{j'>j}
	\splus{}j\bI{j'=k}
	\sminus{}j'\bI{j'=k}
		\\&\hspace{12cm} \nonumber
	\splus{}\bI{j'<k}
	\sminus{}\bI{j'<j},
	\end{align*}
	which is zero by the same computation we performed on~\eqref{eq:lucky} (up to exchanging prime and non-primes and reversing inequalities).

The coefficient of $\alpha$ is 
	\begin{align}
		\sdop{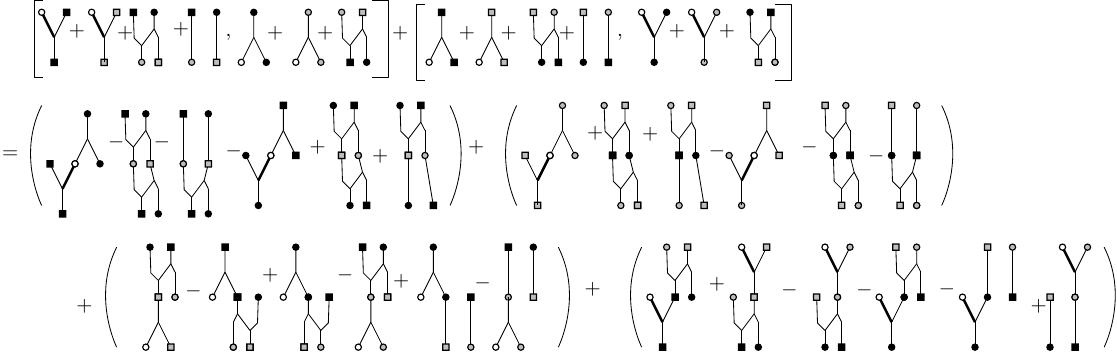},  \label{eq:commalpha1}
	\end{align}
which is a linear combination of operators of the form 
$\bI{j'\splus{}k'=j\splus{}k}\frac{y_{j'}z_{k'}\partial^2}{\partial y_{j} \partial z_k}$,
$\bI{j'\splus{}k'=j\splus{}k}\frac{y'_{j'}z'_{k'}\partial^2}{\partial y'_{j} \partial z'_k}$,
$\bI{i'\splus{}j'\splus{}k'=j\splus{}k}\frac{p_iy'_{j'}z'_{k'}\partial^2}{\partial y_{j} \partial z_k}$,
$\bI{j'\splus{}k'=i\splus{}j\splus{}k}\frac{y_{j'}z_{k'}\partial^3}{\partial p_i\partial y'_{j} \partial z'_k}$.
The first and third  have respective coefficients
$$
(j\sminus{}j')\bI{j'<j}\sminus{}
(j\sminus{}j'\sminus{}1)\bI{j'<j}
\sminus{}\bI{j<j}
\sminus{}(k\sminus{}k')\bI{k'<k}
\splus{}(k\sminus{}k'\sminus{}1)\bI{k'<k}
\splus{}\bI{k'<k}=0,
$$
$$
\bI{j'<k}
-\bI{j'<k-i'}
+\bI{j'>k}
-\bI{j'>k-i'}
+\bI{j'=k}
-\bI{j'=k-i'}=1-1=0,
$$
while the second is the same as the first up to permutation of colours and the fourth is analogous to the third.

The coefficient of $b$ is
	\begin{align}
		\sdop{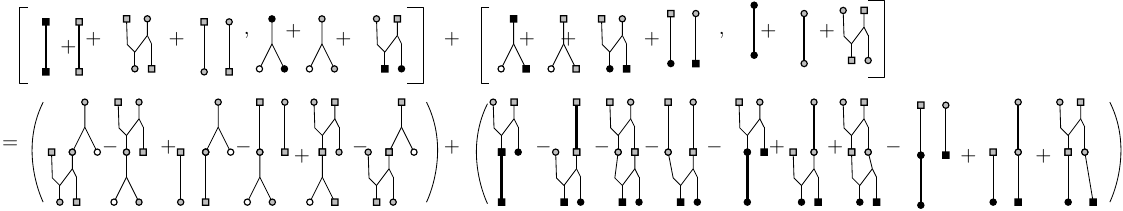}. \label{eq:commb1}	
	\end{align}
The first term is a linear combination of operators $\bI{i'\splus{}j'\splus{}k'=j\splus{}k}\frac{p_{i'}y'_{j'}z'_{k'}\partial^2}{\partial y'_{j} \partial z'_k}$, with coefficient
$$
\bI{j'>k}-\bI{j'>k-i'}+\bI{j'=k}-\bI{j'=k-i'}+\bI{j'<k}-\bI{j'<k-i'}=1-1=0.
$$
(in fact this computation is the same as~\eqref{eq:luckyAgain} with diagrams upside-down), while the second term is again similar to the second term of~\eqref{eq:commalphab} and is zero by the same computation we performed on~\eqref{eq:lucky}.
Finally, the coefficient of $1$ is
	\begin{align}
		\sdop{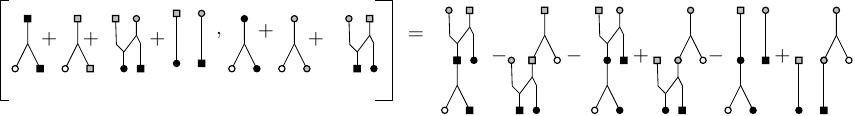}, \label{eq:comm1}	
	\end{align}
	which is the same as the first term of \eqref{eq:commb1} up to an exchange of variables, so is equal to zero too. This concludes the proof that $[\La_{\tilde{Z}}+\Delta,\La_{\tilde{Y}}]=0$.
\end{proof}

The remaining computations are much shorter and can easily be done without any diagrams. We include them here for completeness. We introduce further notations: increment of the index of the variable is represented by an arrow, for example $Y_+=\dop{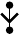}$.

\begin{proof}[Proof of~\eqref{eq:rela2}]
	We directly compute the commutator $[D_\alpha\splus{} D_\alpha',Y_+]$
\begin{align*}
	&\big[\dop{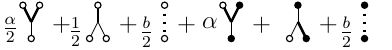}, \dop{11}\big]=\dop{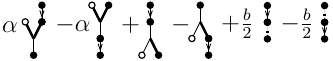}\\
	&=\alpha  \sum_{i,k} [i(k\splus1)\sminus ik  ] \frac{y_{i+k+1}\partial^2}{\partial p_i y_k}
	 + \sum_{i,k} [ (k\splus1)\sminus k ] \frac{p_i y_{k+1}\partial}{\partial y_{i+k}}
	 + \frac{b}{2}\sum_{k} [(k\splus1)k \sminus k (k\sminus1)  ] \frac{y_{k+1}\partial}{\partial y_{k}}\\
	 &=Y_+ \Lambda_Y. \qedhere
\end{align*}
\end{proof}

\begin{proof}[Proof of~\eqref{eq:rela3}]
We want to prove that $[D_\alpha, \Theta_Y]=\Theta_Y D'_\alpha$. We have
	\begin{flalign*}
		[D_\alpha, \Theta_Y] = [\dop{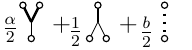},\dop{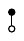}]=\dop{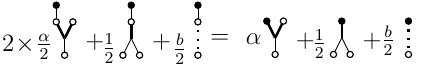},\\
		\ \Theta_Y D'_\alpha = \dop{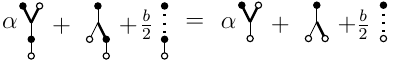}.
	\end{flalign*}
	To conclude it is enough to notice that \dop{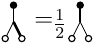}, {\it i.e.} $\sum_{i,j} i \frac{p_{i+j}\partial^2}{\partial p_i \partial p_j} = \frac{1}{2}\sum_{i,j} (i+j) \frac{p_{i+j}\partial^2}{\partial p_i \partial p_j}$.
\end{proof}

\begin{proof}[Proof of~\eqref{eq:rel1}]
	We compute the operators $\La_{\tilde{Z}}\Delta$ and $\Delta \La_{\tilde{Y}}$, which are equal respectively to
	\begin{align*}
		\sdop{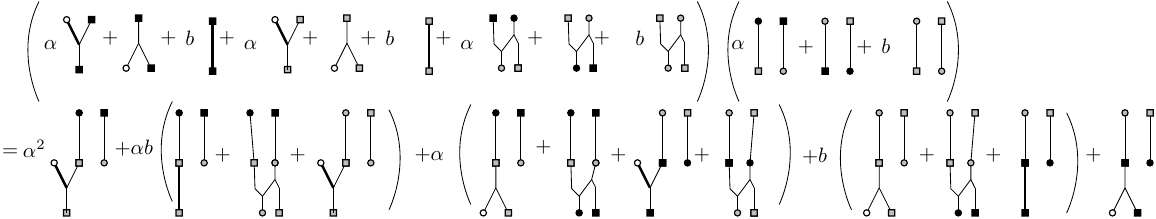}
	\end{align*}
	and
	\begin{align*}
		\sdop{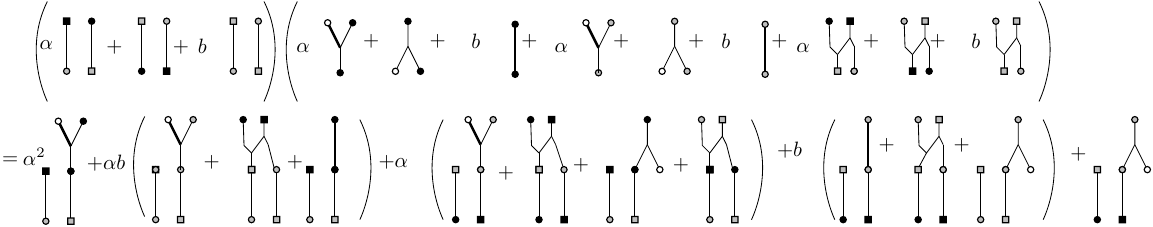}.
	\end{align*}
We conclude by noticing that the two quantities differ only by the colour of inner vertices and reordering of terms, more precisely both are equal to	
	\begin{align*}
		\sdop{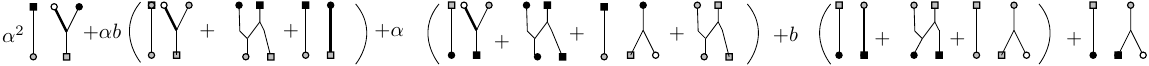}. 
	\end{align*}
\end{proof}

\begin{proof}[Proof of~\eqref{eq:rel3} and~\eqref{eq:rel3bis}]
	It is enough to prove the first one, namely $[\La_{\ZZ},\YY_+] = \YY_+ \Delta$. The commutator $[\La_{\ZZ},\YY_+]$ is equal to
		\begin{align*}
			&	\sdop{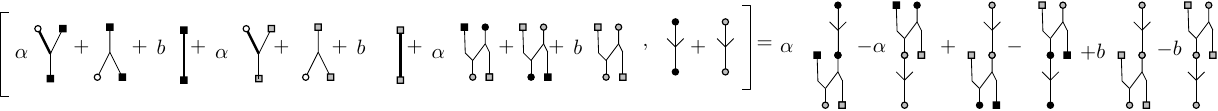}\\
			&=\sum_{i'\splus{}k'=i\splus{}k\splus{}1}
			\Big(\alpha  (\bI{k'<i\splus{}1}\sminus{}\bI{k'<i})  \frac{y'_{i'}z'_{k'}\partial^2}{\partial y_i z_k}
	 \splus{}  (\bI{k'<i\splus{}1}\sminus{}\bI{k'<i})  \frac{y_{i'}z_{k'}\partial^2}{\partial y'_i z'_k}
		\splus{}  b  (\bI{k'<i\splus{}1}\sminus{}\bI{k'<i})  \frac{y'_{i'}z'_{k'}\partial^2}{\partial y'_i z'_k}
			\Big)	\\
	 &
			=\sum_{i,k}
			\alpha  \frac{y'_{k+1}z'_{i}\partial^2}{\partial y_i z_k}
	 \splus{}    \frac{y_{k+1}z_{i}\partial^2}{\partial y'_i z'_k}
		\splus{}  b    \frac{y'_{k+1}z'_{i}\partial^2}{\partial y'_i z'_k}
			=\YY_+ \Delta,
	\end{align*}
	where for the last equality we used that $\bI{k'<i\splus{}1}\sminus{}\bI{k'<i} =\bI{k'=i}$ and that $k'=i$ and  
$i'\splus{}k'=i\splus{}k\splus{}1$ imply that $i'=k+1$.
\end{proof}

\begin{proof}[Proof of~\eqref{eq:rel4}] It is enough to prove the first equality $\Theta_{\tilde{Z}}\La_{\tilde{Y}} = \La_{Y}\Theta_{\tilde{Z}}$, {\it i.e.} $\Theta_{\tilde{Z}}(\La_{\tilde{Y}}-\La_Y) = [\La_{Y},\Theta_{\tilde{Z}}]$.
Now $[\La_{Y},\Theta_{\tilde{Z}}]$ and $\Theta_{\tilde{Z}}(\La_{\tilde{Y}}-\La_Y) $ are respectively equal to	
	\begin{align*}
		\sdop{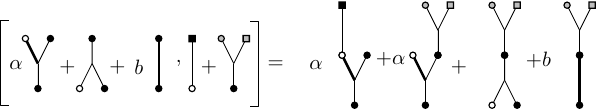}
	\end{align*}
	and
	\begin{align*}
		\sdop{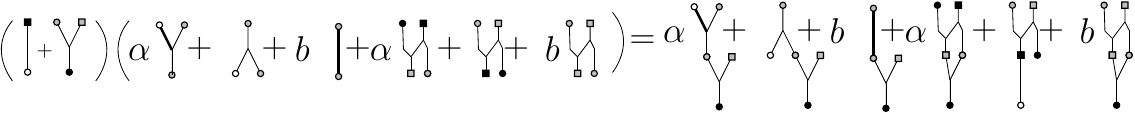}.
	\end{align*}
	The difference of these two quantities is a linear combinations of operators 
	$\frac{y_{i+j+k}\partial^3}{\partial p_i \partial y'_j \partial z'_k}$,
	$\bI{i'+j'=j+k}\frac{p_{i'}y_{j'}\partial^2}{\partial y'_j \partial z'_k}$,
	$\frac{y_{i+j}\partial^2}{\partial y'_i \partial z'_j }$,
	$\frac{y_{i+j}\partial^2}{\partial y_i \partial z_j }$.
The coefficients of these operators in this difference are respectively equal to
	\begin{align*}
		\alpha i  - \alpha i =0 \ , \  
		1 - \bI{i'\leq j} -\bI{i'>j}=0 \ ,\  
		b(i+j) -bi -bj =0 \ ,\  
		\alpha j - \alpha j=0.
	\end{align*}

\end{proof}

\begin{proof}[Proof of~\eqref{eq:rel5}]
	It is enough to prove the first equality
$\Theta_{\tilde{Z}}\tilde{Y}_+ = Y_+\Theta_{\tilde{Z}}$, which is straightforward
	since both are equal to
	$	\sum_{i,j}p_i y_{j+1}\frac{\partial^2}{\partial
          y_j\partial z_i} +
        \sum_{i,j}y_{i+j+1}\frac{\partial^2}{\partial y'_i z'_j}$ with
        the convention that $p_0=1$. To be consistent with the rest we still provide the diagrammatic interpretation:
	\begin{align*}
		\sdop{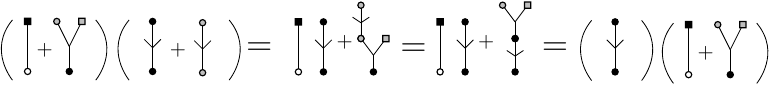}.
	\end{align*}
\end{proof}

\bibliographystyle{amsalpha}

\bibliography{biblio2015}

\end{document}